\documentclass[a4paper, 11pt, pdfLaTex]{amsart}

\usepackage{amsmath}
\usepackage{amssymb}
\usepackage{amsthm}
\usepackage[abbrev,alphabetic]{amsrefs}
\usepackage{amscd}
\usepackage{graphicx}
\usepackage{graphics}
\usepackage{ulem}
\usepackage{url}
\usepackage{multirow}
\usepackage{rotating}

\newcommand{\doi}[1]{\url{https://doi.org/#1}}

\usepackage[top=30truemm,bottom=25truemm,left=35truemm,right=35truemm]{geometry}

\usepackage[all,cmtip]{xy}

\usepackage{color}
\usepackage{xypic}

\newcommand\myurl[1]{\url{#1}}

\BibSpec{webpage}{%
  +{}{\PrintAuthors} {author}
  +{,}{ \textit} {title}
  +{}{ \parenthesize} {date}
  +{,}{ \myurl} {myurl}
  +{,}{ } {note}
  +{.}{ } {transition}
}

\usepackage{listings}
\title[Ehrhart theory on periodic graphs II]
{Ehrhart theory on periodic graphs II: Stratified Ehrhart ring theory}

\author[T.\ Inoue]{Takuya Inoue}
\address{Graduate School of Mathematical Sciences, 
the University of Tokyo, 3-8-1 Komaba, Meguro-ku, Tokyo 153-8914, Japan.}
\email{inoue@ms.u-tokyo.ac.jp}

\author{Yusuke Nakamura}
\address{Graduate School of Mathematics, Nagoya University, Furo-cho, Chikusa-ku, Nagoya, 464-8602, Japan.}
\email{y.nakamura@math.nagoya-u.ac.jp}
\urladdr{https://sites.google.com/site/ynakamuraagmath/}

\newtheorem{thm}{Theorem}[section]
\newtheorem{lem}[thm]{Lemma}

\newtheorem{cor}[thm]{Corollary}
\newtheorem{prop}[thm]{Proposition}

\newtheorem{quest}[thm]{Question}

\newtheorem{fact}[thm]{Fact}

\theoremstyle{definition}
\newtheorem{defi}[thm]{Definition}

\theoremstyle{remark}
\newtheorem{rmk}[thm]{Remark}
\newtheorem{ex}[thm]{Example}
\newtheorem*{ackn}{Acknowledgements}

\begin{document}


\begin{abstract}
We investigate the ``stratified Ehrhart ring theory'' for periodic graphs, 
which gives an algorithm for determining the growth sequences of periodic graphs. 
The growth sequence $(s_{\Gamma, x_0, i})_{i \ge 0}$ is defined for a graph $\Gamma$ and its fixed vertex $x_0$, where 
$s_{\Gamma, x_0, i}$ is defined as the number of vertices of $\Gamma$ at distance $i$ from $x_0$. 
Although the sequences $(s_{\Gamma, x_0, i})_{i \ge 0}$ for periodic graphs are known to be of quasi-polynomial type, 
their determination had not been established, even in dimension two. 
Our theory and algorithm can be applied to arbitrary periodic graphs of any dimension.
As an application of the algorithm, we determine the growth sequences in several new examples. 
\end{abstract}

\maketitle

\tableofcontents

\section{Introduction}

An \textit{$n$-dimensional periodic graph} $(\Gamma, L)$ is a pair of a directed graph $\Gamma$ (that may have loops and multiple edges) and a free abelian group $L$ of rank $n$ such that $L$ freely acts on $\Gamma$ and its quotient graph $\Gamma/L$ is finite (see Definition \ref{defi:pg}). 
For a vertex $x_0$ of $\Gamma$, the \textit{growth sequence} $(s_{\Gamma, x_0, i})_{i \ge 0}$ 
(resp.\ \textit{cumulative growth sequence} $(b_{\Gamma, x_0, i})_{i \ge 0}$) is defined as the number of vertices of $\Gamma$ whose distance from $x_0$ is $i$ (resp.\ at most $i$). 
Periodic graphs naturally appear in crystallography, also appear as periodic tilings in combinatorics, and as Cayley graphs of virtually abelian groups in geometric group theory (see \cite{GS19}*{Section 13}). 
Furthermore, it is shown in \cite{IN} that the theory of growth sequence of periodic graphs potentially includes the Ehrhart theory of rational polytopes. 

In \cite{GKBS96}, Grosse-Kunstleve, Brunner and Sloane conjectured 
that the growth sequences of periodic graphs are of quasi-polynomial type, i.e., 
there exist an integer $M$ and a quasi-polynomial $f_s: \mathbb{Z} \to \mathbb{Z}$ such that 
$s_{\Gamma, x_0, i} = f_s(i)$ holds for any $i \ge M$ (see Definition \ref{defi:qp}). 
In \cite{NSMN21}, the second author, Sakamoto, Mase, and Nakagawa prove that this conjecture is true for any periodic graphs (Theorem \ref{thm:NSMN}). 
Although it was proved to be of quasi-polynomial type, determining the explicit formulae of growth sequences is still difficult. 
Thus, the following natural question arises. 

\begin{quest}[{cf.\ \cite{IN}*{Question 1.1}}]\label{quest:algorithm}
Find an effective algorithm to determine the explicit formulae of the growth sequences.
\end{quest}

So far, various computational methods have been established for several special classes of periodic graphs. 
In \cite{CS97}, Conway and Sloane study the growth sequences of the contact graphs of some lattices from the viewpoint of the Ehrhart theory, and they give explicit formulae for the root lattices $A_d$. 
In \cite{BHV99}, Bacher, de la Harpe, and Venkov give the proof for the conjectural formulae for the root lattices $B_d$, $C_d$ and $D_d$ (cf.\ \cite{ABHPS}). 
We note that the periodic graphs $(\Gamma, L)$ obtained from the lattices satisfy $\# (V_{\Gamma} /L) = 1$, where $V_{\Gamma}$ denotes the set of vertices of $\Gamma$. 
When $\# (V_{\Gamma} /L) = 1$, $V_{\Gamma}$ has a monoid structure, and the growth sequence can be directly studied by the Hilbert series of the corresponding graded monoid. 
However, when $\# (V_{\Gamma} /L) > 1$, it is more difficult to study the growth sequence. 
In \cite{GS19}, Goodman-Strauss and Sloane proposed ``the coloring book approach" 
and obtained the growth sequence for some periodic tilings. 
In \cites{SM19, SM20a}, Shutov and Maleev obtained the growth sequences for 
tilings satisfying certain conditions that contain the 20 2-uniform tilings. 
In \cite{IN}, the authors introduce a class of periodic graphs called ``well-arranged'', and they give an algorithm to calculate their growth sequences. However, as far as we know, no algorithm has been proposed that can be applied to any periodic graph, even in dimension two. 

The purpose of this paper is to give an algorithm for determining 
the growth sequences that can be applied to all periodic graphs in all dimensions. 
In Subsection \ref{subsection:inv}, we define invariants $\beta \in \mathbb{R}_{\ge 0}$ and 
$\operatorname{cpx}_{\Gamma} \in \mathbb{Z}_{>0}$ from combinatorial information of $\Gamma$. 
In Corollary \ref{cor:main}, we prove that the growth sequence $(s_{\Gamma, x_0, i})_{i \ge 0}$ is a quasi-polynomial on 
$i > \beta$ and $\operatorname{cpx}_{\Gamma}$ is its quasi-period.  
More precisely, we prove that its generating function $G_{\Gamma, x_0}(t)$ is given by 
\[
G_{\Gamma, x_0}(t) := \sum _{i \ge 0} s_{\Gamma, x_0, i} t^i = \frac{Q(t)}{(1-t^{\operatorname{cpx}_{\Gamma}}) ^n}
\]
with some polynomial $Q(t)$ of $\deg Q \le \beta + n \cdot \operatorname{cpx}_{\Gamma}$. 
On the other hand, with the help of a computer program (breadth-first search algorithm), 
we can compute the first few terms of $(s_{\Gamma, x_0, i})_{i \ge 0}$. 
After we compute the first $\lfloor \beta \rfloor + n \cdot \operatorname{cpx}_{\Gamma} + 1$ terms, 
we can determine the generating function as follows: For $\gamma := \lfloor \beta \rfloor + n \cdot \operatorname{cpx}_{\Gamma}$, 
we have 
\[
G_{\Gamma, x_0}(t) = \frac{\bigl( \text{The terms of $(1-t^{\operatorname{cpx}_{\Gamma}}) ^n \sum _{i = 0} ^{\gamma} s_{\Gamma, x_0, i} t^i$ of degree $\gamma$ or less}\bigr)}{(1-t^{\operatorname{cpx}_{\Gamma}}) ^n}. 
\]

Corollary \ref{cor:main} follows from Theorem \ref{thm:fingen}, which we call the ``stratified Ehrhart ring theory''.
Theorem \ref{thm:fingen} gives an algebraic meaning to the growth sequence, and by using it, 
Corollary \ref{cor:main} can be proved by a standard technique in algebraic combinatorics. 
In what follows, we will outline the statement of Theorem \ref{thm:fingen}. 

First, we briefly review the Ehrhart ring of a rational polytope and its structure.
Let $Q \subset \mathbb{R}^n$ be a rational polytope. 
Then, the Ehrhart ring of a rational polytope $Q \subset \mathbb{R}^n$ is defined as the group ring $k[A]$ corresponding to the (commutative) monoid
\[
A := \left \{ (d, y) \in \mathbb{Z}_{\ge 0} \times \mathbb{Z}^n \ \middle | \ y \in dQ\right \}. 
\]
This $A$ has a graded monoid structure with respect to the degree function $\operatorname{deg}: \mathbb{Z}_{\ge 0} \times \mathbb{Z}^n \to \mathbb{Z}_{\ge 0}: (d,y) \mapsto d$. 
When $0 \in \operatorname{int}(Q)$, the monoid $A$ has a nice algebraic structure as follows: 
 
\begin{fact}[{cf.\ \cite{BR}*{Section 3.2}}]\label{fact}
Suppose $0 \in \operatorname{int}(Q)$. 
For each $\sigma \in \operatorname{Facet}(Q)$, we take a triangulation $T_{\sigma}$ of $\sigma$ satisfying 
$V(\Delta) \subset V(\sigma)$ for any $\Delta \in T_{\sigma}$ (see Section \ref{section:notation} for the notation of triangulations). 
For $v \in V(Q)$, let $a_v$ be the minimum positive integer satisfying $a_v v \in \mathbb{Z}^n$. 
For $\sigma \in \operatorname{Facet}(Q)$, $\Delta \in T_{\sigma}$ and $\Delta' \in \operatorname{Face}(\Delta)$, we define 
\[
L_{\Delta'} := \mathbb{R}_{\ge 0} \Delta' \subset \mathbb{R}^n, \qquad 
A \left( L_{\Delta'} \right) := A \cap \left( \mathbb{Z}_{\ge 0} \times L_{\Delta'} \right). 
\]
We also define a free submonoid $M_{\sigma , \Delta'} \subset \mathbb{Z}_{\ge 0} \times \mathbb{Z}^n$ by 
\[
M_{\sigma , \Delta'} := 
\mathbb{Z}_{\ge 0}(1,0) + 
\sum _{v \in V(\Delta')} \mathbb{Z}_{\ge 0} (a_v, a_v v). 
\]
Then, for $\sigma \in \operatorname{Facet}(Q)$, $\Delta \in T_{\sigma}$ and $\Delta' \in \operatorname{Face}(\Delta)$, it follows that 
\begin{itemize}
\item
each $A \left( L_{\Delta'} \right)$ is a free $M_{\sigma , \Delta'}$-module. 
Furthermore, $A \left( L_{\Delta'} \right)$ is freely generated by finitely many elements whose degrees are less than $\sum _{v \in V(\Delta')} a_v$. 
\end{itemize}
\end{fact}
\noindent
By using this structure (and the inclusion-exclusion principle), we can prove the counting function 
\[
\mathbb{Z}_{\ge 0} \to \mathbb{Z}_{\ge 0}; \ d \mapsto 
\# \left\{ y \in \mathbb{Z}^n \ \middle | \ y \in dQ  \right\}
= \# \left\{ y \in \mathbb{Z}^n \ \middle | \ (d,y) \in A \right\}
\]
is a quasi-polynomial. 

Next, we shall explain the ``stratified Ehrhart ring theory'' for periodic graphs, which is an analogy of Fact \ref{fact}. 
Let $(\Gamma, L)$ be an $n$-dimensional strongly connected periodic graph, and let $x_0$ be a vertex of $\Gamma$. 
Let $V_{\Gamma}$ denote the set of vertices of $\Gamma$. 
We define 
\[
B := \left \{ (d,y) \in \mathbb{Z}_{\ge 0} \times V_{\Gamma}  \ \middle | \ d_{\Gamma}(x_0, y) \le d \right\}. 
\]
Unlike the case of the Ehrhart ring, $B$ itself does not have a monoid structure when $\# (V_{\Gamma} /L) > 1$, and the situation is more complicated. 
In the statement of Theorem \ref{thm:fingen}, the \textit{growth polytope} $P_{\Gamma}$ plays an important role. 
The growth polytope $P_{\Gamma} \subset L_{\mathbb{R}} := L \otimes _{\mathbb{Z}} \mathbb{R}$ is a rational polytope and is canonically defined from the periodic graph $\Gamma$ (Definition \ref{defi:P}). 
Furthermore, we fix a \textit{periodic realization} $\Phi : V_{\Gamma} \to L_{\mathbb{R}}$, that is a map satisfying $\Phi (u + y) = u + \Phi(y)$ for any $y \in V_{\Gamma}$ and $u \in L$ (see Definition \ref{defi:pr}).

\begin{thm}[{cf.\ Theorem \ref{thm:fingen}}]\label{thm:intro}
There exist 
\begin{itemize}
\item $\beta \in \mathbb{R}_{\ge 0}$, 
\item a triangulation $T_{\sigma}$ for each $\sigma \in \operatorname{Facet}(P_{\Gamma})$, and 
\item $\operatorname{cpx}_{\sigma} (v) \in \mathbb{Z}_{>0}$ and $\beta _{\Delta , v} \in \mathbb{R}_{\ge 0}$ for any $\sigma \in \operatorname{Facet}(P_{\Gamma})$, $\Delta \in T_{\sigma}$ and $v \in V(\Delta)$
\end{itemize}
with the condition ($\spadesuit$) below: 
For $\sigma \in \operatorname{Facet}(P_{\Gamma})$, $\Delta \in T_{\sigma}$, $\Delta' \in \operatorname{Face}(\Delta)$ and $\Delta'' \in \operatorname{Face}(\Delta')$, we define $L_{\Delta'}, L_{\Delta, \Delta', \Delta''} \subset L_{\mathbb{R}}$ by
\[
L_{\Delta'} := \mathbb{R}_{\ge 0} \Delta', \quad 
L_{\Delta, \Delta', \Delta''} := 
\sum_{v \in V(\Delta'')} (\beta _{\Delta , v}, \infty) \cdot v + 
\sum_{v \in V(\Delta') \setminus V(\Delta'')} [0,\beta _{\Delta , v}] \cdot v. 
\]
(Note that we have $L_{\Delta'} 
= 
\bigsqcup\limits_{\Delta'' \in \operatorname{Face}(\Delta')} 
L_{\Delta, \Delta', \Delta''}$.)
For a subset $F \subset L_{\mathbb{R}}$, we define $B \left( F \right) := B \cap \left( \mathbb{Z}_{\ge 0} \times \Phi^{-1} \left( F \right) \right)$. 
We also define a free submonoid $M_{\sigma , \Delta'} \subset \mathbb{Z}_{\ge 0} \times L$ by 
\[
M_{\sigma , \Delta'} := 
\mathbb{Z}_{\ge 0}(1,0) + 
\sum _{v \in V(\Delta')} \mathbb{Z}_{\ge 0} \left( \operatorname{cpx}_{\sigma} (v), \operatorname{cpx}_{\sigma} (v) v \right). 
\]
Then, for $\sigma \in \operatorname{Facet}(P_{\Gamma})$, $\Delta \in T_{\sigma}$, $\Delta' \in \operatorname{Face}(\Delta)$ and $\Delta'' \in \operatorname{Face}(\Delta')$, it follows that 
\begin{enumerate}
\item[($\spadesuit$)]
each $B \left( L_{\Delta, \Delta', \Delta''} \right)$ is a free $M_{\sigma , \Delta''}$-module. Furthermore, $B \left( L_{\Delta, \Delta', \Delta''} \right)$ is freely generated by finitely many elements whose degrees are at most $\beta + \sum _{v \in V(\Delta'')} \operatorname{cpx}_{\sigma} (v)$. 
\end{enumerate}
\end{thm}
\noindent
We define 
\[
\operatorname{cpx}_{\Gamma} :=
\operatorname{LCM} \bigl\{ \operatorname{cpx}_{\sigma}(v) \ \big | \ \sigma \in \operatorname{Facet}(P _{\Gamma}),\ \Delta \in T_{\sigma},\ v \in V(\Delta) \bigr \}.
\]
In Corollary \ref{cor:main}, by using the algebraic structure in Theorem \ref{thm:intro} (and the inclusion-exclusion principle), we prove that the generating function $G_{\Gamma, x_0}(t)$ is given by the form
\[
G_{\Gamma, x_0}(t) = \frac{Q(t)}{(1-t^{\operatorname{cpx}_{\Gamma}}) ^n}
\]
with some polynomial $Q(t)$ of $\deg Q \le \beta + n \cdot \operatorname{cpx}_{\Gamma}$. 

In \cite{SM20a} and \cite{GS19}, 
similar ideas of calculating the growth sequence by partitioning $V_{\Gamma}$ into finite regions, as in Theorem \ref{thm:intro}, also appear (although in both papers, the construction is explained only for a specific class of planar graphs and not for general graphs in arbitrary dimension). 
Theorem \ref{thm:intro} can be said to be a rigorous formulation and a generalization of their ideas. 

In \cite{IN}, the authors introduce a class of periodic graphs called ``well-arranged''. 
For well-arranged periodic graphs, they prove that $\beta = 0$ and $\beta _{\Delta, v} = 0$ satisfy the statement of Theorem \ref{thm:intro} (see \cite{IN}*{Claim 4.10}). As a consequence, they prove that their growth sequences are quasi-polynomial on $i  > 0$. 

In the two-dimensional case, 
it is relatively simple to implement our algorithm into a computer program 
since there is no need to consider the complicated triangulations of the facets of $P_{\Gamma}$ (see Subsection \ref{subsection:2dim} and Remark \ref{rmk:implement}). 
In fact, using the computer program, the growth sequences of the examples treated in \cite{GS19} can be computed automatically without separate consideration (see Subsection \ref{subsection:tilings}). 
In Subsection \ref{subsection:tilings}, we also examine a $6$-uniform tiling. 
As far as we know, this is the first time the growth series of this example has been determined with a proof. 
In Subsection \ref{subsection:ex1}, we illustrate the algorithm for a $3$-uniform tiling. 
In higher dimensions, although it is possible to implement the algorithm in a computer program, 
it is practical to give the triangulation by hand. 
In Subsections \ref{subsection:ex2} and \ref{subsection:ex3}, we illustrate the algorithm for the periodic graphs corresponding to two carbon allotropes called K6 and {\sf CFS}. 
There, the triangulation and other calculations are partially given by hand, and the rest is done by a computer program. 
As far as we know, this is the first time the growth series of these two examples have been determined with proofs.

The paper is organized as follows: 
in Section \ref{section:pre}, we introduce notations related to periodic graphs following \cite{IN}. 
In Subsection \ref{subsection:inv}, we define the invariants $\beta$ and $\operatorname{cpx}_{\Gamma}$. 
In Subsection \ref{subsection:2dim}, we explain the invariants for $2$-dimensional periodic graphs. 
In Subsection \ref{subsection:main}, we prove the main theorem (Theorem \ref{thm:fingen} and Corollary \ref{cor:main}). 
In Section \ref{section:ex}, we apply our algorithm to some specific periodic graphs. 
In Appendix \ref{section:a}, we briefly describe an input form of periodic graphs to implement the algorithm in a computer program.

\begin{ackn}
Figures \ref{fig:sacada12_1} and \ref{fig:sacada29_1} are drawn by using {\sf VESTA} (\cite{MI11}). 
We would like to thank Professors Atsushi Ito, Takafumi Mase, Junichi Nakagawa, and Hiroyuki Ochiai for many discussions. 
The first author is partially supported by FoPM, WINGS Program, the University of Tokyo and JSPS KAKENHI No.\ 23KJ0795.
The second author is partially supported by JSPS KAKENHI No.\ 18K13384, 22K13888 and JPJSBP120219935. 
\end{ackn}

\section{Notation}\label{section:notation}
\begin{itemize}
\item 
For a set $X$, $\# X$ denotes the cardinality of $X$, and $2 ^X$ denotes the power set of $X$.

\item
For a finite subset $S \subset \mathbb{Z}_{>0}$, $\operatorname{LCM}(S)$ denotes the least common multiple of the elements of $S$. 

\item 
For a polytope $P \subset \mathbb{R}^N$, 
$\operatorname{Facet}(P)$ denotes the set of facets of $P$, 
$\operatorname{Face}(P)$ denotes the set of faces of $P$, and 
$V(P)$ denotes the set of vertices of $P$. 
Note that both $P$ itself and the empty set $\emptyset$ are considered as faces of $P$. 

\item 
For a set $C \subset \mathbb{R}^N$, 
$\operatorname{int}(C)$ denotes the interior of $C$, and 
$\operatorname{relint}(C)$ denotes the relative interior of $C$. 

\item
For a polytope $\sigma \subset \mathbb{R}^N$ of dimension $d$, 
a \textit{triangulation} $T_{\sigma}$ means a finite collection of $d$-simplices with the following two conditions: 
\begin{itemize}
\item $\sigma = \bigcup _{\Delta \in T_{\sigma}} \Delta$. 
\item For any $\Delta _1, \Delta _2 \in T_{\sigma}$, $\Delta _1 \cap \Delta _2$ is a face of $\Delta _1$. 
\end{itemize}

\item 
Let $M$ be a set equipped with a binary operation $*$. 
For $u \in M$ and subsets $X, Y \subset M$, we define subsets $u * X, X * Y \subset M$ by
\[
u * X := \{ u * x \mid x \in X \}, \qquad
X * Y := \{ x * y \mid (x,y) \in X \times Y \}. 
\]

\item
In this paper, monoids always mean commutative monoids. 
We refer the reader to \cite{BG2009} and \cite{NSMN21} for the terminology of monoid and its module theory. 

\end{itemize}

\section{Preliminaries}\label{section:pre}

Following \cite{IN}, we introduce notations related to periodic graphs.

\subsection{Graphs and walks}

In this paper, a \textit{graph} means a directed weighted graph which may have loops and multiple edges. 
A graph $\Gamma = (V_{\Gamma}, E_{\Gamma}, s_{\Gamma}, t_{\Gamma}, w_{\Gamma})$ 
consists of the set $V_{\Gamma}$ of vertices, the set $E_{\Gamma}$ of edges, 
the source function $s_{\Gamma}: E_{\Gamma} \to V_{\Gamma}$, 
the target function $t_{\Gamma}: E_{\Gamma} \to V_{\Gamma}$, and 
the weight function $w_{\Gamma}: E_{\Gamma} \to \mathbb{Z}_{> 0}$. 
We often abbreviate $s_{\Gamma}$, $t_{\Gamma}$ and $w_{\Gamma}$ to $s$, $t$ and $w$ 
when no confusion can arise. 

\begin{defi}
Let $\Gamma = (V_{\Gamma}, E_{\Gamma}, s, t, w)$ be a graph. 

\begin{enumerate}
\item 
$\Gamma$ is called to be \textit{unweighted} if $w(e) = 1$ holds for any $e \in E_{\Gamma}$. 
$\Gamma$ is called to be \textit{undirected} when 
there exists an involution $E_{\Gamma} \to E_{\Gamma};\ e \mapsto e'$
such that $s(e)=t(e')$, $t(e)=s(e')$ and $w(e)=w(e')$. 

\item
A \textit{walk} $p$ in $\Gamma$ is a sequence $e_1 e_2 \cdots e_{\ell}$ of edges $e_i$ of $\Gamma$ satisfying 
$t(e_i) = s(e_{i+1})$ for each $i = 1, \ldots , \ell -1$. 
We define 
\[
s(p) := s(e_1), \quad
t(p) := t(e_{\ell}), \quad
w(p) := \sum _{i = 1} ^{\ell} w(e_i), \quad 
\operatorname{length}(p) := \ell. 
\]
Note that we have $w(p) = \operatorname{length}(p)$ if $\Gamma$ is unweighted.

We say that ``$p$ is a walk from $x$ to $y$" when $x = s(p)$ and $y = t(p)$. 
We also define the \textit{support} $\operatorname{supp}(p) \subset V_{\Gamma}$ of $p$ by
\[
\operatorname{supp}(p) := \{ s(e_1), t(e_1), t(e_2), \ldots , t(e_{\ell}) \} \subset V_{\Gamma}. 
\]

By convention, each vertex $v \in V_{\Gamma}$ is also considered as a walk of length $0$. 
This is called the \textit{trivial walk} at $v$ and denoted by $\emptyset _v$: i.e., 
we define 
\[
s(\emptyset _v) := v, \ \ 
t(\emptyset _v) := v, \ \ 
w(\emptyset _v) := 0, \ \  
\operatorname{length}(\emptyset _v) := 0, \ \ 
\operatorname{supp}(\emptyset _v) := \{ v \}.
\]

\item 
A \textit{path} in $\Gamma$ is a walk $e_1 \cdots e_{\ell}$ such that $s(e_1), t(e_1), t(e_2), \ldots, t(e_{\ell})$ are distinct. 
A walk of length $0$ is considered as a path. 

\item
A \textit{cycle} in $\Gamma$ is a walk $e_1 \cdots e_{\ell}$ with $s(e_1) = t(e_{\ell})$ 
such that $t(e_1), t(e_2), \ldots, t(e_{\ell})$ are distinct. 
A walk of length $0$ is \textbf{NOT} considered as a cycle. 
$\operatorname{Cyc}_{\Gamma}$ denotes the set of cycles in $\Gamma$. 

\item 
For $x, y \in V_{\Gamma}$, $d_{\Gamma}(x, y) \in \mathbb{Z}_{\ge 0} \cup \{ \infty \}$ denotes the smallest weight $w(p)$ of any walk $p$ from $x$ to $y$. 
By convention, we define $d_{\Gamma}(x, y) = \infty$ when there is no walk from $x$ to $y$. 
A graph $\Gamma$ is said to be \textit{strongly connected} 
when we have $d_{\Gamma}(x, y) < \infty$ for any $x, y \in V_{\Gamma}$. 
When $\Gamma$ is undirected, we have $d_{\Gamma}(x,y) = d_{\Gamma}(y,x)$ for any $x,y \in V_{\Gamma}$. 

\item 
$C_1(\Gamma, \mathbb{Z})$ denotes the group of $1$-chains on $\Gamma$ with coefficients in $\mathbb{Z}$, i.e., 
$C_1(\Gamma, \mathbb{Z})$ is a free abelian group generated by $E_{\Gamma}$. 
For a walk $p = e_1 \cdots e_{\ell}$ in $\Gamma$, let $\langle p \rangle$ denote the $1$-chain 
$\sum _{i = 1} ^{\ell} e_i \in C_1(\Gamma, \mathbb{Z})$. 
$H_1(\Gamma, \mathbb{Z}) \subset C_1(\Gamma, \mathbb{Z})$ denotes the $1$-st homology group, i.e., 
$H_1(\Gamma, \mathbb{Z})$ is a subgroup generated by $\langle p \rangle$ for $p \in \operatorname{Cyc}_{\Gamma}$. 
We refer the reader to \cite{Sunada} for more detail. 
\end{enumerate}
\end{defi}

\subsection{Periodic graphs}

\begin{defi}\label{defi:pg}
Let $n$ be a positive integer. 
An \textit{$n$-dimensional periodic graph} $(\Gamma, L)$ is a graph $\Gamma$ and 
a free abelian group $L \simeq \mathbb{Z}^n$ of rank $n$ with the following two conditions:
\begin{itemize}
\item $L$ freely acts on both $V_\Gamma$ and $E_\Gamma$, 
and their quotients $V_\Gamma / L$ and $E_\Gamma / L$ are finite sets. 
\item This action preserves the edge relations and the weight function, i.e.,\ for any $u \in L$ and $e \in E_\Gamma$, 
we have $s_\Gamma(u(e))=u(s_\Gamma(e))$, $t_\Gamma(u(e))=u(t_\Gamma(e))$ and 
$w_{\Gamma}(u(e))=w_{\Gamma}(e)$.
\end{itemize}
\end{defi}

If $(\Gamma, L)$ is an $n$-dimensional periodic graph, 
then the \textit{quotient graph} $\Gamma /L = (V_{\Gamma/L}, E_{\Gamma/L}, s_{\Gamma/L},t_{\Gamma/L},w_{\Gamma/L})$ is defined by 
$V_{\Gamma/L} := V_{\Gamma} /L$, $E_{\Gamma/L} := E_{\Gamma} /L$, 
and the functions $s_{\Gamma/L}: E_\Gamma/L \to V_\Gamma/L$, 
$t_{\Gamma/L}: E_\Gamma/L \to V_\Gamma/L$, and 
$w_{\Gamma/L}: E_\Gamma/L \to \mathbb{Z}_{>0}$ induced from $s_{\Gamma}$, $t_{\Gamma}$, and $w_{\Gamma}$. 
Note that the functions $s_{\Gamma/L}$, $t_{\Gamma/L}$ and $w_{\Gamma/L}$ 
are well-defined due to the second condition in Definition \ref{defi:pg}.

\begin{defi}\label{defi:peri}
Let $(\Gamma,L)$ be an $n$-dimensional periodic graph.
\begin{enumerate}
\item Since $L$ is an abelian group, we use the additive notation: 
for $u \in L$, $x \in V_{\Gamma}$, $e \in E_{\Gamma}$ and a walk $p = e_1 \cdots e_{\ell}$,
\[
u+x := u(x), \quad u+e := u(e), \quad u+p := u(e_1) \cdots u(e_{\ell})
\]
denote their translations by $u$. 

\item 
For any $x \in V_{\Gamma}$ and $e \in E_{\Gamma}$, 
let $\overline{x} \in V_{\Gamma /L}$ and $\overline{e} \in E_{\Gamma /L}$ denote their images in 
$V_{\Gamma /L} = V_{\Gamma} /L$ and $E_{\Gamma /L} = E_{\Gamma} /L$. 
For a walk $p = e_1 \cdots e_{\ell}$ in $\Gamma$, we define its image in $\Gamma /L$ by $\overline{p} := \overline{e_1} \cdots \overline{e_{\ell}}$. 

\item 
When $x, y \in V_{\Gamma}$ satisfy $\overline{x} = \overline{y}$, there exists an element $u \in L$ such that $u + x = y$. 
Since the action is free, such $u \in L$ uniquely exists and is denoted by $y - x$. 

\item 
For a walk $p$ in $\Gamma$ with $\overline{s(p)} = \overline{t(p)}$, we define 
\[
\operatorname{vec}(p) := t(p) - s(p) \in L. 
\]
\end{enumerate}
\end{defi}

\begin{defi}\label{defi:pr}
Let $(\Gamma,L)$ be an $n$-dimensional periodic graph.
We define $L_{\mathbb{R}} := L \otimes _{\mathbb{Z}} \mathbb{R}$. 
\begin{enumerate}
\item 
A \textit{periodic realization} $\Phi: V_{\Gamma} \to L_{\mathbb{R}}$ is a map satisfying
$\Phi(u+x) = u + \Phi(x)$ for any $u \in L$ and $x \in V _{\Gamma}$. 

\item 
Let $\Phi$ be a periodic realization of $(\Gamma, L)$. 
For an edge $e$ and a walk $p$ in $\Gamma$, we define 
\begin{align*}
\operatorname{vec}_{\Phi}(e) &:= \Phi(t(e)) - \Phi(s(e)) \in L_{\mathbb{R}}, \\
\operatorname{vec}_{\Phi}(p) &:= \Phi(t(p)) - \Phi(s(p)) \in L_{\mathbb{R}}. 
\end{align*}
It is easy to see that the value $\operatorname{vec}_{\Phi}(e) \in L_{\mathbb{R}}$ depends only on the class $\overline{e} \in E_{\Gamma /L}$, 
and therefore, 
the map 
\[
\mu _{\Phi} : E_{\Gamma /L} \to L_{\mathbb{R}};\quad \overline{e} \mapsto \operatorname{vec}_{\Phi}(e)
\]
is well-defined. 
It can be extended to a homomorphism 
\[
\mu _{\Phi}: C_1(\Gamma/L, \mathbb{Z}) \to L_{\mathbb{R}}; \quad 
\sum a_i \overline{e_i} \mapsto \sum a_i \mu _{\Phi} (\overline{e_i}). 
\]
By construction, we have $\mu _{\Phi} (\langle \overline{p} \rangle) = \operatorname{vec}_{\Phi}(p)$ for any walk $p$ in $\Gamma$. 

\item 
It is known that the restriction map $\mu _{\Phi} |_{H_1(\Gamma/L, \mathbb{Z})}: H_1(\Gamma/L, \mathbb{Z}) \to L_{\mathbb{R}}$ is independent of the choice of $\Phi$ and that its image is contained in $L$ (see \cite{IN}*{Lemma 2.7}). 
This restriction map is denoted by $\mu : H_1(\Gamma/L, \mathbb{Z}) \to L$. 
\end{enumerate}
\end{defi}

\begin{rmk}\label{rmk:finmap}
In Definition \ref{defi:pr}(2), we have $\operatorname{vec}_{\Phi}(p) = \operatorname{vec}(p)$
for any $p$ satisfying $\overline{s(p)} = \overline{t(p)}$. 
\end{rmk}

We finish this subsection with an observation from \cite{IN} on the decomposition and the composition of walks.
The notation differs slightly from that in \cite{IN}, but is essentially the same.

\begin{defi}[{cf.\ \cite{IN}*{Definition 2.11}}]
Let $(\Gamma, L)$ be an $n$-dimensional periodic graph. 
Let $q_0$ be a path in $\Gamma /L$, and let $a : \operatorname{Cyc}_{\Gamma /L} \to \mathbb{Z}_{\ge 0}$ be a function. 
The pair $(q_0, a)$ is called to be \textit{walkable} if there exists a walk $q'$ in $\Gamma / L$ such that 
$\langle q' \rangle = \langle q_0 \rangle + \sum _{q \in \operatorname{Cyc}_{\Gamma /L}} a(q) \langle q \rangle$. 
\end{defi}

\begin{lem}[\cite{IN}*{Lemma 2.12, Remark 2.13}]\label{lem:bunkai}
Let $(\Gamma, L)$ be an $n$-dimensional periodic graph. 
\begin{enumerate}
\item 
For a walk $q'$ in $\Gamma / L$, there exists a walkable pair $(q_0, a)$ such that 
$\langle q' \rangle = \langle q_0 \rangle + \sum _{q \in \operatorname{Cyc}_{\Gamma /L}} a(q) \langle q \rangle$. 
Furthermore, if $q'$ satisfies $s(q') = t(q')$, then $q_0$ should be a trivial path. 

\item 
Let $q_0$ be a path in $\Gamma /L$, and let $a : \operatorname{Cyc}_{\Gamma /L} \to \mathbb{Z}_{\ge 0}$ be a function. 
Then, $(q_0, a)$ is walkable if and only if 
there exists a sequence $q_1, \ldots , q_{\ell} \in a^{-1}(\mathbb{Z}_{>0})$ satisfying
$a^{-1}(\mathbb{Z}_{>0}) = \{ q_1, \ldots , q_{\ell} \}$ such that 
\[
\biggl( \operatorname{supp}(q_0) \cup \bigcup _{1 \le i \le k} \operatorname{supp}(q_i) \biggr) 
\cap \operatorname{supp}(q_{k+1}) \not = \emptyset
\] 
holds for any $0 \le k \le \ell -1$.

\item
For a walk $q$ in $\Gamma / L$ and a vertex $x \in V_{\Gamma}$ satisfying $s(q) = \overline{x}$, there exists a unique walk $p$ in $\Gamma$ such that $q = \overline{p}$ and $s(p) = x$. 
\end{enumerate}
\end{lem}

\subsection{Growth sequences of periodic graphs}\label{subsection:CS}
Let $\Gamma$ be a locally finite graph, and let $x_0 \in V_{\Gamma}$. 
For $i \in \mathbb{Z}_{\ge 0}$, we define subsets $B_{\Gamma, x_0, i}, S_{\Gamma, x_0,i} \subset V _{\Gamma}$ by
\[
B_{\Gamma, x_0,i} := \{ y \in V _{\Gamma} \mid d_{\Gamma}(x_0, y) \le i \}, \quad
S_{\Gamma, x_0,i} := \{ y \in V _{\Gamma} \mid d_{\Gamma}(x_0, y) = i \}. 
\]
Let $b_{\Gamma, x_0,i} := \# B_{\Gamma, x_0,i}$ and $s_{\Gamma, x_0,i} := \# S_{\Gamma, x_0,i}$ denote their cardinarities. 
The sequences $(s_{\Gamma, x_0,i})_i$ and $(b_{\Gamma, x_0,i})_i$ are called the \textit{growth sequence} and the \textit{cumulative growth sequence} of $\Gamma$ with the start point $x_0$, respectively. 

The \textit{growth series} $G_{\Gamma, x_0}(t)$ of $\Gamma$ with the start point $x_0$ is the generating function 
\[
G_{\Gamma, x_0}(t) := \sum _{i \ge 0} s_{\Gamma, x_0,i} t^i
\]
of the growth series $(s_{\Gamma, x_0,i})_{i}$.

The growth sequences of periodic graphs are known to be of quasi-polynomial type. 
\begin{thm}[{\cite{NSMN21}*{Theorem 2.2}}]\label{thm:NSMN}
Let $(\Gamma, L)$ be a periodic graph, and let $x_0 \in V_{\Gamma}$. 
Then, the functions $b: i \mapsto b_{\Gamma, x_0, i}$ and $s: i \mapsto s_{\Gamma, x_0, i}$ are of quasi-polynomial type (see Definition \ref{defi:qp} below). 
In particular, its growth series is rational. 
\end{thm}

\begin{defi}[{cf.\ \cite{Sta96}*{Chapter 0}}]\label{defi:qp}
\begin{enumerate}
\item \label{item:f}
A function $f: \mathbb{Z} \to \mathbb{C}$ is called a \textit{quasi-polynomial} 
if there exist a positive integer $N$ and polynomials $Q_0, \ldots, Q_{N-1} \in \mathbb{C}[x]$ such that 
\[
f(n) = \begin{cases} 
Q_0(n) & \text{if $n \equiv 0 \pmod N$, } \\
Q_1(n) & \text{if $n \equiv 1 \pmod N$, } \\
\hspace{5mm} \vdots & \\
Q_{N-1}(n) & \text{if $n \equiv N-1 \pmod N$. }
\end{cases}
\]

\item 
A function $g: \mathbb{Z} \to \mathbb{C}$ is called to be of \textit{quasi-polynomial type} if 
there exists a non-negative integer $M \in \mathbb{Z}_{\ge 0}$ and a quasi-polynomial $f$ such that
$g(n) = f(n)$ holds for any $n > M$. 
The positive integer $N$ is called a \textit{quasi-period} of $g$ when $f$ is of the form in (\ref{item:f}). 
Note that the notion of quasi-period is not unique. 
The minimum quasi-period is called the \textit{period} of $g$. 
We say that the function $g$ is a \textit{quasi-polynomial on} $n \ge m$ if 
$g(n) = f(n)$ holds for $n \ge m$. 
\end{enumerate}
\end{defi}

\subsection{Growth polytope} \label{subsection:GP}
In this subsection, according to \cite{IN}, we define the \textit{growth polytope} $P_{\Gamma} \subset L_{\mathbb{R}}$ for a periodic graph $(\Gamma, L)$. The concept of a growth polytope has been defined and studied in various papers \cites{KS02, KS06, Zhu02, Eon04, MS11, Fri13, ACIK}. 

\begin{defi}\label{defi:P}
Let $(\Gamma , L)$ be an $n$-dimensional periodic graph. 
\begin{enumerate}
\item 
We define the \textit{normalization map} 
$\nu : \operatorname{Cyc}_{\Gamma /L} \to L_{\mathbb{R}} := L \otimes _{\mathbb{Z}} \mathbb{R}$ by 
\[
\nu: \operatorname{Cyc}_{\Gamma /L} \to L_{\mathbb{R}}; \quad p \mapsto \frac{\mu(\langle p \rangle )}{w(p)}. 
\]
We define the \textit{growth polytope}
\[
P_{\Gamma} := \operatorname{conv} \bigl( \operatorname{Im}(\nu) \cup \{ 0 \} \bigr) \subset L_{\mathbb{R}}
\]
as the convex hull of the set 
$\operatorname{Im}(\nu) \cup \{ 0 \} \subset L_{\mathbb{R}}$. 
 
\item 
For a polytope $Q \subset L_{\mathbb{R}}$ and $y \in L_{\mathbb{R}}$, we define 
\[
d_{Q}(y) := \min \{ t \in \mathbb{R}_{\ge 0} \mid x \in t Q \} \in \mathbb{R}_{\ge 0} \cup \{ \infty \}. 
\]
When $0 \in \operatorname{int}(Q)$, 
we have $d_{Q}(y) < \infty$ for any $y \in L_{\mathbb{R}}$. 

\item
For a periodic realization $\Phi:V_{\Gamma} \to L_{\mathbb{R}}$, we define 
\[
d_{P_{\Gamma}, \Phi}(x, y) := d_{P_{\Gamma}}\bigl( \Phi(y) - \Phi(x) \bigr)
\]
for $x, y \in V_{\Gamma}$. 
\end{enumerate}
\end{defi}

\begin{rmk}
Note that $P_{\Gamma}$ is a rational polytope (i.e., $P_{\Gamma}$ is a polytope whose vertices are on $L_{\mathbb{Q}} := L \otimes _{\mathbb{Z}} \mathbb{Q}$). 
This is because $\operatorname{Cyc}_{\Gamma /L}$ is a finite set, and we have $\operatorname{Im}(\nu) \subset L_{\mathbb{Q}}$. 
Furthermore, when $\Gamma$ is strongly connected, we have $0 \in \operatorname{int}(P_{\Gamma})$ by \cite{Fri13}*{Proposition 21} (cf.\ \cite{IN}*{Lemma A.1}). 
\end{rmk}

We define $C_{1}(\Gamma, \Phi, x_0)$ and $C_{2}(\Gamma, \Phi, x_0)$ as invariants that measure the difference between $d_{\Gamma}$ and $d_{P_{\Gamma}, \Phi}$. 

\begin{defi}\label{defi:C}
Let $(\Gamma , L)$ be a strongly connected periodic graph. 
Let $\Phi : V_{\Gamma} \to L_{\mathbb{R}}$ be a periodic realization, and let $x_0 \in V_{\Gamma}$. 
Then, we define 
\begin{align*}
C_{1}(\Gamma, \Phi, x_0) &:= \sup _{y \in V_{\Gamma}} \bigl( d_{P_{\Gamma}, \Phi} (x_0, y) - d_{\Gamma} (x_0, y) \bigr), \\
C_{2}(\Gamma, \Phi, x_0) &:= \sup _{y \in V_{\Gamma}} \bigl( d_{\Gamma}(x_0, y) - d_{P_{\Gamma}, \Phi} ( x_0, y ) \bigr). 
\end{align*}
By \cite{IN}*{Theorem A.2}, we have $C_{1}(\Gamma, \Phi, x_0) < \infty$ and $C_{2}(\Gamma, \Phi, x_0) < \infty$. 
\end{defi}

\begin{rmk}
As in Proposition \ref{prop:C2}(1), it is easy to determine $C_{1}(\Gamma, \Phi, x_0)$. 
However, it is not easy to determine $C_{2}(\Gamma, \Phi, x_0)$ in general, and we just give an upper bound of it in Proposition \ref{prop:C2}(2). 
In Theorem \ref{thm:C2}, we give an algorithm to determine $C_{2}(\Gamma, \Phi, x_0)$ using the invariants defined in Section \ref{section:main}. 
\end{rmk}

\begin{prop}[{cf.\ \cite{IN}*{Theorem A.2}}]\label{prop:C2}
Let $(\Gamma , L)$ be a strongly connected periodic graph. 
Let $\Phi : V_{\Gamma} \to L_{\mathbb{R}}$ be a periodic realization, and let $x_0 \in V_{\Gamma}$. 
\begin{enumerate}
\item 
We have 
\[
C_{1}(\Gamma, \Phi, x_0) = 
\max _{y \in B'_{c-1}} \bigl( d_{P_{\Gamma}, \Phi} (x_0, y) - d_{\Gamma} (x_0, y) \bigr), 
\]
where $c := \# (V_{\Gamma} /L)$ and 
\[
B'_{c-1} := \{ y \in V_{\Gamma} \mid \text{there exists a walk $p$ from $x_0$ to $y$ with $\operatorname{length}(p) \le c-1$}\}. 
\]

\item 
We have $C_{2}(\Gamma, \Phi, x_0) \le C'_2$ when we define $C'_2$ as follows: 
\begin{itemize}
\item 
First, we define $d_v := \min _{q \in \nu ^{-1}(v)} w(q)$ for each $v \in V(P_{\Gamma})$. 

\item
For each $\sigma \in \operatorname{Facet}(P_{\Gamma})$, we fix a triangulation $T_{\sigma}$ of $\sigma$ such that $V(\Delta) \subset V(\sigma)$ holds for any $\Delta \in T_{\sigma}$. 

\item
We define a bounded set $Q \subset L_{\mathbb{R}}$ as follows: 
\[
Q := 
\bigcup _{\substack{\sigma \in \operatorname{Facet}(P_{\Gamma}),\\ 
\Delta \in T_{\sigma}}}
\left( \sum _{v \in V(\Delta)}  [0,1) d_v v \right)
\subset L_{\mathbb{R}}. 
\]

\item
For $y \in V_{\Gamma}$, we define $d'(x_0, y)$ as the smallest weight $w(p)$ of 
a walk $p$ from $x_0$ to $y$ satisfying $\operatorname{supp}(\overline{p}) = V_{\Gamma} /L$. 

\item
Then, we set
\[
C' _2 := \max \bigl \{ d'(x_0, y) - d_{P_{\Gamma}, \Phi} ( x_0, y ) \ \big | \ 
y \in V_{\Gamma}, \ \Phi(y) - \Phi(x_0) \in Q \bigr \}. 
\]
\end{itemize}

\end{enumerate}
\end{prop}
\begin{proof}
In the proof of \cite{IN}*{Theorem A.2}, (1) and (2) are proved. 
\end{proof}

\section{Stratified Ehrhart ring theory}\label{section:main}

Throughout this section, we fix a strongly connected $n$-dimensional periodic graph $(\Gamma , L)$. 
We also fix $x_0 \in V_{\Gamma}$. 
Let $(b(d))_d$ be the cumulative growth sequence of $\Gamma$ with the start point $x_0$. 
By Theorem \ref{thm:NSMN}, it is known that the function $b: d \mapsto b(d)$ is of quasi-polynomial type. 
The goal of this section is to give an algorithm for finding a quasi-period and an integer $m$ such that 
the function $b$ is a quasi-polynomial on $d \ge m$. 

In Subsection \ref{subsection:inv}, 
we will define invariants $\operatorname{cpx}_{\Gamma} \in \mathbb{Z}_{>0}$ and $\beta \in \mathbb{R}_{\ge 0}$. 
In Subsection \ref{subsection:main}, 
we will prove that the function $b$ is a quasi-polynomial on $d > \beta -1$, 
and $\operatorname{cpx}_{\Gamma}$ is its quasi-period 
(Corollary \ref{cor:main}). 

\subsection{Invariants}\label{subsection:inv}

Let $P := P_{\Gamma}$ be the growth polytope of $\Gamma$. 
We define $\nu: \operatorname{Cyc}_{\Gamma /L} \to L_{\mathbb{R}}$ as in Definition \ref{defi:P}. 
We also fix a periodic realization $\Phi: V_{\Gamma} \to L_{\mathbb{R}}$ such that $\Phi (x_0) = 0$. 
Let $C_1 := C_{1}(\Gamma, \Phi, x_0)$ and $C_2:= C_{2}(\Gamma, \Phi, x_0)$ (see Definition \ref{defi:C}). 

We will define $\beta$ and $\operatorname{cpx}_{\Gamma}$ in Subsection \ref{subsubsection:beta}: 
\begin{itemize}
\item 
In Subsection \ref{subsubsection:T}, for each $\sigma \in \operatorname{Facet}(P)$, 
we take a triangulation $T_{\sigma}$  with conditions $(\diamondsuit)_1$ and $(\diamondsuit)_2$. 

\item
In Subsection \ref{subsubsection:H}, for each $\sigma \in \operatorname{Facet}(P)$ and $v \in \sigma$, 
we define a subset $\mathcal{H}'' _{\sigma, v} \subset 2 ^{\operatorname{Im}(\nu) \cap \sigma}$.

\item 
In Subsection \ref{subsubsection:rs}, 
for $\sigma \in \operatorname{Facet}(P)$, $\Delta \in T_{\sigma}$, $v \in V(\Delta)$ and $F \in \mathcal{H}'' _{\sigma, v}$, 
we take invariants $\operatorname{cpx}_{\sigma}(v) \in \mathbb{Z}_{>0}$, $r^{F} _{\sigma, v} \in \mathbb{Z}_{\ge 0}$ and $s^{F} _{\sigma, v} \in \mathbb{Z}_{\ge 0}$. 

\item 
In Subsection \ref{subsubsection:ah}, 
for $\sigma \in \operatorname{Facet}(P)$, $\Delta \in T_{\sigma}$, $v \in V(\Delta)$ and $F \in \mathcal{H}'' _{\sigma, v}$, 
we take invariants $a^{F} _{\Delta, v} (v) \in (0,1)$ and $h^F_{\Delta, v} \in (0,1)$. 

\item 
In Subsection \ref{subsubsection:beta}, using the invariants above, 
we define invariants  $\beta \in \mathbb{R}_{\ge 0}$ and $\operatorname{cpx}_{\Gamma} \in \mathbb{Z}_{> 0}$. 
\end{itemize}

\subsubsection{Triangulation $T_{\sigma}$}\label{subsubsection:T}
For each facet $\sigma$ of $P$, we take a triangulation $T_{\sigma}$ of $\sigma$ (see Section \ref{section:notation}) with the following two conditions: 
\begin{itemize}
\item[$(\diamondsuit) _1$]
$V(\Delta) \subset L_{\mathbb{Q}}$ holds for any $\Delta \in T_{\sigma}$. 

\item[$(\diamondsuit) _2$]
For any $\Delta \in T_{\sigma}$ and any subset $F \subset \operatorname{Im}(\nu) \cap \sigma$, 
$\Delta \cap \operatorname{conv}(F)$ is a face of $\Delta$. 
\end{itemize}
See Lemma \ref{lem:tri} for the existence of such a triangulation. 
The rationality condition $(\diamondsuit) _1$ will be used in the proof of Lemma \ref{lem:r}. 
Condition $(\diamondsuit) _2$ will be used in the proof of Lemma \ref{lem:a} (see Lemma \ref{lem:sep}). 

\begin{lem}\label{lem:tri}
For each $\sigma \in \operatorname{Facet}(P)$, there exists a triangulation $T_{\sigma}$ of $\sigma$ with the conditions $(\diamondsuit) _1$ and $(\diamondsuit) _2$. 
\end{lem}
\begin{proof}
We fix $\sigma \in \operatorname{Facet}(P)$. 
First, we define a set $\mathcal{S}$ of polytopes by 
\[
\mathcal{S} := \left\{ \operatorname{conv}(F) \ \middle | \  F \in 2^{\operatorname{Im}(\nu) \cap \sigma} \right \}. 
\]
Then, since $\mathcal{S}$ is a finite set of polytopes, by taking a subdivision, we can construct a triangulation $T_{\sigma}$ of $\sigma$ with the following condition: 
\begin{itemize}
\item[(i)] 
For any $Q \in \mathcal{S}$, there exist $\Delta '_1, \ldots, \Delta '_{\ell} 
\in \{ \Delta ' \mid \Delta \in T_{\sigma},\ \Delta ' \in \operatorname{Face}(\Delta) \}$
such that $Q = \bigcup _{i = 1} ^{\ell} \Delta ' _{i}$. 
\end{itemize}
Since $\operatorname{Im}(\nu) \subset L_{\mathbb{Q}}$, $\mathcal{S}$ consists of rational polytopes. 
Therefore, we may take such a triangulation $T_{\sigma}$ to satisfy $(\diamondsuit)_1$. 

We prove that the $T_{\sigma}$ satisfies $(\diamondsuit)_2$. 
Let $F \subset \operatorname{Im}(\nu) \cap \sigma$ and $\Delta \in T_{\sigma}$. 
By (i), $\Delta \cap \operatorname{conv}(F)$ is a union of some faces of $\Delta$. 
Since $\Delta \cap \operatorname{conv}(F)$ is a convex set, $\Delta \cap \operatorname{conv}(F)$ should be a face of $\Delta$. 
\end{proof}

\begin{ex}
Figure \ref{fig:T} shows the example of $T_{\sigma}$. 
In this example, we assume that $\sigma$ is a quadrilateral with vertices $u_1, \ldots, u_4$ and 
$\operatorname{Im}(\nu) \cap \sigma = V(\sigma) = \{ u_1, \ldots , u_4 \}$. 
Then, $T_{\sigma} = \{ \Delta _1,  \Delta _2, \Delta _3, \Delta _4 \}$ in the figure satisfies the condition $(\diamondsuit)_2$. 
\begin{figure}[htbp]
\centering
\includegraphics[width=14cm]{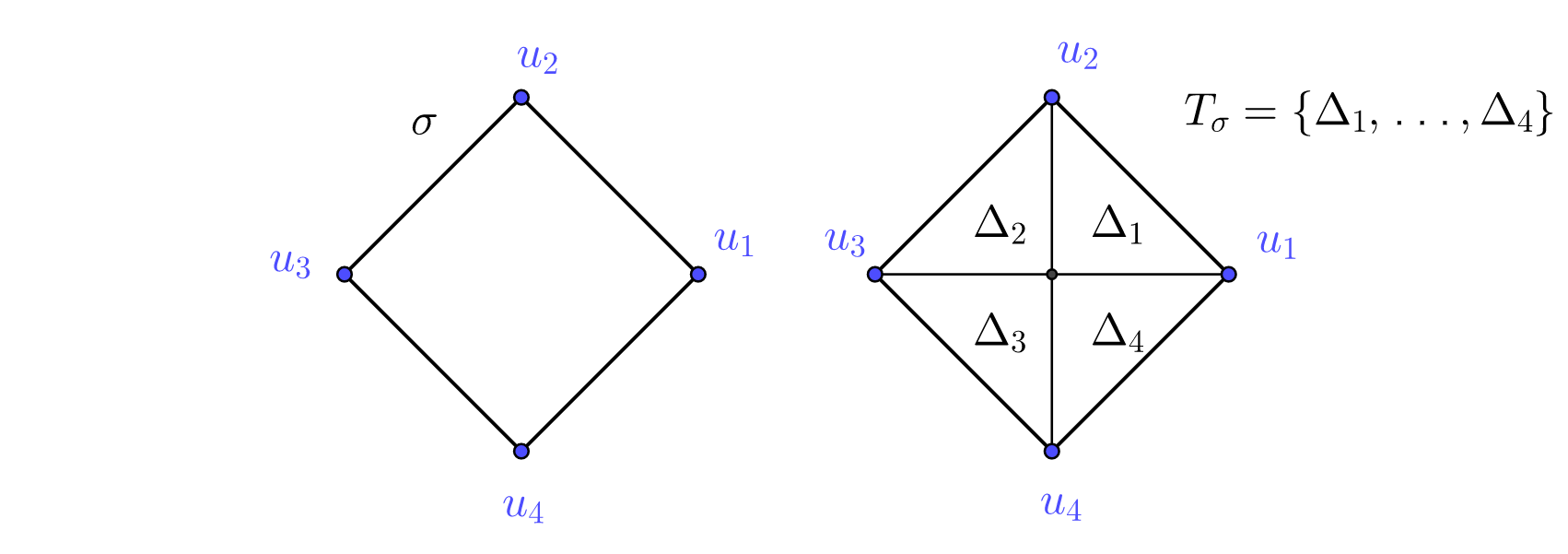}
\caption{$T_{\sigma}$ for a quadrilateral $\sigma$.}
\label{fig:T}
\end{figure}
\end{ex}

\subsubsection{$\mathcal{H}''_{\sigma, v}$}\label{subsubsection:H}
\begin{defi}
Let $\sigma \in \operatorname{Facet}(P)$. 
Let $\overline{\sigma}$ denote the hyperplane of $L_{\mathbb{R}}$ that contains $\sigma$. 
For $v \in \overline{\sigma}$, let $\mathcal{H}_{\sigma, v}$ denote the set of all closed half-spaces $H$ of 
$\overline{\sigma}$ satisfying $v \in \partial H := H  \setminus \operatorname{relint}(H)$. 
Let 
\[
\mathcal{H}' _{\sigma, v} := \{ \operatorname{Im}(\nu) \cap H \mid H \in \mathcal{H}_{\sigma, v} \} 
\subset 2 ^{\operatorname{Im}(\nu) \cap \sigma}. 
\]
Note that $\mathcal{H}' _{\sigma, v}$ is a finite set since $\operatorname{Im}(\nu)$ is a finite set. 
Let $\mathcal{H}'' _{\sigma, v} \subset \mathcal{H}' _{\sigma, v}$ 
be the set of minimal elements of $\mathcal{H}' _{\sigma, v}$ with respect to the inclusion. 
\end{defi}

See Lemma \ref{lem:H} for a property of $\mathcal{H}'' _{\sigma, v}$. 

\begin{lem}\label{lem:H}
Let $\sigma \in \operatorname{Facet}(P)$ and $v \in \sigma$. 
For each $F \in \mathcal{H}'' _{\sigma, v}$, we pick arbitrary $u_F \in F$. 
For any choice of $u_F$, we have $v \in \operatorname{conv}\bigl ( \{ u_F \mid F \in \mathcal{H}'' _{\sigma, v} \} \bigr)$. 
\end{lem}
\begin{proof}
Supose the contrary that $v \not \in \operatorname{conv}\bigl ( \{ u_F \mid F \in \mathcal{H}'' _{\sigma, v} \} \bigr)$. 
Then, $\operatorname{conv}\bigl ( \{ u_F \mid F \in \mathcal{H}'' _{\sigma, v} \} \bigr)$ 
should be contained in an open half-space of $\overline{\sigma}$ whose boundary passes through $v$, 
and hence, we have
\[
\{ u_F \mid F \in \mathcal{H}'' _{\sigma, v} \} \cap H = \emptyset
\]
for some $H \in \mathcal{H}_{\sigma, v}$. 
Let $F_H := \operatorname{Im}(\nu) \cap H \in \mathcal{H}' _{\sigma, v}$. 
Then, by the definition of $\mathcal{H}'' _{\sigma, v}$, there exists $F' \in \mathcal{H}'' _{\sigma, v}$ such that $F' \subset F_H$.
Since we have 
\[
u_{F'} \in \{ u_F \mid F \in \mathcal{H}'' _{\sigma, v} \}, \qquad 
u_{F'} \in F' \subset F_H \subset H, 
\]
we get a contradiction. 
\end{proof}

\begin{ex}
For example, in Figure \ref{fig:H}, 
we have 
\begin{align*}
\mathcal{H}' _{\sigma, v} 
= \bigl\{ 
&\{ x_1 \}, 
\{ x_1, x_2 \}, 
\{ x_1, x_2 , x_6 \}, 
\{ x_2 , x_6 \}, 
\{ x_2, x_6 , x_3 \}, \\
&\{ x_2, x_6 , x_3 , x_4 \}, 
\{ x_2, x_6 , x_3 , x_4, x_5 \}, 
\{ x_6 , x_3 , x_4, x_5 \}, \\
&\{ x_3 , x_4, x_5 \}, 
\{ x_3 , x_4, x_5, x_1 \},
\{ x_4, x_5, x_1 \}, 
\{ x_5, x_1 \}
\bigr\}, \\
\mathcal{H}'' _{\sigma, v} 
= \bigl\{ 
&\{ x_1 \}, 
\{ x_2 , x_6 \}, 
\{ x_3 , x_4, x_5 \}
\bigr\}. 
\end{align*}
\begin{figure}[htbp]
\centering
\includegraphics[width=6cm]{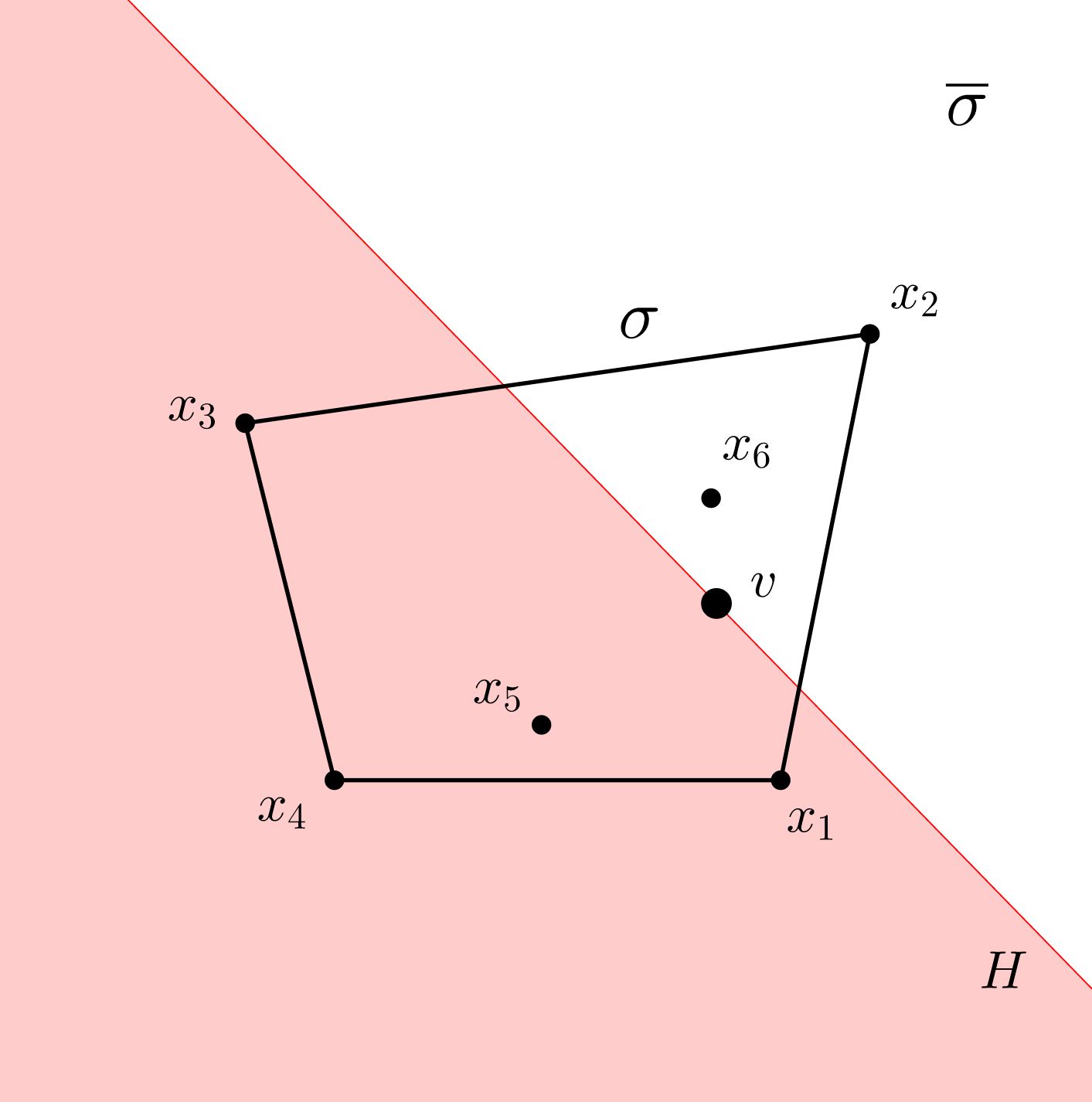}
\caption{$\operatorname{Im}(\nu) \cap \sigma = \{ x_1, \ldots , x_6 \}$.}
\label{fig:H}
\end{figure}
\end{ex}

\subsubsection{ $\operatorname{cpx}_{\sigma}(v)$, $r^{F} _{\sigma, v}$ and $s^{F} _{\sigma, v}$}\label{subsubsection:rs}
\begin{defi}
For a subset $F \subset L_{\mathbb{R}}$, 
we define $\operatorname{Cyc}_{\Gamma /L}(F) := \nu ^{-1} (F) \subset \operatorname{Cyc}_{\Gamma /L}$. 
\end{defi}

We fix $\sigma \in \operatorname{Facet}(P)$, $\Delta \in T_{\sigma}$ and $v \in V(\Delta)$. 
We pick $\operatorname{cpx}_{\sigma}(v) \in \mathbb{Z}_{> 0}$ and 
$r^{F} _{\sigma, v}, s^{F} _{\sigma, v} \in \mathbb{Z}_{\ge 0}$ for each $F \in \mathcal{H}'' _{\sigma, v}$ 
satisfying the following two conditions (R) and (S):
\begin{enumerate}
\item[(R)] 
Suppose that a function $a: \operatorname{Cyc}_{\Gamma /L}(\sigma) \to \mathbb{Z}_{\ge 0}$ satisfies 
\[
\sum _{q \in \operatorname{Cyc}_{\Gamma /L}(F)} a(q) \cdot w(q) > r^{F} _{\sigma, v}
\]
for any $F \in \mathcal{H}'' _{\sigma, v}$. 
Then, there exists a function $b: \operatorname{Cyc}_{\Gamma /L}(\sigma) \to \mathbb{Z}_{\ge 0}$ 
satisfying 
\begin{itemize}
\item 
$b^{-1}(\mathbb{Z}_{>0}) \subset a^{-1}(\mathbb{Z}_{>0})$, and

\item
$\sum _{q \in \operatorname{Cyc}_{\Gamma /L}(\sigma)} b(q) \cdot \mu(\langle q \rangle) = \operatorname{cpx}_{\sigma}(v) v$. 
\end{itemize}

\item[(S)] 
Suppose that a function $a: \operatorname{Cyc}_{\Gamma /L}(\sigma) \to \mathbb{Z}_{\ge 0}$ satisfies 
\[
\sum _{q \in \operatorname{Cyc}_{\Gamma /L}(F)} a(q) \cdot w(q) > s^{F} _{\sigma, v}
\]
for any $F \in \mathcal{H}'' _{\sigma, v}$. 
Then, there exists a function $b: \operatorname{Cyc}_{\Gamma /L}(\sigma) \to \mathbb{Z}_{\ge 0}$ 
satisfying 
\begin{itemize}
\item 
$b(q) \le a(q)$ holds for any $q \in \operatorname{Cyc}_{\Gamma /L}(\sigma)$, 

\item 
$\sum _{q \in \operatorname{Cyc}_{\Gamma /L}(\sigma)} b(q) \cdot \mu(\langle q \rangle) = \operatorname{cpx}_{\sigma}(v) v$, and 

\item
for any $q \in b^{-1}(\mathbb{Z}_{>0})$, 
there exists $q' \in \operatorname{Cyc}_{\Gamma /L}(\sigma)$ 
such that $b(q') < a(q')$ and $\operatorname{supp}(q) \subset \operatorname{supp}(q')$.
\end{itemize}
\end{enumerate}

\noindent
See Lemma \ref{lem:r} for the existence and a construction of $\operatorname{cpx}_{\sigma}(v)$, 
$r^{F} _{\sigma, v}$ and $s^{F} _{\sigma, v}$. 
These technical conditions will be used in the proof of Theorems \ref{thm:mod} and \ref{thm:gen}. 

\begin{lem}\label{lem:r}
Fix $\sigma \in \operatorname{Facet}(P)$, $\Delta \in T_{\sigma}$ and $v \in V(\Delta)$. 
Then, there exist $\operatorname{cpx}_{\sigma}(v) \in \mathbb{Z}_{> 0}$ and  
$r^{F} _{\sigma, v}, s^{F} _{\sigma, v} \in \mathbb{Z}_{\ge 0}$ for each $F \in \mathcal{H}'' _{\sigma, v}$ satisfying the conditions (R) and (S). 
\end{lem}
\begin{proof}
We set $S_{\sigma, v} := \bigcup _{F \in \mathcal{H}'' _{\sigma, v}} F \subset \operatorname{Im}(\nu) \cap \sigma$. 

We define $\operatorname{cpx}_{\sigma}(v)$ as the minimum positive integer with the following condition. 
\begin{itemize}
\item 
If a subset $G \subset \operatorname{Cyc}_{\Gamma /L}(S_{\sigma, v})$ satisfies $v \in \operatorname{conv}( \nu(G) )$, 
then we have 
\[
\operatorname{cpx}_{\sigma}(v) v \in \sum _{q \in G} \mathbb{Z}_{\ge 0} \cdot \mu ( \langle q \rangle ). 
\]
\end{itemize}
\noindent
Such $\operatorname{cpx}_{\sigma}(v) \in \mathbb{Z}_{>0}$ exists 
since we have $v \in L_{\mathbb{Q}}$ by the assumed condition $(\diamondsuit)_1$ on $T_{\sigma}$, 
and the choice of $G$ has finitely many possibilities. 

Next, for $u \in S_{\sigma, v}$, 
we define $m_{\sigma, v}(u)$ as the minimum positive integer $\ell$ with the following condition. 
\begin{itemize}
\item
Let $G \subset S_{\sigma, v}$ be 
a set of $\mathbb{R}$-linear independent elements satisfying $u \in G$ and $v \in \operatorname{conv}(G)$. 
For $u' \in G$, we define $\operatorname{proj}_{G,u'}(v) \in \mathbb{R}_{\ge 0}$ by the unique expression 
$v = \sum _{u' \in G} \operatorname{proj}_{G,u'}(v) \cdot u'$. 
Then, we require $\ell$ to satisfy $\operatorname{cpx}_{\sigma}(v) \cdot \operatorname{proj}_{G,u}(v) \le \ell$. 
\end{itemize}

For $u \in \operatorname{Im}(\nu)$ and $d \in \mathbb{Z}_{>0}$, we define 
\begin{align*}
\operatorname{Len} (\{ u \}) 
&:= \bigl \{ w(q) \ \big| \ q \in \operatorname{Cyc}_{\Gamma /L} (\{ u \}) \bigr \}, \\
\operatorname{Cyc}_{\Gamma /L} ^d (\{ u \}) 
&:= \bigl \{ q \in \operatorname{Cyc}_{\Gamma /L} (\{ u \}) \ \big| \ w(q) = d \bigr \}, \\
\operatorname{num}(u,d) 
&:= \# \bigl\{ \operatorname{supp}(q) \ \big| \  q \in \operatorname{Cyc}_{\Gamma /L}^d (\{ u \}) \bigr \}. 
\end{align*}
Then, we define $s^{F} _{\sigma, v}$ by 
\[
s^{F} _{\sigma, v} := 
\sum_{u \in \operatorname{Im}(\nu) \cap F} 
\sum_{d \in \operatorname{Len} (\{ u \})}
\bigl( m_{\sigma, v}(u) + d (\operatorname{num}(u,d) - 1) \bigr).
\]
In what follows, we shall show that $\operatorname{cpx}_{\sigma}(v)$ and $s^{F} _{\sigma, v}$ satisfy condition (S). 

Suppose that a function $a: \operatorname{Cyc}_{\Gamma /L}(\sigma) \to \mathbb{Z}_{\ge 0}$ satisfies 
\[
\sum _{q \in \operatorname{Cyc}_{\Gamma /L}(F)} a(q) \cdot w(q) > s^{F} _{\sigma, v}
\]
for each $F \in \mathcal{H}'' _{\sigma, v}$. 
Note that 
\[
\sum _{q \in \operatorname{Cyc}_{\Gamma /L}(F)} a(q) \cdot w(q) = 
\sum_{u \in \operatorname{Im}(\nu) \cap F} 
\sum_{d \in \operatorname{Len} (\{ u \})}
\Biggl( d \cdot \sum _{q \in \operatorname{Cyc}_{\Gamma /L}^{d} (\{ u \})} a(q) \Biggr). 
\]
Therefore, for each $F$, we can pick $u_F \in \operatorname{Im}(\nu) \cap F$ and $d_F \in \operatorname{Len} (\{ u_F \})$ such that 
\[
d_F \cdot \sum _{q \in \operatorname{Cyc}_{\Gamma /L}^{d_F} (\{ u_F \})} a(q) 
> m_{\sigma, v}(u_F) + d_F \bigl( \operatorname{num}(u_F, d_F) -1 \bigr). 
\]
By Lemma \ref{lem:H}, we have
$v \in \operatorname{conv}\bigl( \{ u_F \mid F \in \mathcal{H}'' _{\sigma, v} \} \bigr)$. 
We take a minimal subset $I \subset \mathcal{H}'' _{\sigma, v}$ such that 
\[
v \in \operatorname{conv}\bigl( \{ u_{F} \mid F \in I \} \bigr). 
\]
Then by the minimality of $I$, the elements of $\{ u_{F} \mid F \in I \}$ are $\mathbb{R}$-linear independent. 
Therefore, we may uniquely write 
\[
\operatorname{cpx}_{\sigma}(v) v = \sum_{F \in I} c_F d_F u_{F} 
\]
with $c_F \in \mathbb{R}_{> 0}$. 
By the choice of $\operatorname{cpx}_{\sigma}(v)$, we have $c_F \in \mathbb{Z}_{> 0}$. 
Moreover, by the choice of $m_{\sigma, v}(u_{F})$, we have $c_F d_F \le m_{\sigma, v}(u_{F})$. 
In particular, we have 
\[
c_F \le \frac{m_{\sigma, v}(u_{F})}{d_F} 
< 1 - \operatorname{num}(u_F, d_F) + \sum _{q \in \operatorname{Cyc}_{\Gamma /L}^{d_F} (\{ u_F \})} a(q). 
\]
For each $F$, by the definition of $\operatorname{num}(u_F, d_F)$, we can take $q_1, \ldots , q_{\ell} \in  \operatorname{Cyc}_{\Gamma /L}^{d_F} (\{ u_F \})$ such that
\begin{itemize}
\item 
$\ell \le \operatorname{num}(u_F, d_F)$, 

\item 
$a(q_1), \ldots, a(q_{\ell}) > 0$, and 

\item 
$\bigl\{ \operatorname{supp}(q) \ \big| \  q \in \operatorname{Cyc}_{\Gamma /L}^{d_F} (\{ u_F \}),\ a(q)>0 \bigr \}
= \{ \operatorname{supp}(q_1), \ldots ,  \operatorname{supp}(q_{\ell}) \}$. 
\end{itemize} 
Then we define a function $f' _F: \operatorname{Cyc}_{\Gamma /L}^{d_F} (\{ u_F \}) \to \mathbb{Z}_{\ge 0} $ by 
\[
f'_F (q) := 
\begin{cases}
a(q)-1 & \text{if $q = q_i$ for some $i = 1, \ldots , \ell$,}\\
a(q) & \text{otherwise.}
\end{cases}
\]
Then, we have 
\begin{align*}\tag{i}
\bigl\{ \operatorname{supp}(q) \  & \big| \  q \in \operatorname{Cyc}_{\Gamma /L}^{d_F} (\{ u_F \}),\ f'_F(q)>0 \bigr \} \\
&\subset \{ \operatorname{supp}(q_1), \ldots ,  \operatorname{supp}(q_{\ell})  \} \\
&= \bigl\{ \operatorname{supp}(q) \ \big| \  q \in \operatorname{Cyc}_{\Gamma /L}^{d_F} (\{ u_F \}),\ a(q)>f'_F(q) \bigr \}. 
\end{align*}
Since we have
\begin{align*}
c_F 
& \le 
- \operatorname{num}(u_F, d_F) + \sum _{q \in \operatorname{Cyc}_{\Gamma /L}^{d_F} (\{ u_F \})} a(q) \\
& \le 
- \ell + \sum _{q \in \operatorname{Cyc}_{\Gamma /L}^{d_F} (\{ u_F \})} a(q) \\
& =
\sum _{q \in \operatorname{Cyc}_{\Gamma /L}^{d_F} (\{ u_F \})} f' _F(q), 
\end{align*}
we can take a function $f_F : \operatorname{Cyc}_{\Gamma /L}^{d_F} (\{ u_F \}) \to \mathbb{Z}_{\ge 0}$ satisfying the following two conditions:
\begin{itemize}
\item[(ii)] 
$\sum _{q \in \operatorname{Cyc}_{\Gamma /L}^{d_F} (\{ u_F \})} f_F(q) = c_F$. 

\item[(iii)] 
$f_F (q) \le f'_F(q)$ holds for any $q \in \operatorname{Cyc}_{\Gamma /L}^{d_F} (\{ u_F \})$. 
\end{itemize}
By (i) and (iii), we have  
\begin{align*}\tag{iv}
\bigl\{ \operatorname{supp}(q) \ &\big| \  q \in \operatorname{Cyc}_{\Gamma /L}^{d_F} (\{ u_F \}),\ f_F(q)>0 \bigr \} \\
&\subset \bigl\{ \operatorname{supp}(q) \ \big| \  q \in \operatorname{Cyc}_{\Gamma /L}^{d_F} (\{ u_F \}),\ a(q)>f_F(q) \bigr \}. 
\end{align*}
We define a function $b : \operatorname{Cyc}_{\Gamma /L}(\sigma) \to \mathbb{Z}_{\ge 0}$ by 
\[
b(q) := 
\begin{cases}
f_F(q) & \text{if $q \in \operatorname{Cyc}_{\Gamma /L}^{d_F} (\{ u_F \})$ for some $F \in I$,}\\
0 & \text{otherwise.}
\end{cases}
\]
Then, by (iii) and (iv), the function $b$ satisfies the following two conditions:
\begin{itemize}
\item 
$b(q) \le a(q)$ holds for any $q \in \operatorname{Cyc}_{\Gamma /L}(\sigma)$. 

\item 
For any $q \in b^{-1}(\mathbb{Z}_{>0})$, 
there exists $q' \in \operatorname{Cyc}_{\Gamma /L}(\sigma)$ 
such that $b(q') < a(q')$ and $\operatorname{supp}(q) = \operatorname{supp}(q')$ 
(in particular, $\operatorname{supp}(q) \subset \operatorname{supp}(q')$).
\end{itemize}
By (ii), we have 
\begin{align*}
\sum _{q \in \operatorname{Cyc}_{\Gamma /L}(\sigma)} b(q) \cdot \mu (\langle q \rangle) 
&= \sum _{F \in I} \sum _{q \in \operatorname{Cyc}_{\Gamma /L}^{d_F} (\{ u_F \})} f_F(q) \cdot \mu (\langle q \rangle) \\
&= \sum _{F \in I} \sum _{q \in \operatorname{Cyc}_{\Gamma /L}^{d_F} (\{ u_F \})} f_F(q) d_F u_F \\
&= \sum_{F \in I} c_F d_F u_{F} \\
&= \operatorname{cpx}_{\sigma}(v) v. 
\end{align*}
Therefore, this $b$ satisfies the condition (S). 

We define $r_{\sigma, v}^F := 0$ for any $F \in \mathcal{H}''_{\sigma , v}$. 
Then by the same argument as for $s_{\sigma, v}^F$, we can show that 
$\operatorname{cpx}_{\sigma}(v)$ and $r_{\sigma, v}^F$'s satisfy the condition (R).
\end{proof}

\subsubsection{$a^{F} _{\Delta, v} (v')$ and $h^F_{\Delta, v}$}\label{subsubsection:ah}
\begin{defi}\label{defi:ah}
For a subset $F \subset L_{\mathbb{R}}$ that consists of $\mathbb{R}$-linear independent $n$ elements, 
we define 
\[
H(F) := \left \{ \sum _{x \in F} a_x x \ \middle | \ \text{$a_x \in \mathbb{R}$ with $\sum_{x \in F} a_x \le 1$} \right \}. 
\]
This is the closed half-space $H$ of $L_{\mathbb{R}}$ satisfying $0 \in H$ and $F \subset \partial H := H \setminus \operatorname{relint}(H)$. 
\end{defi}

We fix $\sigma \in \operatorname{Facet}(P)$, $\Delta \in T_{\sigma}$, $v \in V(\Delta)$ and $F \in \mathcal{H}'' _{\sigma, v}$. 
For each $v' \in V(\Delta)$, 
we pick $a^{F} _{\Delta, v} (v') \in (0,1]$ with the following conditions: 
\begin{itemize}
\item[(A1)] 
$a^{F} _{\Delta, v} (v) \in (0,1)$. 

\item[(A2)]
$\operatorname{Im}(\nu) \setminus F \subset H \bigl( \bigl \{ a^{F} _{\Delta, v} (v') v' \ \big | \  v' \in V(\Delta) \bigr \} \bigr)$. 
\end{itemize}
See Lemma \ref{lem:a} for the existence of such $a^{F} _{\Delta, v} (v')$'s. 

We set
\[\tag{H}
h^F_{\Delta, v} := \max \bigl \{ \alpha \in \mathbb{R}_{\ge 0} \ \big | \ 
\alpha P \subset 
H \bigl( \bigl \{ a^{F} _{\Delta, v} (v') v' \mid v' \in V(\Delta) \bigr \} \bigr) \bigr \}. 
\]
We have $0 < h^F_{\Delta, v} \le a^{F} _{\Delta, v} (v) < 1$ since $a^{F} _{\Delta, v} (v') > 0$ for each $v' \in V(\Delta)$. 

Lemma \ref{lem:sep} below will be used in the proof of Lemma \ref{lem:a}. 
We prove Lemma \ref{lem:sep} following the proof of the hyperplane separation theorem (cf.\ \cite{HW}*{Theorem 1.17}). 
In the proof below, the finiteness of $F'$ and the assumption $(\diamondsuit) _2$ on 
the triangulation $T_{\sigma}$ will be important. 
\begin{lem}\label{lem:sep}
Fix $\sigma \in \operatorname{Facet}(P)$, $\Delta \in T_{\sigma}$ and $v \in V(\Delta)$. 
Let $F' \subset \operatorname{Im}(\nu) \cap \sigma$ be a subset satisfying $v \not \in \operatorname{conv} (F')$. 
Then, there exists a closed halfspace $H'$ of $\overline{\sigma}$ such that 
\[
v \not \in H', \quad \operatorname{relint}(\Delta) \cap H' = \emptyset, \quad F' \subset H'. 
\]
\end{lem}
\begin{proof}
Note that we have $\operatorname{relint}(\Delta) \cap \operatorname{conv} (F') = \emptyset$. 
Otherwise, we have $\Delta \subset \operatorname{conv} (F')$ by the assumption $(\diamondsuit)_2$, 
and it contradicts $v \not \in \operatorname{conv} (F')$. 
If $\Delta \cap \operatorname{conv} (F') = \emptyset$, then the assetion follows from 
the hyperplane separation theorem (cf.\ \cite{HW}*{Theorem 1.17}). 
In what follows, we assume that $\Delta \cap \operatorname{conv} (F') \not = \emptyset$. 

We identify $\overline{\sigma} = \mathbb{R}^{n-1}$, 
and $\langle \cdot , \cdot \rangle$ denotes the standard inner product on $\mathbb{R}^{n-1}$. 
Let $v_1 = v, v_2, \ldots, v_n$ be the vertices of $\Delta$, and 
let $x_1, \ldots , x_{\ell}$ be the elements of $F'$. 
For $1 \le i \le n$ and $1 \le j \le \ell$, we define $y_{ij} := x_j - v_i$. 
We set 
\[
Q := \operatorname{conv}\bigl( \{ y_{ij} \mid 1 \le i \le n, \ 1 \le j \le \ell \} \bigr). 
\]
Since $\Delta \cap \operatorname{conv} (F') \not = \emptyset$, we have $0 \in Q$. 
Furthermore, since $\operatorname{relint}(\Delta) \cap \operatorname{conv} (F') = \emptyset$, 
we have $0 \not \in \operatorname{relint}(Q)$. 
Let $G$ be the minimal face of $Q$ satisfying $0 \in G$. 
Then, there exists $u \in \mathbb{R}^{n-1} \setminus \{ 0 \}$ such that 
\[
Q \subset \left\{ z \in \mathbb{R}^{n-1} \ \middle | \  \langle z, u \rangle \ge 0 \right \}, \quad 
\left\{ z \in \mathbb{R}^{n-1} \ \middle | \  \langle z, u \rangle = 0 \right \} \cap Q = G. 
\]
We set $\alpha := \min _{1 \le j \le \ell} \langle x_j, u \rangle$. 
Then we have
\[\tag{1}
F' \subset \left\{ z \in \mathbb{R}^{n-1} \ \middle | \  \langle z, u \rangle \ge \alpha \right \}. 
\]
Since $Q \subset \left\{ z \in \mathbb{R}^{n-1} \ \middle | \  \langle z, u \rangle \ge 0 \right \}$, 
we have $\langle x_j, u \rangle \ge \langle v_i, u \rangle$ for any $1 \le i \le n$ and $1 \le j \le \ell$. 
Therefore, we have $\max _{1 \le i \le n} \langle v_i, u \rangle \le \alpha$, and hence, we have
\[\tag{2}
\Delta \subset \left\{ z \in \mathbb{R}^{n-1} \ \middle | \  \langle z, u \rangle \le \alpha \right \}.
\]

In what follows, we prove $\langle v, u \rangle < \alpha$. 
We define 
\[
I := \bigl\{ (i,j) \in \{ 1, \ldots , n\} \times \{ 1 , \ldots, \ell \} \ \big| \  y_{ij} \in G \bigr\}. 
\]
By the minimality of $G$, we have $0 \in \operatorname{relint}(G)$. 
Therefore, there exist $c_{ij} \in \mathbb{R}_{>0}$ for $(i,j) \in I$ such that
\[
\sum _{(i,j) \in I} c_{ij} y_{ij} = 0, \quad \sum _{(i,j) \in I} c_{ij} = 1. 
\]
Then, we have 
\[\tag{3}
\sum _{(i,j) \in I} c_{ij} v_i = \sum _{(i,j) \in I} c_{ij} x_j. 
\]
Let $\Delta '$ be the face of $\Delta$ such that $V(\Delta ') = \{ v_i \mid (i,j) \in I \}$. 
Then, by (3), we have $\operatorname{relint}(\Delta ') \cap \operatorname{conv}(F') \not = \emptyset$. 
Therefore, by $(\diamondsuit)_2$, we have $\Delta' \subset \operatorname{conv}(F')$. 
Since $v \not \in \operatorname{conv}(F')$, we can conclude that $v_1 = v \not \in V(\Delta ')$. 
Therefore, by the definition of $I$, we have $y_{11}, \ldots , y_{1\ell} \not \in G$, and hence, 
$\langle y_{1j}, u \rangle > 0$ for any $1 \le j \le \ell$. 
Thus, for $j'$ satisfying $\langle x_{j'}, u \rangle = \alpha$, we have
\[\tag{4}
\langle v_1, u \rangle = \langle x_{j'}, u \rangle - \langle y_{1j'}, u \rangle < \alpha. 
\]

By (1), (2) and (4), we conclude that
\[
H' := \left\{ z \in \mathbb{R}^{n-1} \ \middle | \  \langle z, u \rangle \ge \alpha \right \}
\]
satisfies the desired conditions. 
\end{proof}

\begin{lem}\label{lem:a}
Fix $\sigma \in \operatorname{Facet}(P)$, $\Delta \in T_{\sigma}$, $v \in V(\Delta)$ and $F \in \mathcal{H}''_{\sigma , v}$. 
Then, for each $v' \in V(\Delta)$, there exists $a^{F} _{\Delta, v} (v') \in (0,1]$ 
such that 
\begin{itemize}
\item[(A1)] $a^{F} _{\Delta, v} (v) \in (0,1)$, and

\item[(A2)]
$\operatorname{Im}(\nu) \setminus F \subset H \bigl( \bigl \{ a^{F} _{\Delta, v} (v') v' \ \big | \  v' \in V(\Delta) \bigr \} \bigr)$. 
\end{itemize}
\end{lem}
\begin{proof}
We set $F' _{\sigma} := ( \operatorname{Im}(\nu) \cap \sigma ) \setminus F$. 
By the definition of $\mathcal{H}''_{\sigma , v}$, we have $v \not \in \operatorname{conv} (F' _{\sigma})$. 
Therefore, by Lemma \ref{lem:sep}, 
there exists a closed halfspace $H'$ of $\overline{\sigma}$ such that 
\[
v \not \in H', \quad
\operatorname{relint}(\Delta) \subset \overline{\sigma} \setminus H', \quad
F'_{\sigma} \subset H'. 
\]
Let $\ell := \partial (H') = H' \setminus \operatorname{relint}(H')$. 

\begin{figure}[htbp]
\centering
\includegraphics[width=7cm]{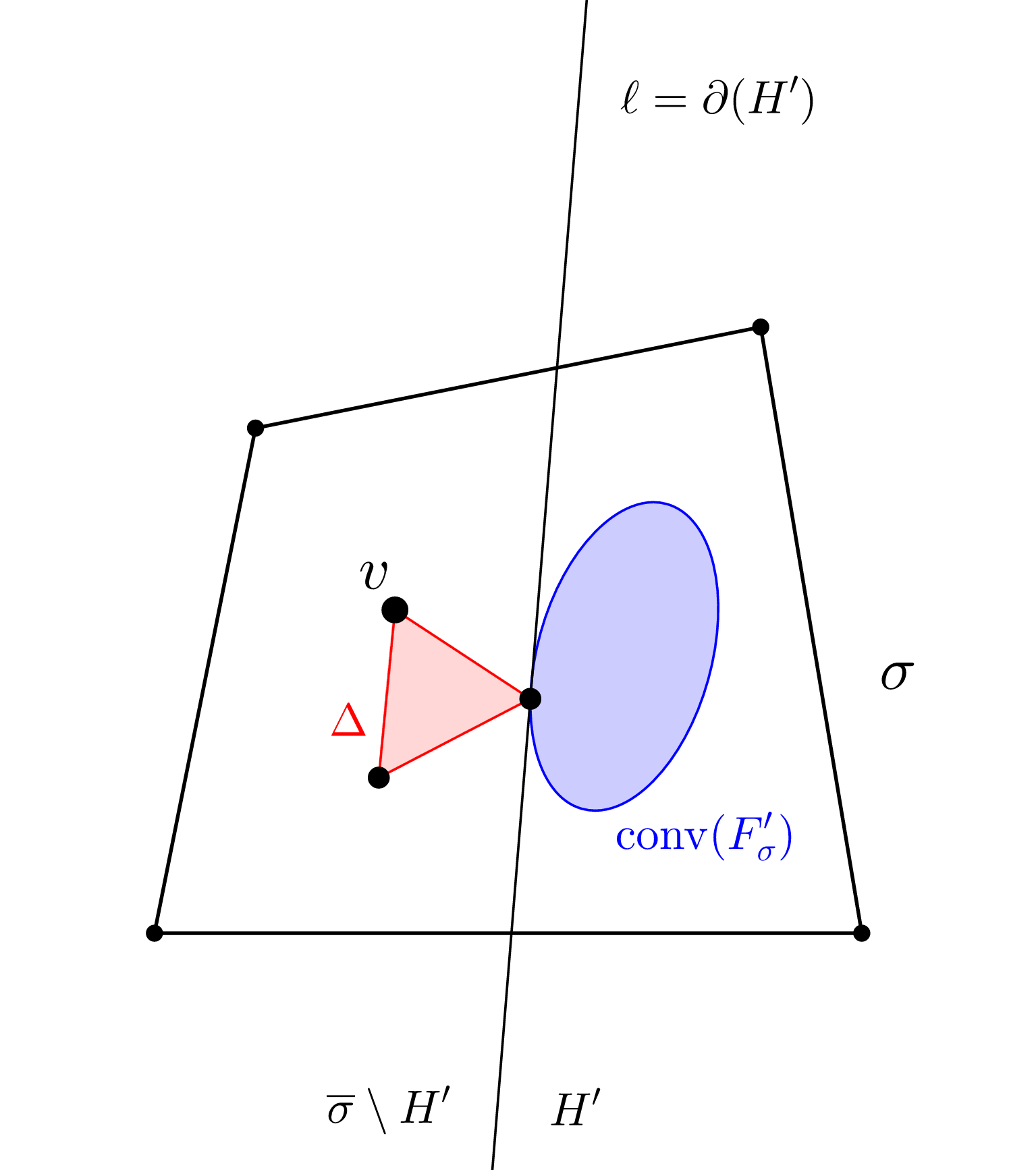}
\caption{$H'$ and $\ell$.}
\label{fig:H'}
\end{figure}

For $a \in \mathbb{R}_{>0}$, let $H(\ell, a v)$ denote
the closed halfspace of $L_{\mathbb{R}}$ uniquely determined by the following conditions: 
\[
0 \in H(\ell, a v), \quad 
\ell \subset \partial (H(\ell, av)), \quad 
av \in \partial (H(\ell, av)), 
\]
where $\partial (H(\ell, av)) := H(\ell, av) \setminus \operatorname{relint}(H(\ell, av))$. 
Note that $\partial (H(\ell, v)) = \overline{\sigma}$. 
Therefore, we have  
$\operatorname{Im}(\nu) \setminus \sigma \subset \operatorname{relint}(H(\ell, v))$. 
Since $\operatorname{Im}(\nu)$ is a finite set, 
there exists $\epsilon \in (0,1)$ such that 
\[
\operatorname{Im}(\nu) \setminus \sigma \subset H(\ell, (1-\epsilon) v). 
\]
For each $v' \in V(\Delta)$, we take $a(v') \in \mathbb{R}_{>0}$ so that 
\[
H(\ell, (1-\epsilon) v) = H \bigl( \{ a(v') v' \mid v' \in V(\Delta) \} \bigr). 
\]
This $a(v')$ can also be determined by $a(v')v' \in \partial (H(\ell, (1- \epsilon) v) )$. 
We have $a(v) = 1 - \epsilon \in (0,1)$. 
Note that on $\overline{\sigma}$, $\Delta$ is on the same side as $v$ with respect to the hyperplane $\ell$. 
Therefore, we have $a(v') \le 1$ for any $v' \in V(\Delta)$. 

Since $v \not \in H'$ and $1 - \epsilon < 1$, we have 
\[
H' = H(\ell, (1 - \epsilon ) v) \cap \overline{\sigma}. 
\]
Therefore, we have 
\begin{align*}
\operatorname{Im}(\nu) \setminus F 
&= (\operatorname{Im}(\nu) \setminus \sigma) \cup F' _{\sigma} \\
&\subset (\operatorname{Im}(\nu) \setminus \sigma) \cup H' \\
&\subset H(\ell, (1- \epsilon) v) = H \bigl( \{ a(v') v' \mid v' \in V(\Delta) \} \bigr). 
\end{align*}
Therefore, $a^{F} _{\Delta, v} (v') := a(v') \in (0,1]$ satisfies the desired conditions. 
\end{proof}

\subsubsection{$C'_2$, $\beta$ and $\operatorname{cpx}_{\Gamma}$}\label{subsubsection:beta}
We define 
\[
W := \max \{ w(e) \mid e \in E_{\Gamma} \}. 
\]
We have $W  < \infty$ by the definition of periodic graphs. 
We pick $C'_2 \in \mathbb{R}_{\ge 0}$ satisfying $C_2 \le C'_2$ (cf.\ Proposition \ref{prop:C2}(2)). 

Let $\sigma \in \operatorname{Facet}(P)$, $\Delta \in T_{\sigma}$ and $v \in V(\Delta)$. 
For $F \in \mathcal{H}'' _{\sigma, v}$, we set $\alpha ^{F} _{\Delta, v}, \alpha ^{\prime F} _{\Delta, v} \in \mathbb{R}_{\ge 0}$ by
\begin{align*}
\alpha ^{F} _{\Delta, v} 
:= \frac{a^{F} _{\Delta, v} (v)}{1-a^{F} _{\Delta, v} (v)}
\left( \frac{C_1}{h^F_{\Delta, v}} + C'_2 + \frac{1-h^F_{\Delta, v}}{h^F_{\Delta, v}} \bigl( r^{F} _{\sigma, v} + W ( \#(V_{\Gamma}/L) -1) \bigr) \right), \\
\alpha ^{\prime F} _{\Delta, v} 
:= \frac{a^{F} _{\Delta, v} (v)}{1-a^{F} _{\Delta, v} (v)}
\left( \frac{C_1}{h^F_{\Delta, v}} + C'_2 + \frac{1-h^F_{\Delta, v}}{h^F_{\Delta, v}} \bigl( s^{F} _{\sigma, v} + W(\#(V_{\Gamma}/L) -1) \bigr) \right). 
\end{align*}
Then, we define $\alpha _{\Delta, v}, \alpha ' _{\Delta, v} \in \mathbb{R}_{\ge 0}$ as follows. 
\[
\alpha _{\Delta, v} := \max _{F \in \mathcal{H}'' _{\sigma, v}} \alpha ^{F} _{\Delta, v}, \qquad
\alpha ' _{\Delta, v} := \max _{F \in \mathcal{H}'' _{\sigma, v}} \alpha ^{\prime F} _{\Delta, v}. 
\]
We define $\beta _{\Delta, v}, \beta' _{\Delta, v} \in \mathbb{R}_{\ge 0}$ by
\[
\beta_{\Delta, v} := \max \bigl\{ \alpha _{\Delta, v}, \alpha ' _{\Delta, v} - \operatorname{cpx}_{\sigma}(v) \bigr\}, \quad 
\beta' _{\Delta, v} := \beta_{\Delta, v} + \operatorname{cpx}_{\sigma}(v). 
\]
Furthermore, for $\Delta ' \in \operatorname{Face}(\Delta)$, we define
\[
\beta _{\Delta, \Delta'} := \sum_{v \in V(\Delta ')} \beta _{\Delta, v}.
\]
We also define $\beta \in \mathbb{R}_{\ge 0}$ by
\[
\beta := C' _2 + \max \bigl\{ \beta _{\Delta, \Delta} \ \big | \ \sigma \in \operatorname{Facet}(P),\ \Delta \in T_{\sigma} \bigr \}. 
\]

Let $\sigma \in \operatorname{Facet}(P)$, $\Delta \in T_{\sigma}$ and $\Delta ' \in \operatorname{Face}(\Delta)$. 
Then we define $\operatorname{cpx}_{\sigma}(\Delta ') \in \mathbb{Z}_{>0}$ by 
\[
\operatorname{cpx}_{\sigma}(\Delta ') := 
\operatorname{LCM}\{ \operatorname{cpx}_{\sigma}(v) \mid v \in V(\Delta ') \}. 
\]
We also define $\operatorname{cpx}_{\Gamma} \in \mathbb{Z}_{>0}$ by
\[
\operatorname{cpx}_{\Gamma} :=
\operatorname{LCM} \bigl\{ \operatorname{cpx}_{\sigma}(\Delta) \ \big | \ \sigma \in \operatorname{Facet}(P),\ \Delta \in T_{\sigma} \bigr \}. 
\]

\begin{rmk}\label{rmk:ineq}
\begin{enumerate}
\item 
The notations $\beta _{\Delta, \Delta '}$ and $\operatorname{cpx}_{\sigma}(\Delta ')$ for $\Delta ' \not = \Delta$ 
were not used to define $\beta$ and $\operatorname{cpx}_{\Gamma}$, but they will be used in Theorem \ref{thm:main}. 
The notation $\beta ' _{\Delta, v}$ will be used in Definition \ref{defi:M}. 

\item
When we use the invariants $\operatorname{cpx}_{\sigma}(v)$, $r_{\sigma, v}^F$ and $s_{\sigma, v}^F$ constructed in the proof of Lemma \ref{lem:r}, we have 
\[
\beta_{\Delta, v} = \alpha ' _{\Delta, v} - \operatorname{cpx}_{\sigma}(v)
\]
by Lemma \ref{lem:rsineq} below. 
\end{enumerate}
\end{rmk}

\begin{lem}\label{lem:rsineq}
Let $\sigma \in \operatorname{Facet}(P)$, $\Delta \in T_{\sigma}$, $v \in V(\Delta)$ and $F \in \mathcal{H}'' _{\sigma, v}$. 
Suppose that $\operatorname{cpx}_{\sigma}(v)$, $r_{\sigma, v}^F$ and $s_{\sigma, v}^F$ are the invariants constructed in the proof of Lemma \ref{lem:r}. 
Then, we have 
$\alpha ^{\prime F} _{\Delta, v} - \alpha ^{F} _{\Delta, v} \ge \operatorname{cpx}_{\sigma}(v)$, 
and hence, $\beta_{\Delta, v} = \alpha ' _{\Delta, v} - \operatorname{cpx}_{\sigma}(v)$. 
\end{lem}
\begin{proof}
Note that we set $r_{\sigma, v}^F = 0$ in the construction. Therefore, it is sufficient to show the inequality
\[\tag{i}
\frac{a_{\Delta,v}^F(v)}{1-a^F _{\Delta,v}(v)} \cdot \frac{1 - h^F _{\Delta,v}}{h^F _{\Delta,v}} 
\cdot s_{\sigma, v}^F  \ge \operatorname{cpx}_{\sigma}(v). 
\]

When $v \in F$, we have $s_{\sigma, v}^F \ge m_{\sigma,v}(v) = \operatorname{cpx}_{\sigma}(v)$ by the construction. 
Since $h_{\Delta, v} ^F \le a_{\Delta, v} ^F (v) <1$, we obtain the desired inequality (i). 

In what follows, we assume $v \not\in F$. 
Take any $v_F \in F$. 
For each $F' \in \mathcal{H}'' _{\sigma, v} \setminus \{ F \}$, we take any $v_{F'} \in F' \setminus F$. 
Then by Lemma \ref{lem:H}, we have 
\[
v \in \operatorname{conv} \bigl( \{ v_{F'} \mid F' \in \mathcal{H}'' _{\sigma, v} \}  \bigr). 
\]
Take a minimal subset $I \subset \mathcal{H}'' _{\sigma, v}$ satisfying $v \in \operatorname{conv} \bigl( \{ v_{F'} \mid F' \in I \}  \bigr)$. 
Then, by the minimality of $I$, the elements of $\{ v_{F'} \mid F' \in I \}$ are $\mathbb{R}$-linear independent. 
Furthermore, by the definition of $\mathcal{H}' _{\sigma, v}$, we conclude $F \in I$ since $v_{F'} \not \in F$ holds for each $F' \in I \setminus \{ F \}$. 

Since  $v \in \operatorname{conv} \bigl( \{ v_{F'} \mid F' \in I \}  \bigr)$, 
there exist $w \in \operatorname{conv} \bigl( \bigl\{ v_{F'} \ \big | \  F' \in I \setminus \{ F \} \bigr\}  \bigr)$ and $t \in [0,1]$ 
such that 
\[
v = t v_F + (1-t) w. 
\]
We have $t \ne 1$ by the assumption $v \not\in F$. 
We also have $t \ne 0$ by the minimality of $I$. 
Therefore, we conclude that $t \in (0,1)$. 
Furthermore, by the construction of $s_{\sigma, v}^F$ in the proof of Lemma \ref{lem:r}, we have 
\[\tag{ii}
s_{\sigma, v}^F \ge m_{\sigma,v}(v_F) \ge t\operatorname{cpx}_{\sigma}(v). 
\]

Let $H := H \bigl( \bigl \{ a^{F} _{\Delta, v} (v') v' \ \big| \  v' \in V(\Delta) \bigr \} \bigr)$ be the half-space defined in Definition \ref{defi:ah}. 
By the conditions (A1) and (A2) on $a^{F} _{\Delta, v} (v')$'s in Subsection \ref{subsubsection:ah}, we have 
\[
v \not \in H, \qquad w \in \operatorname{conv} \bigl( (\operatorname{Im}(\nu) \cap \sigma ) \setminus F \bigr) \subset H. 
\]
Therefore, $\partial H$ intersects with the line through $v$ and $w$.
Let $x$ be the intersection point.  
Then the points $v_F$, $v$, $x$, and $w$ are on the line in this order. 
Let $b \in \mathbb{R}$ be the unique real number such that $bv_F \in \partial H$. 
Then, we have $h^F _{\Delta,v} \le b < 1$ since $v_F \not\in H$ and $h^F _{\Delta,v} v_F \in H$.
Therefore, we have
\begin{align*}\tag{iii}
	\frac{a_{\Delta,v}^F(v)}{1-a^F_{\Delta,v}(v)} \cdot \frac{1-h^F_{\Delta,v}}{h^F_{\Delta,v}}
	&\ge \frac{\left \lvert a_{\Delta,v}^F(v)v-0 \right \rvert}{\left \lvert v-a^F_{\Delta,v}(v)v \right \rvert} \cdot 
		\frac{\left \lvert v_F-bv_F \right \rvert}{\lvert bv_F-0 \rvert } \\
	&=\frac{\lvert v_F-x \rvert}{\lvert x-v \rvert} \\
	&\ge \frac{\lvert v_F-w \rvert}{\lvert w-v \rvert} = \frac{1}{t}. 
\end{align*}
Here, we used Menelaus's theorem for the first equality. 
From (ii) and (iii), we get the desired inequality (i). 
\end{proof}

\subsection{Invariants for $2$-dimensional periodic graphs}\label{subsection:2dim}
In general, it is not easy to find the invariants in Subsection \ref{subsection:inv}. 
However, it is easy for $2$-dimensional periodic graphs. 
In what follows, we assume $n = 2$.

Let $\sigma \in \operatorname{Facet}(P)$, and let $\{ v_1, \ldots , v_{\ell} \} = \operatorname{Im}(\nu) \cap \sigma$. 
Since $\dim \sigma = 1$, the points $v_1, \ldots , v_{\ell}$ determine the triangulation $T_{\sigma}$ of $\sigma$ satisfying 
$\bigcup _{\Delta \in T_{\sigma}} V(\Delta) = \{ v_1, \ldots , v_{\ell} \}$. 
Furthremore, since $\dim \sigma = 1$, this $T_{\sigma}$ satisfies the required conditions $(\diamondsuit)_1$ and $(\diamondsuit)_2$. 

Since the triangulation $T_{\sigma}$ is concretely given, 
$\operatorname{cpx}_{\sigma} (v)$ and $s_{\sigma, v} ^F$ can be easily computed according to the construction in the proof of Lemma \ref{lem:r}. 
We shall explain the construction of the invariant $a_{\sigma, v}^F$ below. 

Suppose that the points $v_1, \ldots , v_{\ell}$ are on $\sigma$ in this order as in Figure \ref{fig:2dim}. 
For $1 \le i \le \ell -1$, let $\Delta _i \in T_{\sigma}$ denote the $1$-simplex determined by $V(\Delta _i) = \{ v_i, v_{i+1} \}$. 
For $1 \le i \le \ell$, we set 
\[
F_i^- := \{ v_1, \ldots , v_i \}, \quad 
F_i^+ := \{ v_i, \ldots , v_{\ell} \}. 
\]
Then we have 
\[
\mathcal{H}''_{\sigma, v_i} = 
\begin{cases}
\bigl \{ F_i ^-  \bigr\} = \{ \{ v_1 \} \} & \text{(if $i = 1$)} \\
\bigl \{ F_i^-, F_i^+  \bigr \} & \text{(if $2 \le i \le \ell -1$)} \\
\bigl \{ F_i ^+ \bigr \} = \{ \{ v_{\ell} \} \} & \text{(if $i = \ell$)} 
\end{cases}. 
\]
By symmetry, it is sufficient to see the construction of 
$a_{\Delta _{i-1}, v_i} ^{F^+ _i}(v_{i-1})$ and $a_{\Delta _{i-1}, v_i} ^{F^+ _i}(v_{i})$ for $2 \le i \le \ell$,
and $a_{\Delta _{i}, v_i} ^{F^+ _i}(v_i)$ and $a_{\Delta _{i}, v_i} ^{F^+ _i}(v_{i+1})$ for $2 \le i \le \ell -1$. 
We fix $2 \le i \le \ell$. 
First, we define $t_1 \in (0,1)$ by
\[
t_1 := \max \Bigl \{ \epsilon,  \min \bigl\{ \alpha \in \mathbb{R}_{\ge 0} \ \big| \  
\operatorname{Im}(\nu) \setminus F^+ _i \subset H(v_{i-1}, \alpha v_i) \bigr \} \Bigr \}, 
\]
where $\epsilon$ is any real number satisfying $\epsilon \in (0,1)$. 
Then, $a^{F^+ _i} _{\Delta _{i-1}, v_i} (v_i) := t_1$ and $a^{F^+ _i} _{\Delta _{i-1}, v_i} (v_{i-1}) := 1$ satisfy 
the conditions (A1) and (A2) in Subsection \ref{subsubsection:ah} for $\Delta _{i-1}$ and $v_{i}$. 
When $i \le \ell -1$, we define $t_2 \in (0,1)$ to satisfy 
\[
H ( v_{i-1}, t_1 v_i ) = H ( t_1 v_{i}, t_2 v_{i+1} \bigr). 
\]
Then, $a^{F^+ _i} _{\Delta _{i}, v_i} (v_i) := t_1$ and $a^{F^+ _i} _{\Delta _{i}, v_i} (v_{i+1}) := t_2$ satisfy 
the conditions (A1) and (A2) in Subsection \ref{subsubsection:ah} for $\Delta _{i}$ and $v_{i}$. 

\begin{figure}[htbp]
\centering
\includegraphics[width=12cm]{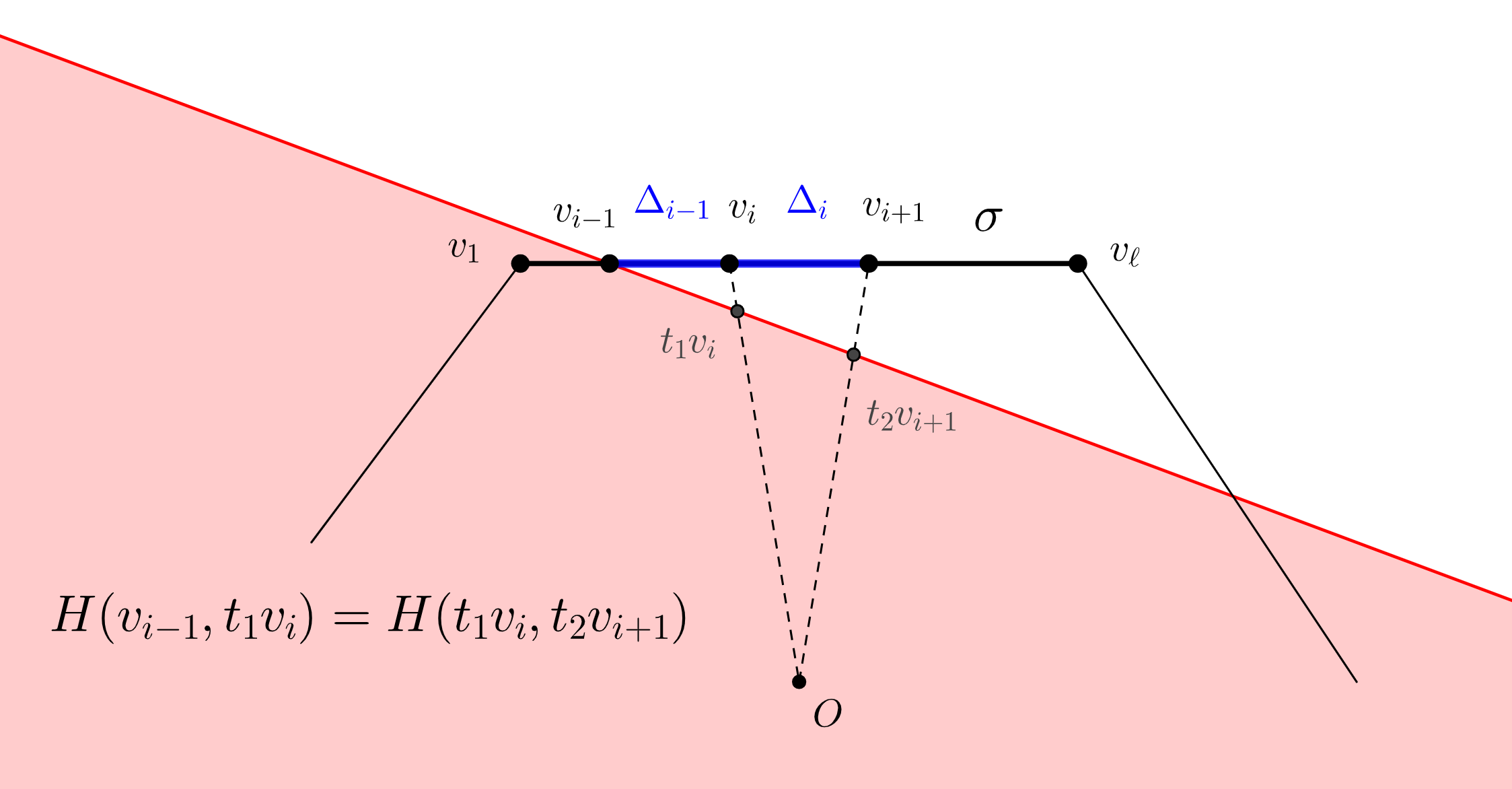}
\caption{$2$-dimensional case.}
\label{fig:2dim}
\end{figure}

\begin{ex}\label{ex:inv}
We shall calculate the invariants in Subsection \ref{subsection:inv} for the Wakatsuki graph and the start point $x_0 = v'_2$ in Figure \ref{fig:WG1}. 
See \cite{IN}*{Example 2.6} for the detailed definition of the Wakatsuki graph. 

\begin{figure}[htbp]
\begin{tabular}{cc}
\begin{minipage}{0.5\hsize}
\centering
\includegraphics[width=7.7cm]{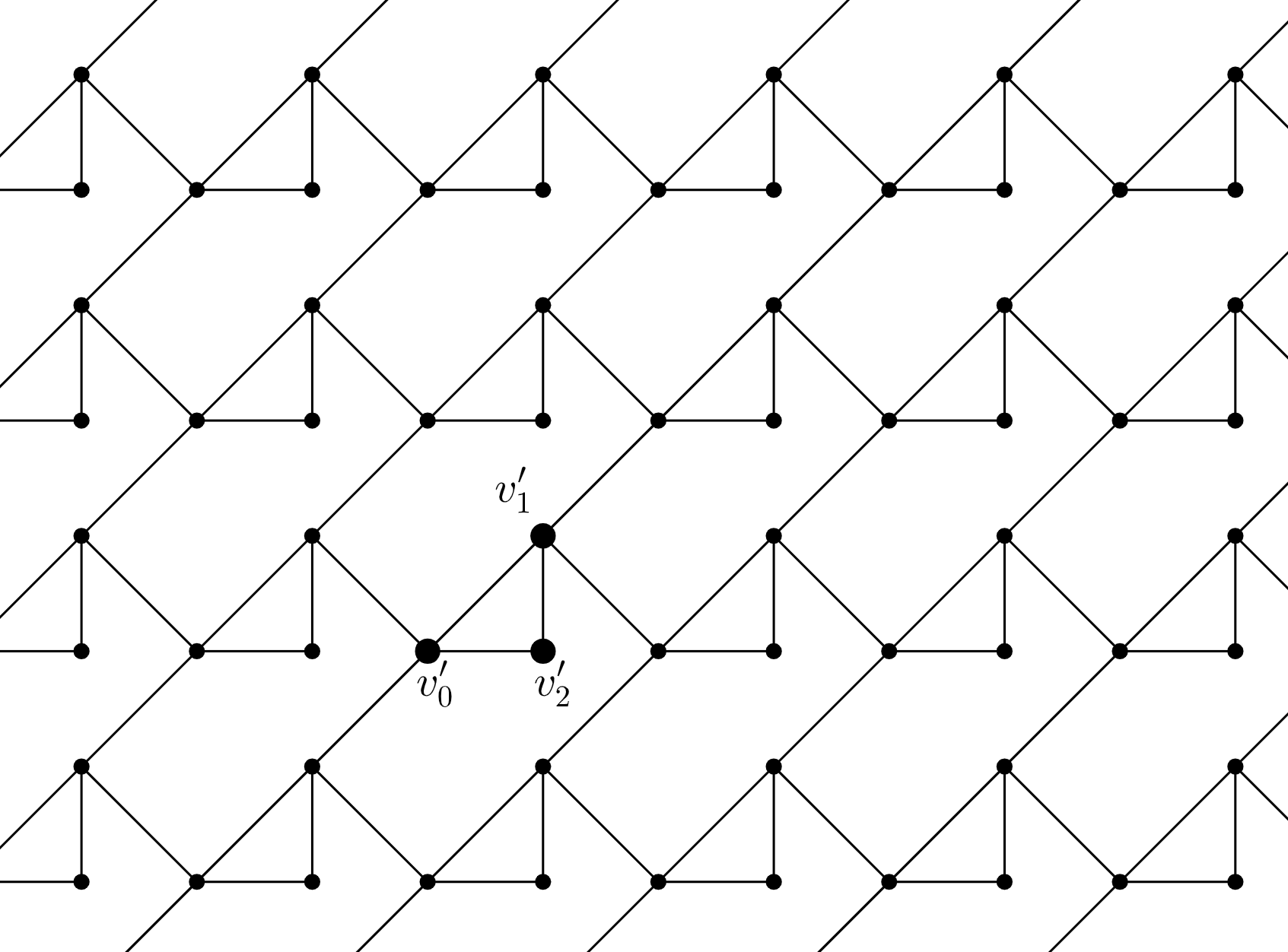}
\caption{The Wakatsuki graph $\Gamma$.}
\label{fig:WG1}
\end{minipage}&
\begin{minipage}{0.5\hsize}
\centering
\includegraphics[width=5.2cm]{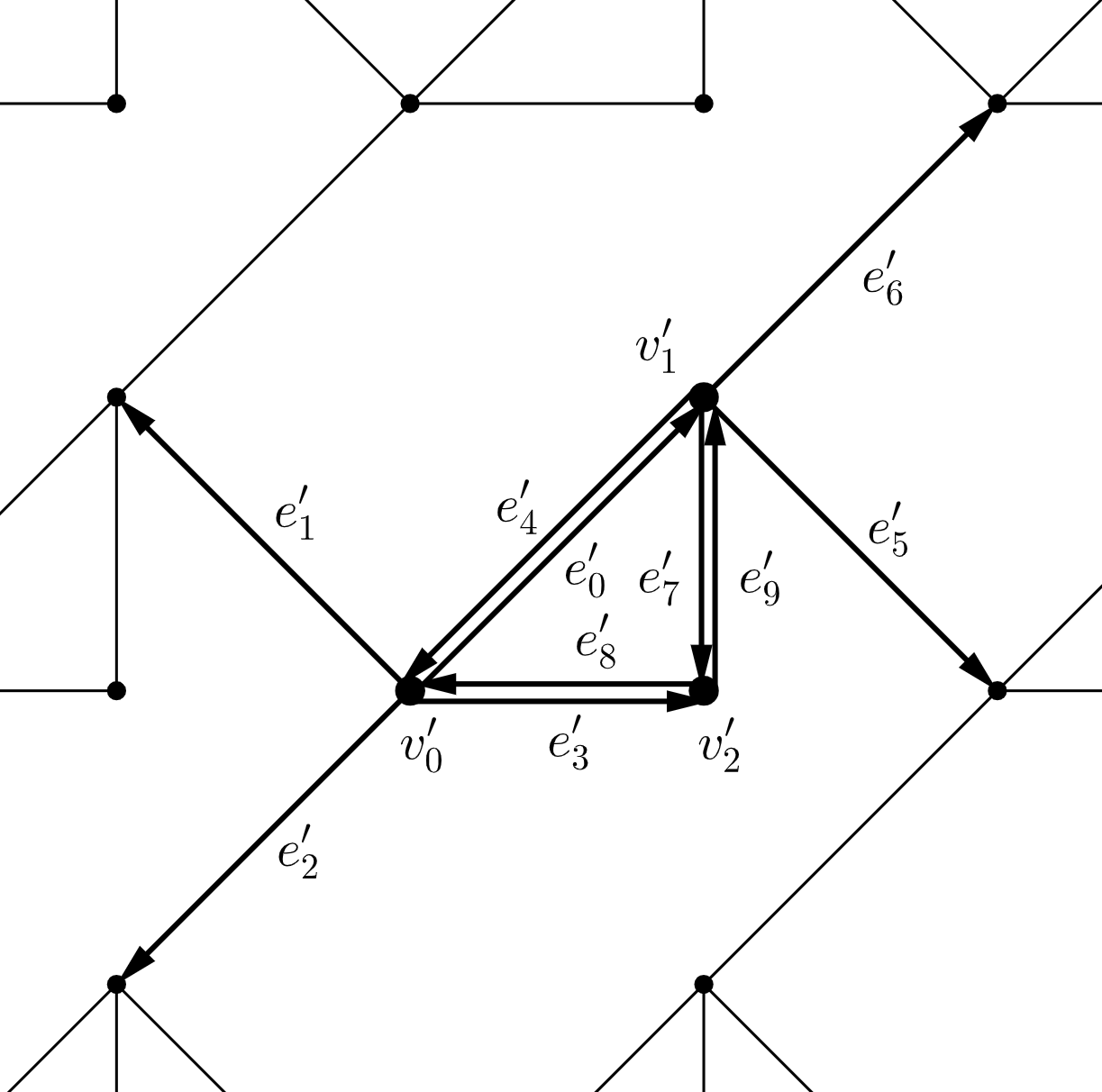}
\caption{$e'_0, \ldots, e'_9$.}
\label{fig:WG2}
\end{minipage}
\end{tabular}
\end{figure}

It is known that $C_1 = 1$ and we can take $C'_2 = 3$ (\cite{IN}*{Example A.3}). 
$\operatorname{Im}(\nu)$ consists of $11$ points as in Figure \ref{fig:inv} (\cite{IN}*{Example 2.22}). 

\begin{figure}[htbp]
\centering
\includegraphics[width=8cm]{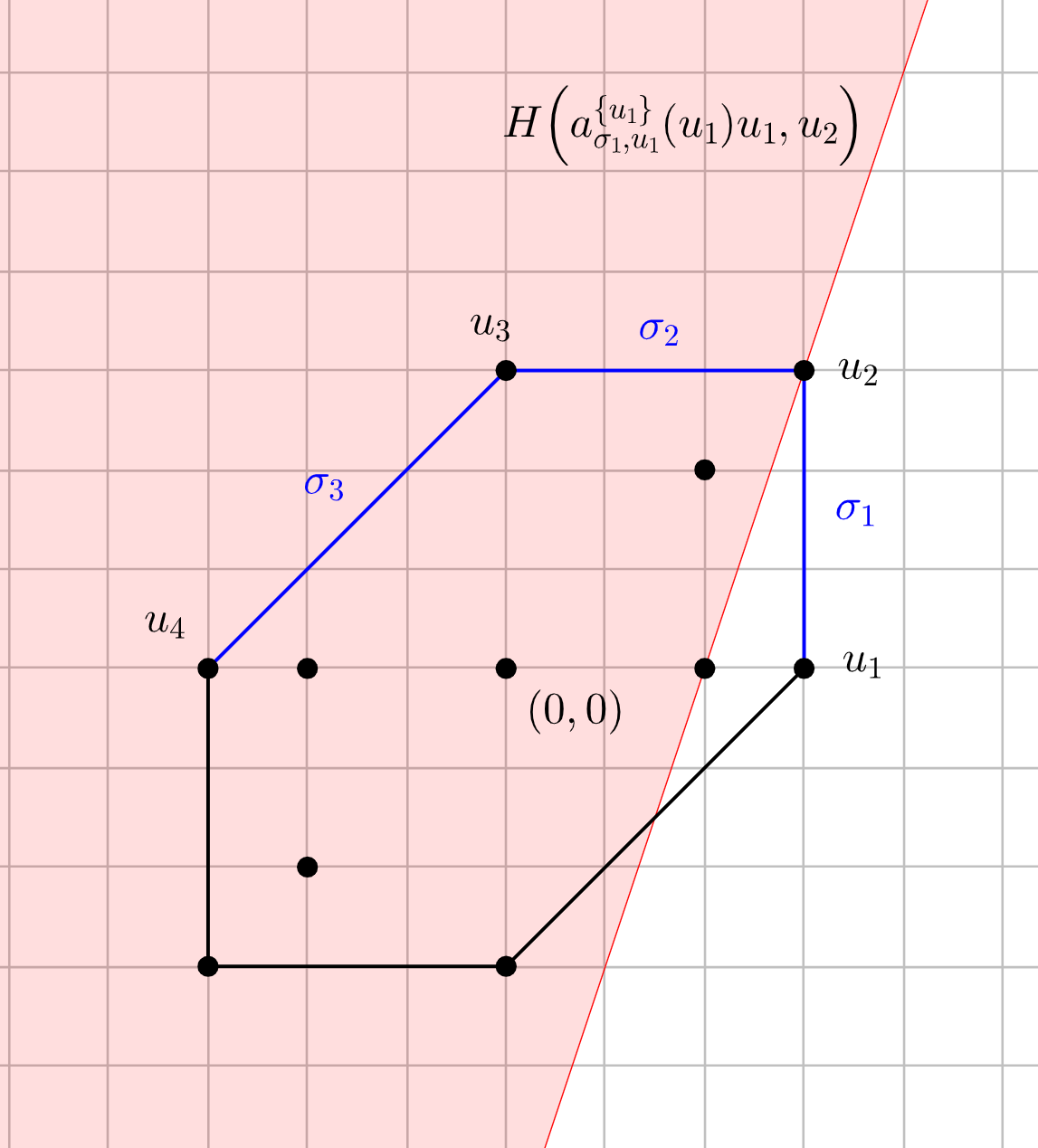}
\caption{$\sigma _1, \sigma _2, \sigma _3$, and $u_1, u_2, u_3, u_4$.}
\label{fig:inv}
\end{figure}

Let $\sigma _1, \sigma _2, \sigma _3 \in \operatorname{Facet}(P)$ and $u_1, u_2, u_3, u_4 \in V(P)$ as in Figure \ref{fig:inv}. 
Then, we have $\operatorname{cpx}_{\sigma}(v) = 2$ for any $v \in V(P)$ and $\sigma \in \operatorname{Facet}(P)$, 
and the values $s_{\sigma, v}^{\{v\}}$, $a_{\sigma, v}^{\{v\}}(v)$, $\beta _{\sigma, v}$ and 
$\beta _{\sigma, \sigma}$ are given by the following table. 
By symmetry, we have $\operatorname{cpx}_{\Gamma} = 2$ and $\beta = C'_2 + \max _{i = 1,2,3} \beta _{\sigma _i, \sigma _i} = 25$. 
\begin{table}[h]
  \centering
  \begin{tabular}{ccccccc}
    \hline
    $(\sigma, v)$ & $(\sigma _1, u_1)$ & $(\sigma _1, u_2)$ & $(\sigma _2, u_2)$ & $(\sigma _2, u_3)$ & $(\sigma _3, u_3)$ & $(\sigma _3, u_4)$   \\
    \hline \hline
    $s_{\sigma, v}^{\{v\}}$ & $2$ & $2$ & $2$ & $2$ & $2$ & $2$ \\
    $a_{\sigma, v}^{\{v\}}(v)$ & $\frac{2}{3}$ & $\frac{2}{3}$ & $\frac{2}{3}$ & $\frac{1}{2}$ & $\frac{1}{2}$ & $\frac{2}{3}$ \\
    $\beta _{\sigma, v}$ & $11$ & $11$ & $11$ & $7$ & $7$ & $11$ \\
    $\beta _{\sigma, \sigma}$ & \multicolumn{2}{c}{22} & \multicolumn{2}{c}{18} & \multicolumn{2}{c}{18} \\
    \hline
\end{tabular}
\vspace{4mm}
\caption{Invariants for the Wakatsuki graph.}
\label{table:WG}
\end{table}

We shall briefly explain the values 
$\operatorname{cpx}_{\sigma _1}(u_1)$, 
$s_{\sigma _1, u_1}^{\{ u_1 \}}$ and $a_{\sigma _1, u_1}^{\{ u_1 \}} (u_1)$. 
First, we can see that 
\[
\operatorname{Cyc}_{\Gamma /L}(\{ u_1 \})
= \left \{ \overline{e'_0}\ \overline{e'_5}, \ \overline{e'_5}\ \overline{e'_0} \right \}. 
\]
Therefore, we have $\operatorname{cpx}_{\sigma _1}(u_1) = \operatorname{LCM}\{ 2,2 \} = 2$. 
Since $\operatorname{supp}\left( \overline{e'_0}\ \overline{e'_5} \right) = 
\left \{ \overline{v'_0}, \overline{v'_1} \right \} =
\operatorname{supp}\left( \overline{e'_5}\ \overline{e'_0} \right)$, 
we have $\operatorname{num}(u_1, 2) = 1$, and therefore, 
we have $s_{\sigma _1, u_1}^{\{ u_1 \}} = \operatorname{cpx}_{\sigma _1}(u_1)  = 2$. 
Next, $a_{\sigma _1, u_1}^{\{ u_1 \}} (u_1)$ is defined as the minimum real number $\alpha$ satisfying
\[
\operatorname{Im}(\nu) \setminus \{ u_1 \} \subset H(\alpha u_1, u_2). 
\]
$H \left( \alpha u_1, u_2 \right)$ is illustrated in Figure \ref{fig:inv}. 
\end{ex}

\subsection{Main theorem}\label{subsection:main}
We keep the notations in Subsection \ref{subsection:inv}. 

\begin{defi}
Let $\sigma \in \operatorname{Facet}(P)$ and $\Delta \in T_{\sigma}$. 
For $x \in L_{\mathbb{R}}$ and $v \in V(\Delta)$, we define $\operatorname{proj}_{\Delta, v}(x) \in \mathbb{R}$ by the unique expression
$x = \sum _{v \in V(\Delta)} \operatorname{proj}_{\Delta, v}(x) v$. 
\end{defi}

The following lemma explains why we defined $\alpha _{\Delta, v}$ and $\alpha' _{\Delta, v}$ as in Subsection \ref{subsubsection:beta}. 

\begin{lem}\label{lem:rs}
Let $\sigma \in \operatorname{Facet}(P)$, $\Delta \in T_{\sigma}$ and $v \in V(\Delta)$. 
Let $y \in V_{\Gamma} \cap \Phi^{-1} (\mathbb{R}_{\ge 0} \Delta)$, 
and let $p$ be a shortest walk in $\Gamma$ from $x_0$ to $y$. 
By applying Lemma \ref{lem:bunkai}(1) to $\overline{p}$, 
we take a path $q_0$ in $\Gamma /L$ and a function $a : \operatorname{Cyc}_{\Gamma /L} \to \mathbb{Z}_{\ge 0}$ 
such that $\langle \overline{p} \rangle = \langle q_0 \rangle + \sum _{q \in \operatorname{Cyc}_{\Gamma /L}} a(q) \langle q \rangle$. 
Then the following assertions hold. 

\begin{enumerate}
\item
We fix any $F \in \mathcal{H}'' _{\sigma, v}$ and $t \in \mathbb{R}_{\ge 0}$. 
If $\operatorname{proj}_{\Delta, v} (\Phi(y)) > f^F(t)$ holds for
\[
f^F(t) := \frac{a^{F} _{\Delta, v} (v)}{1-a^{F} _{\Delta, v} (v)}
\left( \frac{C_1}{h^F_{\Delta, v}} + C'_2 + \frac{1-h^F_{\Delta, v}}{h^F_{\Delta, v}} \bigl( t + W ( \#(V_{\Gamma}/L) -1) \bigr) \right), 
\]
then we have
\[
\sum _{q \in \operatorname{Cyc}_{\Gamma /L}(F)} a(q) \cdot w(q) > t. 
\]

\item
In particular, if $\operatorname{proj}_{\Delta, v} (\Phi(y)) > \alpha _{\Delta, v}$ (resp.\ $\operatorname{proj}_{\Delta, v} (\Phi(y)) > \alpha' _{\Delta, v}$) holds, then we have 
\[
\sum _{q \in \operatorname{Cyc}_{\Gamma /L}(F)} a(q) \cdot w(q) > r^{F} _{\sigma, v} \qquad
(\text{resp.\ $\sum _{q \in \operatorname{Cyc}_{\Gamma /L}(F)} a(q) \cdot w(q) > s^{F} _{\sigma, v}$})
\]
for each $F \in \mathcal{H}'' _{\sigma, v}$. 
\end{enumerate}
\end{lem}
\begin{proof}
We prove (1). 
Let $F \in \mathcal{H}'' _{\sigma, v}$ and $t \in \mathbb{R}_{\ge 0}$. 
We assume 
\[
\sum _{q \in \operatorname{Cyc}_{\Gamma /L}(F)} a(q) \cdot w(q) \le t. 
\]

First, we have
\begin{itemize}
\item 
$\mu(\langle q \rangle) \in w(q) \cdot
H \bigl( \bigl \{ a^{F} _{\Delta, v} (v') v' \mid v' \in V(\Delta) \bigr \} \bigr)$ 
for $q \in \operatorname{Cyc}_{\Gamma /L} \setminus \operatorname{Cyc}_{\Gamma /L}(F)$, 

\item 
$\mu(\langle q \rangle) \in w(q) \cdot P$ for 
$q \in \operatorname{Cyc}_{\Gamma /L}$, and 

\item 
$\mu(\langle q_0 \rangle) \in (C_1 + w(q_0)) P$. 
\end{itemize}
For simplicity, we put 
\begin{align*}
\ell _0 &:= w(q_0), \\
\ell _1 &:= \sum _{q \in \operatorname{Cyc}_{\Gamma /L}(F)} a(q) \cdot w(q), \\
\ell _2 &:= \sum _{q \in \operatorname{Cyc}_{\Gamma /L} \setminus \operatorname{Cyc}_{\Gamma /L}(F)} 
a(q) \cdot w(q), \\
H' &:= H \bigl( \bigl \{ a^{F} _{\Delta, v} (v') v' \ \big | \ v' \in V(\Delta) \bigr \} \bigr). 
\end{align*}
Then, we have 
\[
\sum _{q \in \operatorname{Cyc}_{\Gamma /L}(F)} a(q) \cdot \mu(\langle q \rangle) 
\in \ell _1 P, \qquad
\sum _{q \in \operatorname{Cyc}_{\Gamma /L} \setminus \operatorname{Cyc}_{\Gamma /L}(F)} a(q) \cdot \mu(\langle q \rangle) 
\in \ell _2 H'. 
\]
Furthermore, we have
\begin{itemize}
\item
$\ell _0 + \ell _1 + \ell _2 = w(p) = d_{\Gamma}(x_0, y)$, 

\item
$\ell _0 \le W \cdot \operatorname{length}(q_0)  \le W(\#(V_{\Gamma}/L) -1)$, and

\item
$\ell _1 = 
\sum _{q \in \operatorname{Cyc}_{\Gamma /L}(F)} a(q) \cdot w(q) \le t$.
\end{itemize}
Therefore, we have
\begin{align*}
\Phi (y) &= \mu(\langle \overline{p} \rangle) \\
&= \mu(\langle q_0 \rangle)
+ \sum _{q \in \operatorname{Cyc}_{\Gamma /L}(F)} a(q) \cdot \mu(\langle q \rangle)+ 
\sum _{q \in \operatorname{Cyc}_{\Gamma /L} \setminus \operatorname{Cyc}_{\Gamma /L}(F)} a(q) \cdot \mu(\langle q \rangle) \\
&\in (C_1 + \ell _0)P + \ell _1 P + \ell _2 H'. 
\end{align*}
Since we have $P \subset (h^F _{\Delta, v})^{-1} H'$ by the choice of $h^F _{\Delta, v}$, we have
\[
\Phi (y) \in \bigl( (h^F _{\Delta, v})^{-1} ( C_1 + \ell _0 + \ell _1) + \ell _2 \bigr) H'. 
\]
Since we have
\begin{align*}
&(h^F _{\Delta, v})^{-1} ( C_1 + \ell _0 + \ell _1) + \ell _2 \\
&= (h^F _{\Delta, v})^{-1} C_1 + \bigl( (h^F _{\Delta, v})^{-1} - 1 \bigr) (\ell _0 + \ell _1) + d_{\Gamma}(x_0, y)\\
&\le (h^F _{\Delta, v})^{-1} C_1 + \bigl( (h^F _{\Delta, v})^{-1} - 1 \bigr) \bigl( t + W( \#(V_{\Gamma}/L) -1) \bigr) + d_{\Gamma}(x_0, y), 
\end{align*}
we have 
\[
\Phi (y) \in 
\Bigl( 
(h^F _{\Delta, v})^{-1} C_1 + \bigl( (h^F _{\Delta, v})^{-1} - 1 \bigr) 
\bigl( t + W(\#(V_{\Gamma}/L) -1) \bigr) + d_{\Gamma}(x_0, y)
\Bigr)H'. 
\]
Therefore, we have 
\[
\sum _{v' \in V(\Delta)} \frac{\operatorname{proj}_{\Delta, v'}(\Phi(y))}{a^F _{\Delta, v}(v')} \le 
\frac{C_1}{h^F _{\Delta, v}} + \frac{1 - h^F _{\Delta, v}}{h^F _{\Delta, v}} \bigl( t + W( \#(V_{\Gamma}/L) -1) \bigr) +  d_{\Gamma}(x_0, y). 
\]
On the other hand, by the definition of $C_2$, we have 
\[
d_{\Gamma}(x_0, y) \le C'_2 + d_{P, \Phi}(x_0, y) = 
C'_2 + \sum _{v' \in V(\Delta)} \operatorname{proj}_{\Delta, v'}(\Phi(y)). 
\]
Since $a^F _{\Delta, v}(v') \le 1$ holds for any $v' \in V(\Delta)$, we have
\begin{align*}
&\frac{1- a^F _{\Delta, v}(v)}{a^F _{\Delta, v}(v)} \operatorname{proj}_{\Delta, v}(\Phi(y)) \\
&\le
\left( \sum _{v' \in V(\Delta)} \frac{\operatorname{proj}_{\Delta, v'}(\Phi(y))}{a^F _{\Delta, v}(v')} \right) - 
\left( \sum _{v' \in V(\Delta)} \operatorname{proj}_{\Delta, v'}(\Phi(y)) \right) \\
&\le
\frac{C_1}{h^F _{\Delta, v}} + C'_2 + \frac{1 - h^F _{\Delta, v}}{h^F _{\Delta, v}} 
\bigl( t + W (\#(V_{\Gamma}/L) -1) \bigr). 
\end{align*}
Therefore, we have
\begin{align*}
\operatorname{proj}_{\Delta, v}(\Phi(y))
&\le \frac{a^F _{\Delta, v}(v)}{1- a^F _{\Delta, v}(v)}
	\left( \frac{C_1}{h^F _{\Delta, v}} + C'_2 + \frac{1 - h^F _{\Delta, v}}{h^F _{\Delta, v}} 
\bigl( t + W (\#(V_{\Gamma}/L) -1) \bigr) \right) \\
&= f^F(t), 
\end{align*}
which completes the proof of (1). 

(2) follows from (1) and the definitions of $\alpha _{\Delta, v}$ and $\alpha' _{\Delta, v}$. 
\end{proof}

\begin{thm}\label{thm:mod}
Let $\sigma \in \operatorname{Facet}(P)$, $\Delta \in T_{\sigma}$ and $v \in V(\Delta)$. 
Then for any $y \in V_{\Gamma} \cap \Phi ^{-1}( \mathbb{R}_{\ge 0} \Delta)$ with $\operatorname{proj}_{\Delta, v} (\Phi(y)) > \alpha _{\Delta, v}$, we have  
\[
d_{\Gamma}\bigl( x_0, y + \operatorname{cpx}_{\sigma}(v)v \bigr) 
\le d_{\Gamma}( x_0, y) + \operatorname{cpx}_{\sigma}(v).
\]
\end{thm}
\begin{proof}
Suppose that  $y \in V_{\Gamma} \cap \Phi ^{-1}( \mathbb{R}_{\ge 0} \Delta)$ satisfies $\operatorname{proj}_{\Delta, v} (\Phi(y)) > \alpha _{\Delta, v}$. 
Let $p$ be a shortest walk in $\Gamma$ from $x_0$ to $y$. 
By applying Lemma \ref{lem:bunkai}(1) to $\overline{p}$, 
we can take a path $q_0$ in $\Gamma /L$ and a function $a : \operatorname{Cyc}_{\Gamma /L} \to \mathbb{Z}_{\ge 0}$ 
such that $\langle \overline{p} \rangle = \langle q_0 \rangle + \sum _{q \in \operatorname{Cyc}_{\Gamma /L}} a(q) \langle q \rangle$. 
Then, by Lemma \ref{lem:rs}, we have 
\[
\sum _{q \in \operatorname{Cyc}_{\Gamma /L}(F)} a(q) \cdot w(q) > r^{F} _{\sigma, v}
\]
for each $F \in \mathcal{H}'' _{\sigma, v}$. 
Then, by the condition (R) in Subsection \ref{subsubsection:rs}, 
there exists a function $b: \operatorname{Cyc}_{\Gamma /L}(\sigma) \to \mathbb{Z}_{\ge 0}$ such that 
$b^{-1}(\mathbb{Z}_{>0}) \subset a^{-1}(\mathbb{Z}_{>0})$ and that 
\[
\sum _{q \in \operatorname{Cyc}_{\Gamma /L}(\sigma)} b(q) \cdot \mu(\langle q \rangle) = \operatorname{cpx}_{\sigma}(v) v.
\]
Since any $q \in \operatorname{Cyc}_{\Gamma /L}(\sigma)$ satisfies $\frac{\mu(\langle q \rangle)}{w(q)} \in \sigma$, 
we also have 
\[
\sum _{q \in \operatorname{Cyc}_{\Gamma /L}(\sigma)} b(q) \cdot w(q) = \operatorname{cpx}_{\sigma}(v). 
\]
We define a function $a' : \operatorname{Cyc}_{\Gamma /L} \to \mathbb{Z}_{\ge 0}$ by 
\[
a'(q) := 
\begin{cases}
a(q) + b(q) & \text{if $q \in \operatorname{Cyc}_{\Gamma /L}(\sigma)$,} \\
a(q)  & \text{if $q \in \operatorname{Cyc}_{\Gamma /L} \setminus \operatorname{Cyc}_{\Gamma /L}(\sigma)$}. 
\end{cases}
\]
Then, we have $a^{\prime -1}(\mathbb{Z}_{>0}) = a^{-1}(\mathbb{Z}_{>0})$ since $b^{-1}(\mathbb{Z}_{>0}) \subset a^{-1}(\mathbb{Z}_{>0})$. 
Therefore, $(q_0, a')$ is also a walkable pair by Lemma \ref{lem:bunkai}(2). 
Hence, there exists a walk $p'$ in $\Gamma$ with $s(p') = x_0$ such that 
\[
\langle \overline{p'} \rangle = \langle q_0 \rangle + \sum _{q \in \operatorname{Cyc}_{\Gamma /L}} a'(q) \cdot \langle q \rangle
= \langle \overline{p} \rangle + \sum _{q \in \operatorname{Cyc}_{\Gamma /L}(\sigma)} b(q) \cdot \langle q \rangle. 
\]
Hence, we have 
\begin{align*}
w(p') 
&= w(p) + \sum _{q \in \operatorname{Cyc}_{\Gamma /L}(\sigma)} b(q) \cdot w(q) 
= d_{\Gamma}( x_0, y) + \operatorname{cpx}_{\sigma}(v), \\
t(p') 
&= t(p) + \sum _{q \in \operatorname{Cyc}_{\Gamma /L}(\sigma)} b(q) \cdot \mu(\langle q \rangle) 
= y + \operatorname{cpx}_{\sigma}(v)v. 
\end{align*}
Therefore, $p'$ is a walk from $x_0$ to $y + \operatorname{cpx}_{\sigma}(v)v$ of weight 
equal to $d_{\Gamma}( x_0, y) + \operatorname{cpx}_{\sigma}(v)$, which shows the desired inequality. 
\end{proof}

\begin{thm}\label{thm:gen}
Let $\sigma \in \operatorname{Facet}(P)$, $\Delta \in T_{\sigma}$ and $v \in V(\Delta)$. 
Then for any $y \in V_{\Gamma} \cap \Phi ^{-1}(\mathbb{R}_{\ge 0} \Delta)$ with $\operatorname{proj}_{\Delta, v} (\Phi(y)) > \alpha ' _{\Delta, v}$, we have  
\[
d_{\Gamma}\bigl( x_0, y - \operatorname{cpx}_{\sigma}(v)v \bigr) 
\le d_{\Gamma}( x_0, y) - \operatorname{cpx}_{\sigma}(v). 
\]
\end{thm}
\begin{proof}
Suppose that $y \in V_{\Gamma} \cap \Phi ^{-1}(\mathbb{R}_{\ge 0} \Delta)$ satisfies $\operatorname{proj}_{\Delta, v} (\Phi(y)) > \alpha' _{\Delta, v}$. 
Let $p$ be a shortest walk in $\Gamma$ from $x_0$ to $y$. 
By applying Lemma \ref{lem:bunkai}(1) to $\overline{p}$, 
we can take a path $q_0$ in $\Gamma /L$ and a function $a : \operatorname{Cyc}_{\Gamma /L} \to \mathbb{Z}_{\ge 0}$ 
such that $\langle \overline{p} \rangle = \langle q_0 \rangle + \sum _{q \in \operatorname{Cyc}_{\Gamma /L}} a(q) \langle q \rangle$. 
Then, by Lemma \ref{lem:rs}, we have 
\[
\sum _{q \in \operatorname{Cyc}_{\Gamma /L}(F)} a(q) \cdot w(q) > s^{F} _{\sigma, v}
\]
for each $F \in \mathcal{H}'' _{\sigma, v}$. 
Then, by the condition (S) in Subsection \ref{subsubsection:rs}, 
there exists a function $b: \operatorname{Cyc}_{\Gamma /L}(\sigma) \to \mathbb{Z}_{\ge 0}$ such that 
\begin{itemize}
\item 
$b(q) \le a(q)$ holds for any $q \in \operatorname{Cyc}_{\Gamma /L}(\sigma)$, 

\item 
$\sum _{q \in \operatorname{Cyc}_{\Gamma /L}(\sigma)} b(q) \cdot \mu(\langle q \rangle) = \operatorname{cpx}_{\sigma}(v) v$, and 

\item
for any $q \in b^{-1}(\mathbb{Z}_{>0})$, 
there exists $q' \in \operatorname{Cyc}_{\Gamma /L}(\sigma)$ 
such that $b(q') < a(q')$ and $\operatorname{supp}(q) \subset \operatorname{supp}(q')$. 
\end{itemize}
Since any $q \in \operatorname{Cyc}_{\Gamma /L}(\sigma)$ satisfies $\frac{\mu(\langle q \rangle)}{w(q)} \in \sigma$, 
we also have 
\[
\sum _{q \in \operatorname{Cyc}_{\Gamma /L}(\sigma)} b(q) \cdot w(q) = \operatorname{cpx}_{\sigma}(v). 
\]
We define a function $a' : \operatorname{Cyc}_{\Gamma /L} \to \mathbb{Z}_{\ge 0}$ by 
\[
a'(q) := 
\begin{cases}
a(q) - b(q) & \text{(if $q \in \operatorname{Cyc}_{\Gamma /L}(\sigma)$)} \\
a(q)  & \text{(if $q \in \operatorname{Cyc}_{\Gamma /L} \setminus \operatorname{Cyc}_{\Gamma /L}(\sigma)$)}. 
\end{cases}
\]
Then, by the third condition on $b$ above, 
$(q_0, a')$ is also a walkable pair by Lemma \ref{lem:bunkai}(2). 
Hence, there exists a walk $p'$ in $\Gamma$ with $s(p') = x_0$ such that 
\[
\langle \overline{p'} \rangle = \langle q_0 \rangle + \sum _{q \in \operatorname{Cyc}_{\Gamma /L}} a'(q) \cdot \langle q \rangle
= \langle \overline{p} \rangle - \sum _{q \in \operatorname{Cyc}_{\Gamma /L}(\sigma)} b(q) \cdot \langle q \rangle. 
\]
Therefore, we have
\begin{align*}
w(p') 
&= w(p) - \sum _{q \in \operatorname{Cyc}_{\Gamma /L}(\sigma)} b(q) \cdot w(q)
= d_{\Gamma}( x_0, y) - \operatorname{cpx}_{\sigma}(v), \\
t(p') 
&= t(p) - \sum _{q \in \operatorname{Cyc}_{\Gamma /L}(\sigma)} b(q) \cdot  \mu(\langle q \rangle) 
= y  - \operatorname{cpx}_{\sigma}(v)v. 
\end{align*}
Therefore, $p'$ is a walk from $x_0$ to $y - \operatorname{cpx}_{\sigma}(v)v$ of weight 
equal to $d_{\Gamma}( x_0, y) - \operatorname{cpx}_{\sigma}(v)$, which shows the desired inequality. 
\end{proof}

\begin{thm}\label{thm:free}
Let $\sigma \in \operatorname{Facet}(P)$, $\Delta \in T_{\sigma}$ and $v \in V(\Delta)$. 
Then for any $y \in V_{\Gamma} \cap \Phi ^{-1}(\mathbb{R}_{\ge 0} \Delta)$ with $\operatorname{proj}_{\Delta, v} (\Phi(y)) > \beta _{\Delta, v}$, we have  
\[
d_{\Gamma}\bigl( x_0, y + \operatorname{cpx}_{\sigma}(v)v \bigr) 
= d_{\Gamma}( x_0, y) + \operatorname{cpx}_{\sigma}(v). 
\]
\end{thm}
\begin{proof}
Since $\beta_{\Delta, v} = \max \bigl\{ \alpha _{\Delta, v}, \alpha ' _{\Delta, v} - \operatorname{cpx}_{\sigma}(v) \bigr\}$, we have
\begin{align*}
&\operatorname{proj}_{\Delta, v} (\Phi(y)) > \beta _{\Delta, v}  \ge \alpha _{\Delta, v}, \\
&\operatorname{proj}_{\Delta, v} \bigl( \Phi(y + \operatorname{cpx}_{\sigma}(v)v ) \bigr) 
> \beta _{\Delta, v} + \operatorname{cpx}_{\sigma}(v) \ge \alpha' _{\Delta, v}. 
\end{align*}
Therefore, the assertion follows from Theorems \ref{thm:mod} and \ref{thm:gen}. 
\end{proof}

\begin{defi}\label{defi:M}
\begin{enumerate}
\item 
We define 
\[
B := \{ (d,y) \in \mathbb{Z}_{\ge 0} \times V_{\Gamma} \mid d_{\Gamma}(x_0, y) \le d \} \subset \mathbb{Z}_{\ge 0} \times V_{\Gamma}. 
\]

\item
For any subset $F \subset L_{\mathbb{R}}$, we define 
\[
B(F) := B \cap  \bigl( \mathbb{Z}_{\ge 0} \times \Phi^{-1}(F) \bigr). 
\]
We also define $S(F)$ by 
\[
S(F) := \bigl\{ \bigl( d_{\Gamma}(x_0, y), y \bigr) \ \big| \  y \in V_{\Gamma} \cap \Phi ^{-1}(F)  \bigr\}. 
\]
Note that we have $B(F) = \mathbb{Z}_{\ge 0}(1,0) + S(F)$. 

\item 
Let $\sigma \in \operatorname{Facet}(P)$, $\Delta \in T_{\sigma}$ and $\Delta ' \in \operatorname{Face}(\Delta)$. 
We define 
\[
L_{\Delta '} := \mathbb{R}_{\ge 0}\Delta ' \subset L_{\mathbb{R}}. 
\]
For $v \in V(\Delta)$, we also define 
\begin{align*}
L_{\Delta, v} ^{> \beta} 
&:= \{ z \in L_{\Delta} \mid \operatorname{proj}_{\Delta, v}(z) > \beta _{\Delta, v} \}, \\
L_{\Delta, v} ^{\le \beta} 
&:= \{ z \in L_{\Delta} \mid \operatorname{proj}_{\Delta, v}(z) \le \beta _{\Delta, v} \}, \\
L_{\Delta, v} ^{(\beta, \beta']} 
&:= \{ z \in L_{\Delta} \mid  \beta' _{\Delta, v}  \ge \operatorname{proj}_{\Delta, v}(z) > \beta _{\Delta, v} \} \subset L_{\Delta, v} ^{> \beta}. 
\end{align*}
For $\Delta '' \in \operatorname{Face}(\Delta ')$, we define  
\begin{align*}
L_{\Delta, \Delta ', \Delta''}
&:= L_{\Delta '}
\cap \Biggl( \bigcap _{v \in V(\Delta ') \setminus V(\Delta '')} L_{\Delta, v} ^{\le \beta} \Biggr)
\cap \Biggl( \bigcap _{v \in V(\Delta '')} L_{\Delta, v} ^{> \beta} \Biggr)
\subset L_{\Delta '}, \\
\overline{L}_{\Delta, \Delta ', \Delta''} 
&:= L_{\Delta '}
\cap \Biggl( \bigcap _{v \in V(\Delta ') \setminus V(\Delta '')} L_{\Delta, v} ^{\le \beta} \Biggr) 
\cap \Biggl( \bigcap _{v \in V(\Delta '')} L_{\Delta, v} ^{(\beta, \beta']} \Biggr)
\subset L_{\Delta, \Delta ', \Delta''}. 
\end{align*}
Then we have 
\[
L_{\Delta '} = \bigsqcup _{\Delta '' \in \operatorname{Face}(\Delta ')} L_{\Delta, \Delta ', \Delta''}. 
\]

\item 
For $\sigma \in \operatorname{Facet}(P)$, $\Delta \in T_{\sigma}$ and $\Delta ' \in \operatorname{Face}(\Delta)$, 
we define a monoid $M_{\sigma, \Delta '}$ by 
\[
M_{\sigma, \Delta '} := \mathbb{Z}_{\ge 0}(1,0) + 
\sum _{v \in V(\Delta ')} \mathbb{Z}_{\ge 0} \bigl( \operatorname{cpx}_{\sigma}(v), \operatorname{cpx}_{\sigma}(v) v \bigr) 
\subset \mathbb{Z} _{\ge 0} \times L. 
\]
This is a free submonoid of $\mathbb{Z} _{\ge 0} \times L$ generated by $(1,0)$ and 
$\bigl( \operatorname{cpx}_{\sigma}(v), \operatorname{cpx}_{\sigma}(v) v \bigr)$ for $v \in V(\Delta ')$. 
\end{enumerate}
\end{defi}

\begin{figure}[htbp]
\centering
\includegraphics[width=140mm]{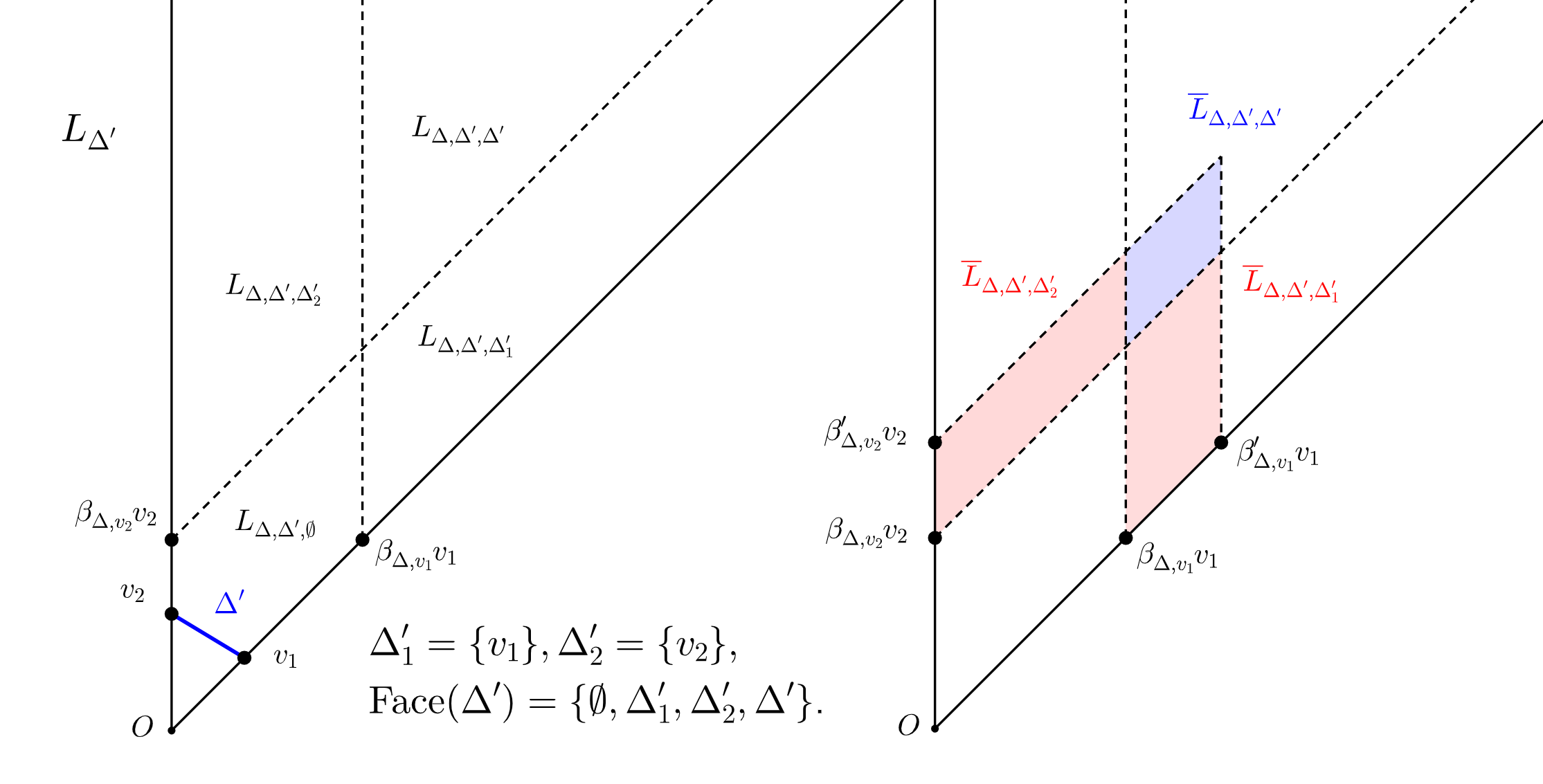}
\caption{$L_{\Delta '}$, $L_{\Delta, \Delta', \Delta''}$, and $\overline{L}_{\Delta, \Delta', \Delta''}$.}
\label{fig:L2}
\end{figure}

\begin{thm}\label{thm:fingen}
Let $\sigma \in \operatorname{Facet}(P)$, $\Delta \in T_{\sigma}$, 
$\Delta ' \in \operatorname{Face}(\Delta)$ and $\Delta '' \in \operatorname{Face}(\Delta ')$. 
Then, the following assertions hold. 
\begin{enumerate}
\item 
$B \left( L_{\Delta, \Delta ', \Delta''} \right)$ is a free $M_{\sigma, \Delta ''}$-module with basis 
$S \left(\overline{L}_{\Delta, \Delta ', \Delta''} \right)$. 

\item
For any $(d, y) \in S \left(\overline{L}_{\Delta, \Delta ', \Delta''} \right)$, 
we have $d \le C'_2 + \beta _{\Delta, \Delta '} + \sum _{v \in V(\Delta '')} \operatorname{cpx}_{\sigma}(v)$. 
\end{enumerate}
\end{thm}
\begin{proof}
We prove (1). First, we prove that $B \left( L_{\Delta, \Delta ', \Delta''} \right)$ is an  $M_{\sigma, \Delta ''}$-module. 
By definition, it is easy to see that $B \left( L_{\Delta, \Delta ', \Delta''} \right) + (1,0) \subset B \left( L_{\Delta, \Delta ', \Delta''} \right)$. 
We fix $(d, y) \in B \left( L_{\Delta, \Delta ', \Delta''} \right)$ and $v \in V(\Delta'')$. 
By the definition of $L_{\Delta, \Delta ', \Delta''}$, we have $\operatorname{proj}_{\Delta, v} (\Phi (y)) > \beta _{\Delta, v}$. 
Therefore, By Theorem \ref{thm:free}, we have 
\[
d_{\Gamma}(x_0, y + \operatorname{cpx}_{\sigma}(v) v) 
= d_{\Gamma}(x_0, y) + \operatorname{cpx}_{\sigma}(v) 
\le d + \operatorname{cpx}_{\sigma}(v). 
\]
Furthermore, we have $\Phi (y + \operatorname{cpx}_{\sigma}(v) v) \in L_{\Delta, \Delta ', \Delta''}$. 
Hence, we conclude that 
\[
(d,y) + \bigl( \operatorname{cpx}_{\sigma}(v), \operatorname{cpx}_{\sigma}(v) v \bigr) \in B \left( L_{\Delta, \Delta ', \Delta''} \right). 
\]

We fix $(d, y) \in B \left( L_{\Delta, \Delta ', \Delta''} \right)$. 
By the definitions of $L_{\Delta, \Delta ', \Delta''}$ and $\overline{L}_{\Delta, \Delta ', \Delta''}$, 
there exist $y' \in V_{\Gamma} \cap \Phi^{-1} \left( \overline{L}_{\Delta, \Delta ', \Delta''} \right)$ and 
$b_v \in \mathbb{Z}_{\ge 0}$ for $v \in V(\Delta '')$ such that 
\[
y = y' + \sum _{v \in V(\Delta '')} b_v \cdot \operatorname{cpx}_{\sigma}(v) \cdot v. 
\]
By Theorem \ref{thm:free}, we have 
\[
d_{\Gamma}(x_0, y) 
= d_{\Gamma}(x_0, y') + \sum _{v \in V(\Delta '')} b_v \cdot \operatorname{cpx}_{\sigma}(v).
\]
Therefore, for $d' := d_{\Gamma}(x_0, y')$, we have $(d', y') \in S \left(\overline{L}_{\Delta, \Delta ', \Delta''} \right)$ and 
\begin{align*}
(d,y) 
&= 
\left( d', y' \right) + \left( d-d_{\Gamma}(x_0, y) \right) \cdot (1, 0) 
+ \sum _{v \in V(\Delta '')} b_v \cdot \bigl( \operatorname{cpx}_{\sigma}(v), \operatorname{cpx}_{\sigma}(v) v \bigr) \\
&\in
\left( d', y' \right) + M_{\sigma, \Delta ''}. 
\end{align*}
We have proved that $B \left( L_{\Delta, \Delta ', \Delta''} \right)$ is an $M_{\sigma, \Delta ''}$-module generated by 
$S \left(\overline{L}_{\Delta, \Delta ', \Delta''} \right)$. 

Suppose that $(d, y), (d', y') \in S \left(\overline{L}_{\Delta, \Delta ', \Delta''} \right)$ satisfies 
\begin{align*}
&(d,y) + a \cdot (1,0) + \sum _{v \in V(\Delta '')} b_v \cdot \operatorname{cpx}_{\sigma}(v)\\
&= 
(d',y') + a' \cdot (1,0) + \sum _{v \in V(\Delta '')} b'_v \cdot \operatorname{cpx}_{\sigma}(v)
\end{align*}
with some $a, a' \in \mathbb{Z}_{\ge 0}$ and $b_v, b'_v \in \mathbb{Z}_{\ge 0}$ for $v \in V(\Delta '')$. 
Then, for each $v \in V(\Delta '')$, we have 
\[
\operatorname{proj}_{\Delta, v} (\Phi (y) ) + b_v \cdot \operatorname{cpx}_{\sigma}(v) = 
\operatorname{proj}_{\Delta, v} (\Phi(y')) + b'_v \cdot \operatorname{cpx}_{\sigma}(v). 
\]
Here, by the definition of $\overline{L}_{\Delta, \Delta ', \Delta''}$, 
we have 
\[
\operatorname{proj}_{\Delta, v} (\Phi (y) ), \ \operatorname{proj}_{\Delta, v} (\Phi(y')) \in \left( \beta_{\Delta, v}, \beta' _{\Delta, v} \right]. 
\]
Since $\beta' _{\Delta, v} - \beta _{\Delta, v} = \operatorname{cpx}_{\sigma}(v) $, we have $y=y'$ and $b_v = b'_v$ for any $v \in V(\Delta '')$. 
Since $y = y'$, we have $d = d_{\Gamma} (x_0, y) = d_{\Gamma} (x_0, y')  = d'$ by the definition of $S \left(\overline{L}_{\Delta, \Delta ', \Delta''}  \right)$. 
Thus, we also have $a = a'$. 
Therefore, we conclude that $B \left( L_{\Delta, \Delta ', \Delta''} \right)$ is freely generated by $S \left(\overline{L}_{\Delta, \Delta ', \Delta''}  \right)$. We complete the proof of (1). 

Next, we prove (2). 
For $y \in V_{\Gamma} \cap \Phi ^{-1}\left( \overline{L}_{\Delta, \Delta ', \Delta''} \right)$, we have 
\begin{align*}
d_{\Gamma}(x_0, y)
&\le C' _2 + d_{P, \Phi}(x_0, y) \\
&= C' _2 + \sum _{v \in V(\Delta ')} \operatorname{proj}_{\Delta, v}(\Phi(y)) \\
&\le C'_2 + \sum _{v \in V(\Delta '')} \beta' _{\Delta, v} + \sum _{v \in V(\Delta ') \setminus V(\Delta '')} \beta _{\Delta, v}\\
&\le C'_2 + \beta _{\Delta, \Delta '} + \sum _{v \in V(\Delta '')} \operatorname{cpx}_{\sigma}(v), 
\end{align*}
which completes the proof. 
\end{proof}

\begin{defi}
For $d \in \mathbb{Z}_{\ge 0}$ and a subset $S \subset \mathbb{Z}_{\ge 0} \times V_{\Gamma}$, we define 
\[
S_d := \{ x \in V_{\Gamma} \mid (d,x) \in S \}. 
\]
Furthermore, we define a function $f_{S} : \mathbb{Z}_{\ge 0} \to \mathbb{Z}_{\ge 0}$ by $f_{S} (d) := \# S_d$. 
\end{defi}

\begin{thm}\label{thm:main}
Let $\sigma \in \operatorname{Facet}(P)$, $\Delta \in T_{\sigma}$, 
$\Delta ' \in \operatorname{Face}(\Delta)$ and $\Delta '' \in \operatorname{Face}(\Delta ')$. 
Then the following assertions hold. 

\begin{enumerate}
\item 
The generating function of $f_{B ( L_{\Delta, \Delta ', \Delta''} )}$ is given by 
\[
\sum _{i \ge 0} f_{B ( L_{\Delta, \Delta ', \Delta''} )}(i) t^i = 
\frac{Q(t)}{(1-t) \prod _{v \in V(\Delta '')} \bigl( 1 - t^{\operatorname{cpx}_{\sigma}(v)} \bigr)}, 
\]
where $Q(t)$ is a polynomial with $\deg Q \le C'_2 + \beta _{\Delta, \Delta '} + \sum _{v \in V(\Delta '')} \operatorname{cpx}_{\sigma}(v)$. 

\item 
The generating function of $f_{B ( L_{\Delta '} )}$ is given by 
\[
\sum _{i \ge 0} f_{B ( L_{\Delta '} )}(i) t^i = 
\frac{Q(t)}{(1-t) \prod _{v \in V(\Delta ')} \left( 1 - t^{\operatorname{cpx}_{\sigma}(v)} \right)}, 
\]
where $Q(t)$ is a polynomial with $\deg Q \le C'_2 + \beta _{\Delta, \Delta '} + \sum _{v \in V(\Delta ')} \operatorname{cpx}_{\sigma}(v)$. 

\item 
The generating function of $f_{B}$ is given by 
\[
\sum _{i \ge 0} f_{B}(i) t^i = 
\frac{Q(t)}{(1-t)R(t)}, 
\]
where 
\[
R(t) := \operatorname{LCM} \left\{ \prod _{v \in V(\Delta)} \bigl( 1 - t^{\operatorname{cpx}_{\sigma}(v)} \bigr) \ \middle | \ \sigma \in \operatorname{Facet}(P),\ \Delta \in T_{\sigma} \right \},  
\]
and $Q(t)$ is a polynomial with $\deg Q \le \beta + \deg R$. 
\end{enumerate}
\end{thm}
\begin{proof}
We prove (1). 
By Theorem \ref{thm:fingen}(1), $B \left( L_{\Delta, \Delta ', \Delta''} \right)$ is freely generated by 
the finite set $S \left(\overline{L}_{\Delta, \Delta ', \Delta''} \right)$ as an $M_{\sigma, \Delta ''}$-module. 
Note that $M_{\sigma, \Delta ''}$ is a free monoid with basis $(1,0)$ and 
$\bigl( \operatorname{cpx}_{\sigma}(v), \operatorname{cpx}_{\sigma}(v)v \bigr)$ for 
$v \in V(\Delta '')$. 
Therefore, we have 
\[
\sum _{i \ge 0} f_{B ( L_{\Delta, \Delta ', \Delta''} )}(i) t^i = 
\frac{Q(t)}{(1-t) \prod _{v \in V(\Delta '')} \bigl( 1 - t^{\operatorname{cpx}_{\sigma}(v)} \bigr)}, 
\]
where 
\[
Q(t) = \sum _{i \ge 0} f_{S ( \overline{L}_{\Delta, \Delta ', \Delta''} )}(i) t^i. 
\]
Here, by Theorem \ref{thm:fingen}(2), we have  
\begin{align*}
\deg Q(t) 
&= \max \left \{ d \ \middle | \ (d,y) \in S \left(\overline{L}_{\Delta, \Delta ', \Delta''} \right) \right \} \\
&\le C'_2 + \beta _{\Delta, \Delta '} + \sum _{v \in V(\Delta '')} \operatorname{cpx}_{\sigma}(v), 
\end{align*}
which completes the proof of (1). 

Since we have 
$L_{\Delta '} = \bigsqcup _{\Delta '' \in \operatorname{Face}(\Delta ')} L_{\Delta, \Delta ', \Delta''}$, (2) follows from (1). 

(3) follows from (2) by the inclusion-exclusion principle as detailed below.  
We fix a face $\tau$ of $\sigma$. We set
\[
\Xi _{\tau} := \{ \Delta ' \mid \Delta \in T_{\sigma}, \Delta ' \in \operatorname{Face}(\Delta), \Delta ' \subset \tau \}. 
\]
Then we have $\tau = \bigcup _{\Delta ' \in \Xi _{\tau}} \Delta '$, 
and furthermore, the set $\Xi _{\tau}$ is closed under taking intersection. 
Therefore, by (2), the generating function of $f_{B(\mathbb{R}_{\ge 0} \tau)}$ is given by the form
\[
\frac{Q(t)}{(1-t)R(t)}, 
\]
where $Q(t)$ is a polynomial with $\deg Q \le \beta + \deg R$. 
Since the set $\{ \tau \mid \tau \in \operatorname{Face}(P) \}$ is closed under taking intersection, 
we can conclude that the generating function of $f_B$ is also given by the same form. 
We complete the proof of (3). 
\end{proof}

Here, we will summarize our main theorem. 

\begin{cor}\label{cor:main}
Let $(\Gamma , L)$ be a strongly connected $n$-dimensional periodic graph, and let $x_0 \in V_{\Gamma}$. 
Let $(b(d))_d$ and $(s(d))_d$ be the cumulative growth sequence and the growth sequence of $\Gamma$ with the start point $x_0$ 
(see Subsection \ref{subsection:CS}). 
Let $\beta \in \mathbb{R}_{\ge 0}$ and $\operatorname{cpx}_{\Gamma} \in \mathbb{Z}_{>0}$ be as in Subsection \ref{subsection:inv}. 
Then, the following assertions hold. 
\begin{enumerate}
\item 
The generating function of $b: \mathbb{Z}_{\ge 0} \to \mathbb{Z}_{\ge 0}; d \mapsto b(d)$ is given by 
\[
\sum _{i \ge 0} b(i) t^i = 
\frac{Q(t)}{(1-t) ( 1 - t^{\operatorname{cpx}_{\Gamma}} )^{n}}, 
\]
where $Q(t)$ is a polynomial with $\deg Q \le \beta + n \cdot \operatorname{cpx}_{\Gamma}$. 
In particular, the function $b: \mathbb{Z}_{\ge 0} \to \mathbb{Z}_{\ge 0}; d \mapsto b(d)$ is a quasi-polynomial on $d > \beta -1$, and 
$\operatorname{cpx}_{\Gamma}$ is its quasi-period. 

\item 
The generating function of $s: \mathbb{Z}_{\ge 0} \to \mathbb{Z}_{\ge 0}; d \mapsto s(d)$ is given by 
\[
\sum _{i \ge 0} s(i) t^i = 
\frac{Q(t)}{(1 - t^{\operatorname{cpx}_{\Gamma}} )^{n}}, 
\]
where $Q(t)$ is a polynomial with $\deg Q \le \beta + n \cdot \operatorname{cpx}_{\Gamma}$. 
In particular, 
the function $s: \mathbb{Z}_{\ge 0} \to \mathbb{Z}_{\ge 0}; d \mapsto s(d)$ is a quasi-polynomial on $d > \beta$, and 
$\operatorname{cpx}_{\Gamma}$ is its quasi-period. 
\end{enumerate}
\end{cor}
\begin{proof}
Note that the polynomial $R(t)$ in Theorem \ref{thm:main}(3) divides $( 1 - t^{\operatorname{cpx}_{\Gamma}} )^{n}$. 
Therefore, 
(1) follows from Theorem \ref{thm:main}(3) since $b = f_B$. 
(2) follows from (1). 
\end{proof}

\begin{rmk}\label{rmk:NSMN}
The proof of Corollary \ref{cor:main} does not rely on the result in \cite{NSMN21}. 
Therefore, the proof of Corollary \ref{cor:main} also gives a different proof of Theorem \ref{thm:NSMN}. 
\end{rmk}

\begin{ex}
We saw in Example \ref{ex:inv} that $C'_2 = 3$, $\beta = 25$ and $\operatorname{cpx}_{\Gamma} = 2$ 
for the Wakatsuki graph $\Gamma$ with the start point $x_0 = v'_2$. 
Therefore, Corollary \ref{cor:main} shows that the growth sequence $(s_{\Gamma, x_0, i})_i$  is 
a quasi-polynomial on $i > \beta = 25$, and 
$\operatorname{cpx}_{\Gamma} = 2$ is its quasi-period. 
Actually, it is known that $(s_{\Gamma, x_0, i})_i$ is a quasi-polynomial on $i \ge 3$, and $2$ is its period (see \cite{IN}*{Example 2.18}). 
\end{ex}

As a corollary of Theorem \ref{thm:free}, we give an algorithm to compute the precise value of $C_{2}(\Gamma, \Phi, x_0)$. 

\begin{thm}\label{thm:C2}
For $\sigma \in \operatorname{Facet}(P)$ and $\Delta \in T_{\sigma}$, 
we define 
\[
L_{\Delta} ^{\le \beta'} 
:= \left \{ z \in L_{\Delta} \ \middle | \ \text{$\operatorname{proj}_{\Delta, v}(z) \le \beta' _{\Delta, v}$ holds for any $v \in V(\Delta)$} \right \}. 
\]
We also define 
\[
L ^{\le \beta'} := 
\bigcup _{\substack{\sigma \in \operatorname{Facet}(P),\\ \Delta \in T_{\sigma}}}
L_{\Delta} ^{\le \beta'}.  
\]
Then, we have 
\[
C_{2}(\Gamma, \Phi, x_0) = 
\max \left \{ d_{\Gamma}(x_0, y) - d_{P, \Phi}(x_0, y) \ \middle | \ 
y \in V_{\Gamma} \cap \Phi ^{-1} \left( L ^{\le \beta'} \right) \right \}. 
\]

In particular, we have 
\[
C_{2}(\Gamma, \Phi, x_0) = \max \left \{ d_{\Gamma}(x_0, y) - d_{P, \Phi}(x_0, y) \ \middle | \ 
y \in B_{\lfloor \beta ' \rfloor} \right \}, 
\]
where we define 
\[
\beta ' := C'_2 + \max 
\Biggl \{ \sum _{v \in V(\Delta)} \beta' _{\Delta, v} \ \Bigg| \ \sigma \in \operatorname{Facet}(P), \ \Delta \in T_{\sigma} \Biggr \}. 
\]
\end{thm}
\begin{proof}
Let $y \in V_{\Gamma}$. 
Take $\sigma \in \operatorname{Facet}(P)$ and $\Delta \in T_{\sigma}$ such that $\Phi(y) \in L_{\Delta}$.
For each $v \in V(\Delta)$, we define 
\[
b_v := \max 
\left\{ 0, 
\left \lceil  \frac{\operatorname{proj}_{\Delta, v}(\Phi(y)) - \beta' _{\Delta, v}}{\operatorname{cpx}_{\sigma}(v)} 
\right \rceil \right\}. 
\]
We define
\[
y' := y - \sum _{v \in V(\Delta)} b_v \operatorname{cpx}_{\sigma}(v) v. 
\]
Then, we have $\Phi(y') \in L _{\Delta} ^{\le \beta'}$. By Theorem \ref{thm:free}, we have 
\[
d_{\Gamma}(x_0, y') = d_{\Gamma}(x_0, y) - \sum _{v \in V(\Delta)} b_v \operatorname{cpx}_{\sigma}(v). 
\]
On the other hand, we have 
\[
d_{P, \Phi}(x_0, y') = d_{P, \Phi}(x_0, y) - \sum _{v \in V(\Delta)} b_v \operatorname{cpx}_{\sigma}(v). 
\]
Therefore, we have 
\[
d_{\Gamma}(x_0, y') - d_{P, \Phi}(x_0, y') = d_{\Gamma}(x_0, y) - d_{P, \Phi}(x_0, y). 
\]
Hence, we have 
\begin{align*}
C_{2}(\Gamma, \Phi, x_0) 
&= \sup \bigl \{ d_{\Gamma}(x_0, y) - d_{P, \Phi}(x_0, y) \ \big | \ y \in V_{\Gamma} \bigr\} \\
&= \max \left \{ d_{\Gamma}(x_0, y) - d_{P, \Phi}(x_0, y) \ \middle | \ 
y \in V_{\Gamma} \cap \Phi ^{-1} \left( L ^{\le \beta'} \right) \right \}, 
\end{align*}
which proves the first assertion. 

For $y \in V_{\Gamma}$ with $\Phi(y) \in L_{\Delta} ^{\le \beta'}$, we have 
\[
d_{\Gamma}(x_0, y) \le C'_2 + d_{P, \Phi}(x_0, y) \le C'_2 +  \sum _{v \in V(\Delta)} \beta' _{\Delta, v}. 
\]
Therefore, we have $V_{\Gamma} \cap \Phi ^{-1} \left( L ^{\le \beta'} \right) \subset B_{\lfloor \beta ' \rfloor}$, 
which proves the second assertion. 
\end{proof}

\section{Examples} \label{section:ex}

In this section, using Corollary \ref{cor:main}, we calculate the growth series for some specific periodic graphs. 
In Subsection \ref{subsection:tilings}, we list the invariants and the growth series of seven periodic tilings from \cite{GS19} using a computer program explained in Remark \ref{rmk:implement} below. 
In Subsection \ref{subsection:tilings}, we also examine a $6$-uniform tiling. 
To the best of our knowledge, this is the first time the growth series for this example has been determined with a proof. 
In Subsection \ref{subsection:ex1}, we treat another tiling called the $(3^6;3^2.4.3.4;3^2.4.3.4)$ 3-uniform tiling. 
Its growth series can be computed with the same computer program, but we also give the details of the computation to illustrate our algorithm. 
In Subsections \ref{subsection:ex2} and \ref{subsection:ex3}, we treat two $3$-dimensional periodic graphs obtained by carbon allotropes. 
As far as we know, this is the first time the growth series in these two examples have been determined with proofs.

\begin{rmk}\label{rmk:implement}
When $n=2$, it is not difficult to implement our algorithm to compute the invariants $\beta$ and $\operatorname{cpx}_{\Gamma}$ as follows (cf.\ Subsection \ref{subsection:2dim}): 

\begin{enumerate}
\item 
When $n=2$, each facet $\sigma$ of $P := P_{\Gamma}$ is one dimensional. 
Therefore, the points in $\operatorname{Im}(\nu) \cap \sigma$ give a triangulation $T_{\sigma}$ of $\sigma$ with the required conditions $(\diamondsuit)_1$ and $(\diamondsuit)_2$, and no further subdivision is necessary.

\item
For $\sigma \in \operatorname{Facet}(P)$, $\Delta \in T_{\sigma}$ and $v \in V(\Delta)$, 
the set $\mathcal{H}'' _{\sigma, v}$ can be concretely given (see  Subsection \ref{subsection:2dim}). 

\item 
For $\sigma \in \operatorname{Facet}(P)$, $\Delta \in T_{\sigma}$, $v \in V(\Delta)$ and $F \in \mathcal{H}'' _{\sigma, v}$, 
the invariants $\operatorname{cpx}_{\sigma} (v)$, $r^F _{\Delta, v}$ and $s^F _{\Delta, v}$ can be computed by the construction in the proof of Lemma \ref{lem:rs}. 

\item
For $\sigma \in \operatorname{Facet}(P)$, $\Delta \in T_{\sigma}$, $v \in V(\Delta)$ and $F \in \mathcal{H}'' _{\sigma, v}$, 
the invariants $a^F _{\Delta, v}(v)$ and $h^F _{\Delta, v}$ can be computed by the construction in Subsection \ref{subsection:2dim}. 
\end{enumerate}

Furthermore, the first few terms of the growth sequence can be computed by the breadth-first search algorithm. 
Therefore, it is not difficult to implement an algorithm to compute the growth series of two dimensional periodic graphs. 
\end{rmk}

\subsection{2-dimensional periodic graphs}\label{subsection:tilings}
In \cite{GS19}, Goodman-Strauss and Sloane determine the growth sequences for seven specific periodic tilings (Figures \ref{fig:cairo}-\ref{fig:snub632}). 
The parallelogram drawn with red lines represents a fundamental region of the corresponding periodic graphs $\Gamma$ (see \cite{GS19} and \cite{IN}*{Subsection 5.1} for more information). 
With the help of a computer program (Remark \ref{rmk:implement} and Appendix \ref{section:a}), we can automatically compute the invariants and their growth series as in Table \ref{table:tilings}. 

We also examine a $6$-uniform tiling illustlated in Figure \ref{fig:t6uni673}, which is the tiling \#673 in the Galebach list of 673 $6$-uniform tilings \cite{Gal} (see also \underline{A313961} in the OEIS \cite{OEIS}). 
The invariants and the growth series are computed as in Table \ref{table:tilings2}. 
From the form of these growth series, it can be determined that the actual (minimum) period is $36$ for all cases.
To the best of our knowledge, this is the first time the growth series for this example has been determined with a proof. 

\begin{figure}[htbp]
\begin{tabular}{cc}
\begin{minipage}{0.5\hsize}
\centering
\includegraphics[height=6cm]{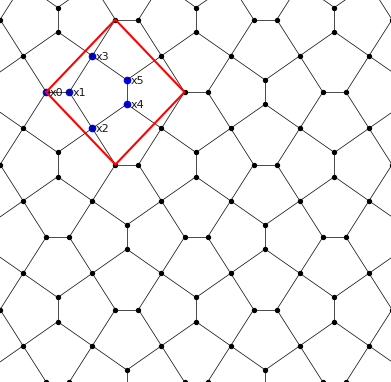}
\caption{The Cairo tiling.}
\label{fig:cairo}
\end{minipage}&
\begin{minipage}{0.5\hsize}
\centering
\includegraphics[height=6cm]{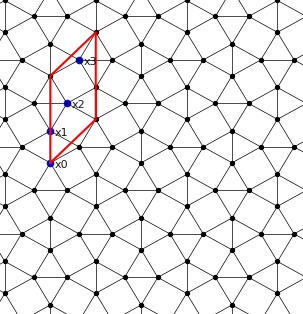}
\caption{The $3^2.4.3.4$ uniform tiling.}
\label{fig:32434}
\end{minipage}\\ 

\begin{minipage}{0.5\hsize}

\vspace{5mm}

\centering
\includegraphics[height=6cm]{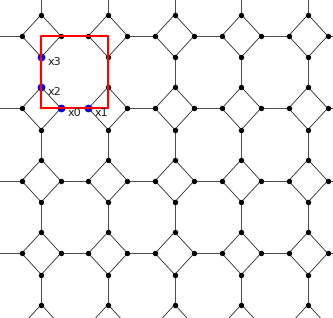}
\caption{The $4.8^2$ uniform tiling.}
\label{fig:482}
\end{minipage}&
\begin{minipage}{0.5\hsize}

\vspace{5mm}

\centering
\includegraphics[height=6cm]{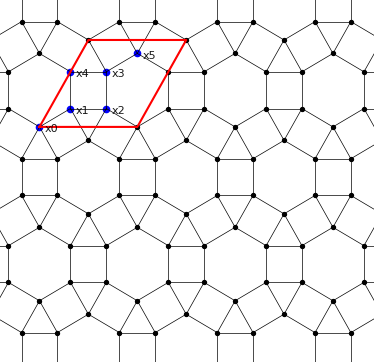}
\caption{The $3.4.6.4$ uniform tiling.}
\label{fig:3464}
\end{minipage}\\

\begin{minipage}{0.5\hsize}

\vspace{5mm}

\centering
\includegraphics[height=6cm]{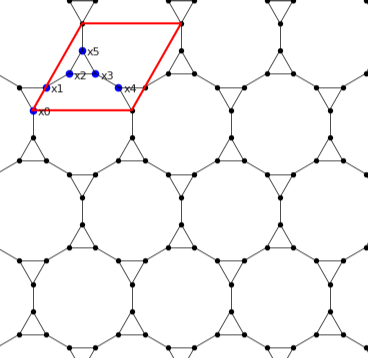}
\caption{The $3.12^2$ uniform tiling.}
\label{fig:3122}
\end{minipage}&
\begin{minipage}{0.5\hsize}

\vspace{5mm}

\centering
\includegraphics[height=6cm]{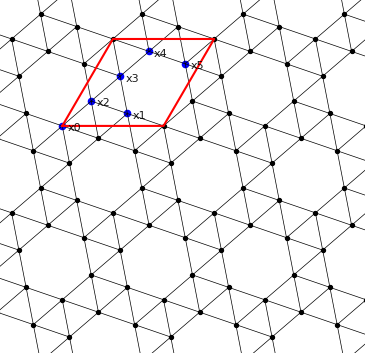}
\caption{The $3^4.6$ uniform tiling.}
\label{fig:346}
\end{minipage}\\
\end{tabular}
\end{figure}

\begin{figure}[htbp]
\begin{tabular}{cc}
\begin{minipage}{0.5\hsize}

\vspace{5mm}

\centering
\includegraphics[height=6cm]{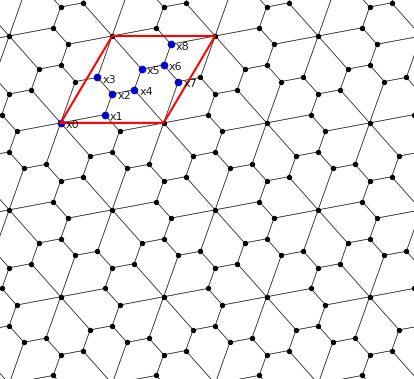}
\caption{The snub-632 3-uniform tiling.}
\label{fig:snub632}
\end{minipage}&
\begin{minipage}{0.5\hsize}

\vspace{5mm}

\centering
\includegraphics[height=6cm]{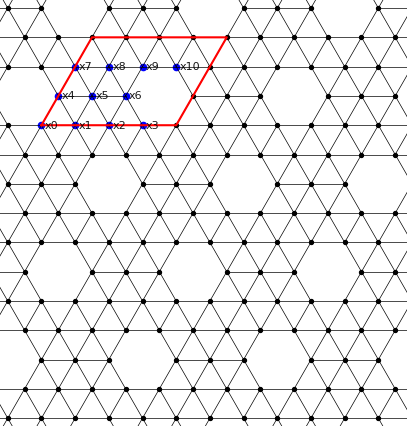}
\caption{The 6-uniform tiling (\#673).}
\label{fig:t6uni673}
\end{minipage}
\end{tabular}
\end{figure}

\renewcommand{\arraystretch}{1.5}
\begin{table}[htp]

\vspace{5mm}

\centering
\begin{tabular}{l||cccccc} \hline
\multirow{2}{*}{Tiling / start pt.} & $\#(V_{\Gamma /L})$ & $C_1$ & $C'_2$ & $\beta$ & $\operatorname{cpx}_{\Gamma}$  \\
& \multicolumn{5}{c}{Growth series} \\
\hline \hline
\multirow{2}{*}{Fig.\ \ref{fig:cairo} / $x_2, x_3$} & 6 & 1/3 & 7 & 72 & 12 \\
& \multicolumn{5}{c}{\Large $\frac{1+ 2t + t^2}{(1 - t)^2}$} \\
\hline
\multirow{2}{*}{Fig.\ \ref{fig:cairo} / otherwise} & 6 & 2/3 & 7 & 75 & 12 \\
& \multicolumn{5}{c}{\Large $\frac{1+2 t+5 t^{2}+4 t^{3}+2 t^{4}+3 t^{5}-t^{7}}{(1 - t)(1 - t^4)}$} \\
\hline
\multirow{2}{*}{Fig.\ \ref{fig:32434} / all} & 4 &  0.36... & 4 & 37.5... & 6 \\
& \multicolumn{5}{c}{\Large $\frac{1+4 t+6 t^{2}+4 t^{3}+t^{4}}{(1-t)(1-t^3)}$} \\
\hline
\multirow{2}{*}{Fig.\ \ref{fig:482} / all} & 4 & 0.58... & 4 & 37.1... & 12 \\
& \multicolumn{5}{c}{\Large $\frac{1+2t+2t^{2}+2t^{3}+t^{4}}{(1-t)(1-t^3)}$} \\
\hline
\multirow{2}{*}{Fig.\ \ref{fig:3464} / all} & 6 & 0.46... & 6 & 101.5... & 6 \\
& \multicolumn{5}{c}{\Large $\frac{1 + 2t + t^2}{(1-t)^2}$} \\
\hline
\multirow{2}{*}{Fig.\ \ref{fig:3122} / all} & 6 & 0.38... & 8 & 85.8... & 4 \\
& \multicolumn{5}{c}{\Large $\frac{1+3 t+4 t^{2}+6 t^{3}+6 t^{4}+6 t^{5}+6 t^{6}+3 t^{7}+3 t^{8}-2 t^{10}}{(1-t^4)^2}$} \\
\hline
\multirow{2}{*}{Fig.\ \ref{fig:346} / all} & 6 & 5/7 & 6 & 878/7 & 15 \\
& \multicolumn{5}{c}{\Large $\frac{1+4 t +4 t^{2}+6 t^{3}+4 t^{4}+4 t^{5}+t^{6}}{(1-t)(1-t^5)}$} \\
\hline
\multirow{2}{*}{Fig.\ \ref{fig:snub632} / $x_0$} & 9 & 3/7 & 10 & 626/7 & 3 \\
& \multicolumn{5}{c}{\Large $\frac{1+6 t+12 t^{2}+10 t^{3}+12 t^{4}+12 t^{5}+t^{6}}{(1 - t^3)^2}$} \\
\hline
\multirow{2}{*}{Fig.\ \ref{fig:snub632} / $x_2, x_6$} & 9 & 3/7 & 10 & 626/7 & 3 \\
& \multicolumn{5}{c}{\Large $\frac{1+3 t+6 t^{2}+13 t^{3}+15 t^{4}+6 t^{5}+4 t^{6}+9 t^{7}-3 t^{10}}{(1- t^3)^2}$} \\
\hline
\multirow{2}{*}{Fig.\ \ref{fig:snub632} / otherwise} & 9 & 6/7 & 10 & 650/7 & 3 \\
& \multicolumn{5}{c}{\Large $\frac{1+3 t+9 t^{2}+13 t^{3}+12 t^{4}+9 t^{5}+8 t^{6}+4 t^{7}-t^{8}-2 t^{9}-2 t^{10}}{(1 -t^3)^2}$} \\ 
\hline
\end{tabular}

\vspace{3mm}

\caption{Growth series of the seven tilings in \cite{GS19}.}
\label{table:tilings}
\end{table}
\renewcommand{\arraystretch}{1}

\renewcommand{\arraystretch}{1.8}
\begin{table}[htp]

\vspace{5mm}

\centering
\begin{sideways}
\begin{tabular}{l||cccccc} \hline
\multirow{2}{*}{Start} & OEIS\# & $C_1$ & $C'_2$ & $\beta$ & $\operatorname{cpx}_{\Gamma}$ & $\deg R$ (see Theorem \ref{thm:main}(3))  \\
& \multicolumn{5}{c}{Growth series} \\
\hline \hline
\multirow{2}{*}{$x_0, x_9$} & \underline{A314154} & 5/6 & 11 & 2648.4... & $2^3 \cdot 3^2 \cdot 5 \cdot 7 \cdot 11$ & 132 \\
& \multicolumn{6}{c}{\Large $\frac{1+5 t+10 t^{2}+10 t^{3}+5 t^{4}-4 t^{5}-9 t^{6}-12 t^{7}-6 t^{8}-6 t^{10}-12 t^{11}-9 t^{12}-4 t^{13}+5 t^{14}+10 t^{15}+10 t^{16}+5 t^{17}+t^{18}}{(1 - t^2)(1-t^3)(1-t^4)(1-t^9)}$} \\
\hline
\multirow{2}{*}{$x_1, x_8$} & \underline{A315346} & 3/4 & 11 & 2644.1... & $2^3 \cdot 3^2 \cdot 5 \cdot 7 \cdot 11$ & 132 \\
& \multicolumn{6}{c}{\Large $\frac{1+6 t+10 t^{2}+15 t^{3}+14 t^{4}+7 t^{5}-5 t^{6}-15 t^{7}-15 t^{8}-11 t^{9}-14 t^{10}-11 t^{11}-15 t^{12}-15 t^{13}-5 t^{14}+7 t^{15}+14 t^{16}+15 t^{17}+10 t^{18}+6 t^{19}+t^{20}}{(1 - t^3)(1 - t^4)^2(1-t^9)}$} \\
\hline
\multirow{2}{*}{$x_2, x_7$} & \underline{A315356} & 1/2 & 11 & 2631.3... & $2^3 \cdot 3^2 \cdot 5 \cdot 7 \cdot 11$ & 132 \\
& \multicolumn{6}{c}{\Large $\frac{1+6 t+9 t^{2}+10 t^{3}+6 t^{4}-7 t^{5}-9 t^{6}-10 t^{7}-6 t^{8}-6 t^{10}-10 t^{11}-9 t^{12}-7 t^{13}+6 t^{14}+10 t^{15}+9 t^{16}+6 t^{17}+t^{18}}{(1 - t^2)(1-t^3)(1-t^4)(1-t^9)}$} \\
\hline
\multirow{2}{*}{$x_3, x_{10}$} & \underline{A313961} & 7/12 & 11 & 2635.6... & $2^3 \cdot 3^2 \cdot 5 \cdot 7 \cdot 11$ & 132 \\
& \multicolumn{6}{c}{\Large $\frac{1+5 t+10 t^{2}+16 t^{3}+14 t^{4}+6 t^{5}-3 t^{6}-12 t^{7}-18 t^{8}-13 t^{9}-13 t^{10}-11 t^{11}-14 t^{12}-14 t^{13}-6 t^{14}+3 t^{15}+11 t^{16}+19 t^{17}+13 t^{18}+6 t^{19}+2 t^{20}-t^{21}-t^{22}}{(1 - t^3)(1 - t^4)^2(1-t^9)}$} \\
\hline
\multirow{2}{*}{$x_4, x_6$} & \underline{A314064} & 2/3 & 11 & 2639.8... & $2^3 \cdot 3^2 \cdot 5 \cdot 7 \cdot 11$ & 132 \\
& \multicolumn{6}{c}{\Large $\frac{1+5 t+10 t^{2}+9 t^{3}+4 t^{4}-t^{5}-10 t^{6}-9 t^{7}-5 t^{8}-8 t^{9}-4 t^{10}-8 t^{11}-10 t^{12}-3 t^{13}+2 t^{14}+8 t^{15}+11 t^{16}+6 t^{17}+3 t^{18}-t^{21}}{(1 - t^2)(1-t^3)(1-t^4)(1-t^9)}$} \\
\hline
\multirow{2}{*}{$x_5$} & \underline{A315238} & 5/12 & 11 & 2627.0... & $2^3 \cdot 3^2 \cdot 5 \cdot 7 \cdot 11$ & 132 \\
& \multicolumn{6}{c}{\Large $\frac{1+6 t+10 t^{2}+13 t^{3}+16 t^{4}+6 t^{5}-4 t^{6}-12 t^{7}-17 t^{8}-13 t^{9}-12 t^{10}-11 t^{11}-15 t^{12}-12 t^{13}-8 t^{14}+2 t^{15}+14 t^{16}+15 t^{17}+12 t^{18}+10 t^{19}+t^{20}-2 t^{22}}{(1 - t^3)(1 - t^4)^2(1-t^9)}$} \\
\hline
\end{tabular}
\end{sideways}

\vspace{3mm}

\caption{Growth series of the $6$-uniform tiling \#673 (Fig.\ \ref{fig:t6uni673}).}
\label{table:tilings2}
\end{table}
\renewcommand{\arraystretch}{1}

\subsection{The $(3^6;3^2.4.3.4;3^2.4.3.4)$ 3-uniform tiling}\label{subsection:ex1}

The tiling illustrated in Figure \ref{fig:3uni1} is called the $(3^6;3^2.4.3.4;3^2.4.3.4)$ 3-uniform tiling (see \cite{Cha89}). 
Let $(\Gamma, L)$ be the corresponding $2$-dimensional (unweighted and undirected) periodic graph, 
and let $\Phi$ be the periodic realization shown in the figure. 
The parallelogram drawn with red lines represents a fundamental region of the periodic graph $\Gamma$. 
We have $\#(V_{\Gamma}/L) = 13$. 
We will determine the growth series with the start point $x_0$ shown in Figure \ref{fig:3uni1}.

\begin{figure}[htbp]
\centering
\includegraphics[width=8cm]{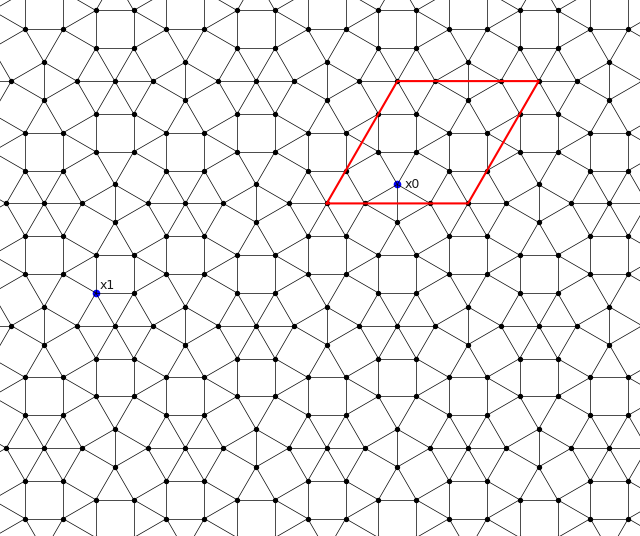}
\caption{The $(3^6;3^2.4.3.4;3^2.4.3.4)$ 3-uniform tiling.}
\label{fig:3uni1}
\end{figure}

With the help of a computer program, 
$\operatorname{Im}(\nu)$ can be computed as in Figure \ref{fig:3uniuvec}. 
The growth polytope $P := P_{\Gamma}$ is a dodecagon, whose vertices are shown in red. 
Note that the dihedral group $D_6$ acts on $\operatorname{Im}(\nu)$ by the rotations and the reflections. 

\begin{figure}[htbp]
\centering
\includegraphics[width=8cm]{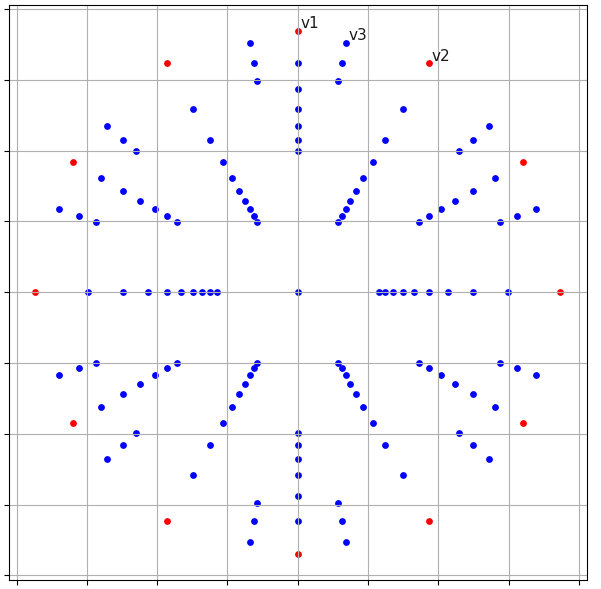}
\caption{$\operatorname{Im}(\nu)$.}
\label{fig:3uniuvec}
\end{figure}

With the help of a computer program, 
$C_1$ and $C'_2$ can be computed according to Proposition \ref{prop:C2} as follows:
\begin{align*}
C_1 &= d_{P, \Phi}(x_0, x_1) - d_{\Gamma}(x_0, x_1) = \left( 10 - \frac{\sqrt{3}}{3} \right) - 9 = 1 - \frac{\sqrt{3}}{3} = 0.4226...\ , \\
C'_2 &= d' (x_0, x_0) - d_{P, \Phi} (x_0, x_0) = 13 - 0 = 13. 
\end{align*}
Here, $x_1$ is a vertex shown in Figure \ref{fig:3uni1}. 

We define $v_1, v_2, v_3 \in \operatorname{Im}(\nu)$ as in Figure \ref{fig:3uniuvec}. 
Let $\sigma \in \operatorname{Facet}(P)$ denote the face of $P$ satisfying $V(\sigma) = \{ v_1, v_2 \}$. 
We take a triangulation $T_{\sigma} = \{ \Delta_1, \Delta_2 \}$ 
such that $V(\Delta _1) = \{ v_1, v_3 \}$ and $V(\Delta _2) = \{ v_2, v_3 \}$. 
In what follows, we shall calculate $\beta$ according to the way in Subsection \ref{subsection:2dim}. 

\begin{figure}[htbp]
\centering
\includegraphics[width=7cm]{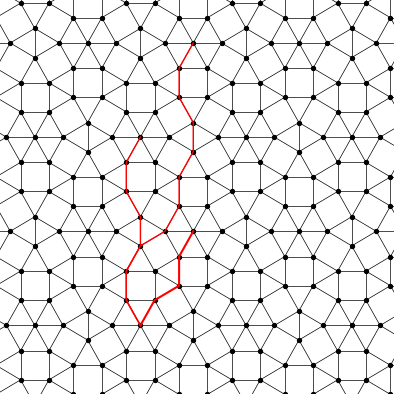}
\caption{Walks that give $u_1$, $u_2$ and $u_3$.}
\label{fig:3uni2}
\end{figure}

First, we shall calculate the invariants $s^F _{\sigma, v_i}$'s according to the construction in the proof of Lemma \ref{lem:r}. 
With the help of a computer program, we have 
\begin{align*}
\operatorname{Len}(\{ v_1 \}) &= \{ 7 \}, \quad
\operatorname{Len}(\{ v_2 \}) = \{ 4 \}, \quad
\operatorname{Len}(\{ v_3 \}) = \{ 11 \}, \\
\operatorname{num}(v_1,7) &=6, \quad
\operatorname{num}(v_2,4) =5, \quad
\operatorname{num}(v_3,11) =6. 
\end{align*}
Figure \ref{fig:3uni2} shows examples of walks in $\Gamma$ that give $v_1$, $v_2$ and $v_3$. 
Since we have $7 v_1 + 4 v_2 = 11 v_3$, 
we have 
\begin{align*}
\operatorname{cpx}_{\sigma}(v_1) &= m_{\sigma, v_1}(v_1) = m_{\sigma, v_3}(v_1) = 7, \\
\operatorname{cpx}_{\sigma}(v_2) &= m_{\sigma, v_2}(v_2) = m_{\sigma, v_3}(v_2) = 4, \\
\operatorname{cpx}_{\sigma}(v_3) &= m_{\sigma, v_3}(v_3) = 11. 
\end{align*}
Therefore, we have 
\[
s^{\{ v_1 \}}_{\sigma, v_1} = 42, \quad 
s^{\{ v_2 \}}_{\sigma, v_2} = 20, \quad
s^{\{ v_1, v_3 \}}_{\sigma, v_3} = 42+66=108, \quad
s^{\{ v_2, v_3 \}}_{\sigma, v_3} = 20+66=86. 
\]

\begin{table}[h]
  \centering

{\renewcommand\arraystretch{1.3}
\begin{tabular}{c||ccccc}
    \hline
    $v$ & $v_1$ & $v_2$ & \multicolumn{3}{c}{$v_3$} \\ 
    $u$ & $v_1$ & $v_2$ & $v_1$ & $v_2$ & $v_3$ \\
    \hline \hline
    $\operatorname{cpx}_{\sigma}(v)$ & $7$ & $4$ & \multicolumn{3}{c}{$11$} \\
	$m_{\sigma, v}(u)$ & $7$ & $4$ & $7$ & $4$ & $11$\\
    \hline
\end{tabular}
}

\vspace{3mm}

{\renewcommand\arraystretch{1.3}
\begin{tabular}{c||cccc}
    \hline
    $v$ & $v_1$ & $v_2$ & \multicolumn{2}{c}{$v_3$} \\ 
    $F$ & $\{ v_1 \}$ & $\{ v_2 \}$ & $\{ v_1 , v_3 \}$ & $\{ v_2, v_3 \}$ \\
    \hline \hline
    $s^F _{\sigma, v}$ & $42$ & $20$ & $108$ & $86$ \\
    \hline
\end{tabular}
}

\vspace{3mm}

{\renewcommand\arraystretch{1.3}
\begin{tabular}{c||ccc}
    \hline
    $v$ & $v_1$ & \multicolumn{2}{c}{$v_3$} \\ 
    $F$ & $\{ v_1 \}$ & $\{ v_1 , v_3 \}$ & $\{ v_2, v_3 \}$ \\
    \hline \hline
    $a^F _{\Delta _1, v}(v)$ & $\frac{21}{22}$ & $\frac{22}{23}$ & $\frac{22}{23}$ \\
    $h^F _{\Delta _1, v}$ & $\frac{21}{22}$ & $\frac{14}{15}$ & $\frac{8}{9}$ \\
    $\alpha ^{\prime F} _{\Delta _1, v}$ & $22 C_1 + 21 C'_2 + 54$ & $\frac{165}{7}C_1 + 22C'_2 + \frac{1320}{7}$ & $\frac{99}{4}C_1 + 22C'_2 + \frac{539}{2}$ \\
	$\alpha ' _{\Delta _1, v}$ & $22 C_1 + 21 C'_2 + 54$ & \multicolumn{2}{c}{$\frac{99}{4}C_1 + 22C'_2 + \frac{539}{2}$}\\
	$\beta_{\Delta _1, \Delta _1}$ & \multicolumn{3}{c}{$\frac{187}{4} C_1 + 43C'_2 + \frac{611}{2}$}\\
    \hline
\end{tabular}
}

\vspace{3mm}

{\renewcommand\arraystretch{1.3}
\begin{tabular}{c||ccc}
    \hline
    $v$ & $v_2$ & \multicolumn{2}{c}{$v_3$} \\ 
    $F$ & $\{ v_2 \}$ & $\{ v_1 , v_3 \}$ & $\{ v_2, v_3 \}$ \\
    \hline \hline
    $a^F _{\Delta _2, v}(v)$ & $\frac{10}{11}$ & $\frac{22}{23}$ & $\frac{22}{23}$ \\
    $h^F _{\Delta _2, v}$ & $\frac{10}{11}$ & $\frac{14}{15}$ & $\frac{8}{9}$ \\
    $\alpha ^{\prime F} _{\Delta _2, v}$ & $11 C_1 + 10 C'_2 + 32$ & $\frac{165}{7}C_1 + 22C'_2 + \frac{1320}{7}$ & $\frac{99}{4}C_1 + 22C'_2 + \frac{539}{2}$ \\
	$\alpha ' _{\Delta _2, v}$ & $11 C_1 + 10 C'_2 + 32$ & \multicolumn{2}{c}{$\frac{99}{4}C_1 + 22C'_2 + \frac{539}{2}$}\\
	$\beta_{\Delta _2, \Delta _2}$ & \multicolumn{3}{c}{$\frac{143}{4} C_1 + 32 C'_2 + \frac{573}{2}$}\\
    \hline
\end{tabular}
}

\vspace{4mm}
\caption{Invariants for $\Gamma$.}
\label{table:ah1}
\end{table}

Next, we have 
\begin{align*}
\operatorname{Im}(\nu) \setminus \{ v_1 \} &\subset H \left( \frac{21}{22}v_1, v_3 \right), \\
\operatorname{Im}(\nu) \setminus \{ v_2 \} &\subset H \left( \frac{10}{11}v_2, v_3 \right), \\
\operatorname{Im}(\nu) \setminus \{ v_1, v_3 \} &\subset H \left( \frac{14}{15}v_1, \frac{22}{23}v_3 \right) = H \left( v_2, \frac{22}{23}v_3 \right), \\
\operatorname{Im}(\nu) \setminus \{ v_2, v_3 \} &\subset H \left( v_1, \frac{22}{23}v_3 \right) = H \left( \frac{8}{9}v_2, \frac{22}{23}v_3 \right).
\end{align*}
Hence, the invariants $a_{\Delta, v}^F(v)$'s and $h_{\Delta, v}^F$'s are given as in Table \ref{table:ah1}. 
According to these $a_{\Delta, v}^F(v)$'s and $h_{\Delta, v}^F$'s, 
the invariants $\alpha ^{\prime F} _{\Delta , v}$'s, $\alpha ' _{\Delta , v}$'s and $\beta _{\Delta, \Delta}$'s are computed as in Table \ref{table:ah1}.

By symmetry, we have 
\begin{align*}
\beta 
= C'_2 + \max \{ \beta _{\Delta_1, \Delta_1}, \beta _{\Delta_2, \Delta_2} \}
= C'_2 + \beta _{\Delta_1, \Delta_1} 
= \frac{3697}{4} - \frac{187}{12} \sqrt{3} = 897.258...\ . 
\end{align*}

Therefore, by Corollary \ref{cor:main} (and Theorem \ref{thm:main}(3)), 
it follows that the growth series $G_{\Gamma, x_0}$ is of the form 
\[
G_{\Gamma, x_0}(t) = \frac{Q(t)}{(1-t^4)(1-t^7)(1-t^{11})}, 
\]
where $Q(t)$ is a polynomial of degree $\deg Q \le \beta + 22$. 

With the help of a computer program (breadth-first search algorithm), 
the first $\lfloor \beta \rfloor + 22 + 1$ ($= 920$) terms of the growth sequence $(s_{\Gamma, x_0, i})_{i \ge 0}$ can be computed. 
Using them, we can calculate $G_{\Gamma, x_0}(t)$ as follows:  
\begin{align*}
G_{\Gamma, x_0}(t) 
&=  \frac{\bigl( \text{The terms of $(1-t^4)(1-t^7)(1-t^{11}) \sum _{i = 0}^{919}  s_{\Gamma, x_0, i} t^i$ of degree $919$ or less.} \bigr)}{(1-t^4)(1-t^7)(1-t^{11})} \\
&= \frac{1+4 t+8 t^{2}+9 t^{3}+13 t^{4}+8 t^{5}+13 t^{6}+9 t^{7}+8 t^{8}+4 t^{9}+t^{10}}{1-t+t^{2}-t^{3}-t^{7}+t^{8}-t^{9}+t^{10}} \\
&= \frac{1+5 t+12 t^{2}+17 t^{3}+22 t^{4}+21 t^{5}+21 t^{6}+22 t^{7}+17 t^{8}+12 t^{9}+5 t^{10}+t^{11}}{(1-t^4)(1-t^7)}.
\end{align*}
By the form of $G_{\Gamma, x_0}(t)$, 
we can conclude that the growth sequence $(s_{\Gamma, x_0, i})_{i \ge 0}$ is a quasi-polynomial on $i \ge 1$, 
and $28$ is its period. 

\begin{rmk}
\begin{enumerate}
\item 
In \cite{IN}*{Subsection 5.2}, we obtain the same formula for $G_{\Gamma, x_0}$, but by a different method. 

\item 
By Theorem \ref{thm:C2}, we can obtain $C_2 = 0.4226...\ $. 
\end{enumerate}
\end{rmk}

\subsection{The carbon allotrope K6} \label{subsection:ex2}
In this subsection, we consider the $3$-dimensional periodic graph $\Gamma$ shown in Figure \ref{fig:sacada12_1}. 
This graph corresponds to a carbon allotrope called the K6 carbon (\#12 in {\sf SACADA} database). 
In {\sf SACADA} database \cite{SACADA}, the fundamental region of the K6 carbon is taken as shown in Figure \ref{fig:sacada12_1}. 
In what follows, we will proceed with this fundamental region, 
although it is possible to replace it with a smaller fundamental region. 
Then, we have $\#(V_{\Gamma}/L) = 12$. 
Note that all vertices of $\Gamma$ are symmetric, and hence, 
the growth sequence does not depend on the choice of its start point $x_0$. 

\begin{figure}[htbp]
\centering
\includegraphics[width=8cm]{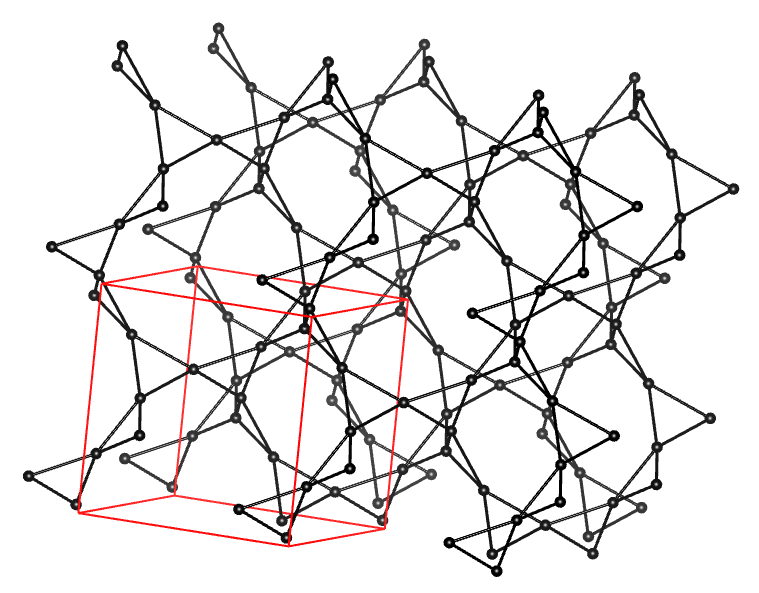}
\caption{The K6 carbon.}
\label{fig:sacada12_1}
\end{figure}

With the help of a computer program, 
$\operatorname{Im}(\nu)$ can be computed as in Figure \ref{fig:sacada12uvec}. 
The growth polytope $P := P_{\Gamma}$ has $14$ vertices (marked in red) and $24$ facets, and each facet is a triangle as in Figure \ref{fig:sacada12p}. 
Note that $\operatorname{Im}(\nu)$ has a symmetry such that the $24$ facets of $P$ are symmetric. 

\begin{figure}[htbp]
\centering
\includegraphics[width=9cm]{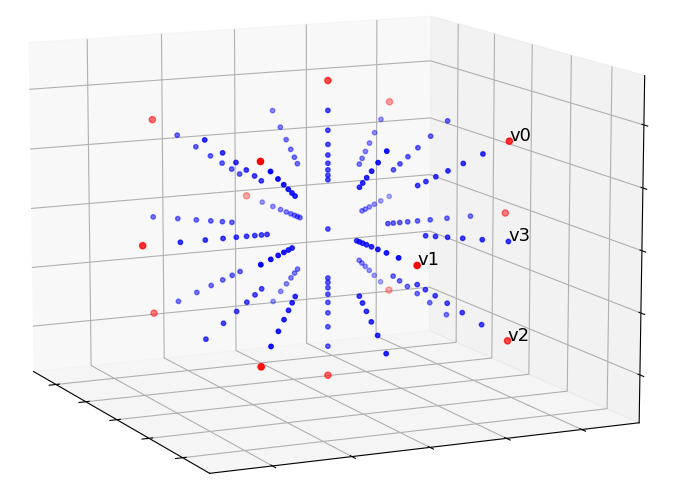}
\caption{$\operatorname{Im}(\nu)$.}
\label{fig:sacada12uvec}
\end{figure}

\begin{figure}[htbp]
\centering
\includegraphics[width=7cm]{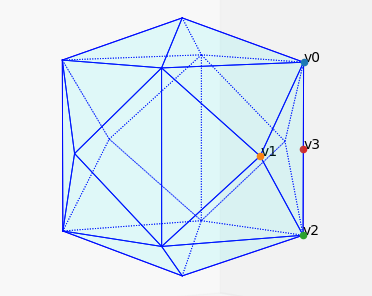}
\caption{$P$.}
\label{fig:sacada12p}
\end{figure}

With the help of a computer program, 
$C_1$ and $C'_2$ can be computed according to Proposition \ref{prop:C2} as follows:
\[
C_1 = 0.5 , \qquad C'_2 = 13. 
\]

We define $v_0, v_1, v_2, v_3 \in \operatorname{Im}(\nu)$ as in Figure \ref{fig:sacada12uvec}. 
Let $\sigma \in \operatorname{Facet}(P)$ denote the face of $P$ satisfying $V(\sigma) = \{ v_0, v_1, v_2 \}$. 
Then, we have $\operatorname{Im}(\nu) \cap \sigma = \{ v_0, v_1, v_2, v_3 \}$. 
We take a triangulation $T_{\sigma} = \{ \Delta_1, \Delta_2 \}$ 
such that $V(\Delta _1) = \{ v_1, v_2, v_3 \}$ and $V(\Delta _2) = \{ v_0, v_1, v_3 \}$. 
Then, $T_{\sigma}$ satisfies the required conditions $(\diamondsuit)_1$ and $(\diamondsuit)_2$. 
Note that $v_3$ is the midpoint of $v_2$ and $v_0$. 
Furthermore, $\operatorname{Im}(\nu)$ has a symmetry such that the three points $v_1$, $v_2$ and $v_3$ 
are translated to $v_1$, $v_0$ and $v_3$. 
In what follows, we shall compute the invariants for $\Delta _1$. 

First, we shall calculate the invariants $s^F _{\sigma, v_i}$'s according to the construction in the proof of Lemma \ref{lem:r}. 
We have 
\[
\mathcal{H}''_{\sigma, v_1} = \{ \{ v_1 \} \}, \quad
\mathcal{H}''_{\sigma, v_2} = \{ \{ v_2 \} \}, \quad
\mathcal{H}''_{\sigma, v_3} = \{ \{ v_0, v_3 \}, \{ v_2, v_3 \} \}. 
\]
With the help of a computer program, we have 
\begin{align*}
\operatorname{Len}(\{ v_1 \}) &= \{ 4 \}, \quad
\operatorname{Len}(\{ v_0 \}) = \operatorname{Len}(\{ v_2 \}) = \operatorname{Len}(\{ v_3 \}) = \{ 6 \}, \\
\operatorname{num}(v_0,6) &= \operatorname{num}(v_2,6)=1, \quad
\operatorname{num}(v_1,4) = \operatorname{num}(v_3,6)=2. 
\end{align*}
Since $v_3$ is the midpoint of $v_2$ and $v_0$, we have 
\begin{align*}
\operatorname{cpx}_{\sigma}(v_1) &= m_{\sigma, v_1}(v_1) = 4, \quad
\operatorname{cpx}_{\sigma}(v_2) = m_{\sigma, v_2}(v_2) = 6, \\
\operatorname{cpx}_{\sigma}(v_3) &= m_{\sigma, v_3}(v_3) = 12, \quad m_{\sigma, v_3}(v_0) = m_{\sigma, v_3}(v_2) = 6. 
\end{align*}
Therefore, we have 
\[
s^{\{ v_1 \}}_{\sigma, v_1} = 8, \quad 
s^{\{ v_2 \}}_{\sigma, v_2} = 6, \quad
s^{\{ v_0, v_3 \}}_{\sigma, v_3} = s^{\{ v_2, v_3 \}}_{\sigma, v_3} = 6 + 18 = 24. 
\]

Next, we have 
\begin{align*}
\operatorname{Im}(\nu) \setminus \{ v_1 \} &\subset H \left( \frac{4}{5} v_1 , v_2, v_3 \right), \\
\operatorname{Im}(\nu) \setminus \{ v_2 \} &\subset H \left( v_1 , \frac{6}{7} v_2, v_3 \right), \\
\operatorname{Im}(\nu) \setminus \{ v_0, v_3 \} &\subset H \left( v_1 , v_2, \frac{6}{7}v_0 \right) = H \left( v_1 , v_2, \frac{12}{13} v_3 \right), \\
\operatorname{Im}(\nu) \setminus \{ v_2, v_3 \} &\subset H \left( v_0 , v_1, \frac{6}{7}v_2 \right) = H \left( v_1 , \frac{6}{7} v_2, \frac{12}{13} v_3 \right). 
\end{align*}
Therefore, we can take the invariants $a_{\Delta _1, v}^F (v)$'s and $h_{\Delta _1, v}^F$'s as in Table \ref{table:ah2}. 
According to these $a_{\Delta _1, v}^F (v)$'s and $h_{\Delta _1, v}^F$'s, 
the invariants $\alpha^{\prime F} _{\Delta _1, v}$'s and $\alpha' _{\Delta _1, v}$'s are computed as in Table \ref{table:ah2}. 
\begin{table}[h]
  \centering
{\renewcommand\arraystretch{1.3}
\begin{tabular}{c||ccccc}
    \hline
    $v$ & $v_1$ & $v_2$ & \multicolumn{3}{c}{$v_3$} \\ 
    $u$ & $v_1$ & $v_2$ & $v_0$ & $v_2$ & $v_3$ \\
    \hline \hline
    $\operatorname{cpx}_{\sigma} (v)$ & $4$ & $6$ & \multicolumn{3}{c}{$12$} \\
    $m_{\sigma, v} (u)$ & $4$ & $6$ & $6$ & $6$ & $12$ \\ 
    \hline
\end{tabular}
}

\vspace{2mm}

{\renewcommand\arraystretch{1.3}
\begin{tabular}{c||cccc}
    \hline
    $v$ & $v_1$ & $v_2$ & \multicolumn{2}{c}{$v_3$} \\ 
    $F$ & $\{ v_1 \}$ & $\{ v_2 \}$ & $\{ v_0, v_3 \}$ & $\{ v_2, v_3 \}$ \\
    \hline \hline
    $s_{\sigma , v}^F (v)$ & $8$ & $6$ & $24$ & $24$ \\
    $a_{\Delta _1, v}^F (v)$ & $\frac{4}{5}$ & $\frac{6}{7}$ & $\frac{12}{13}$ & $\frac{12}{13}$ \\ 
    $h_{\Delta _1, v}^F$ & $\frac{4}{5}$ & $\frac{6}{7}$ & $\frac{6}{7}$ & $\frac{6}{7}$ \\ 
    $\alpha^{\prime F} _{\Delta _1, v}$ & $5 C_1 + 4 C'_2 + 19$ & $7 C_1 + 6 C'_2 + 17$ & $14C_1 + 12 C'_2 + 70$ & $14C_1 + 12 C'_2 + 70$ \\ 
    $\alpha' _{\Delta _1, v}$ & $5 C_1 + 4 C'_2 + 19$ & $7 C_1 + 6 C'_2 + 17$ & \multicolumn{2}{c}{$14C_1 + 12 C'_2 + 70$} \\
    $\beta _{\Delta_1, \Delta_1}$ & \multicolumn{4}{c}{$26 C_1 + 22 C'_2 + 84$} \\
    $\beta$ & \multicolumn{4}{c}{$26 C_1 + 23 C'_2 + 84 = 396$}\\
    \hline
\end{tabular}
}
\vspace{4mm}
\caption{Invariants for $\Gamma$.}
\label{table:ah2}
\end{table}

Hence, we have 
\[
\beta _{\Delta_1, \Delta_1} 
= \sum _{i = 1,2,3} \bigl( \alpha ' _{\Delta _1, v_i} - \operatorname{cpx}_{\sigma}(v_i) \bigr) 
= 26 C_1 + 22 C'_2 + 84.
\]
By symmetry, we have 
\begin{align*}
\beta 
= C'_2 + \beta _{\Delta_1, \Delta_1}
= 26 C_1 + 23 C'_2 + 84
= 396. 
\end{align*}

Therefore, by Corollary \ref{cor:main} (and Theorem \ref{thm:main}(3)), 
it follows that the growth series $G_{\Gamma, x_0}$ is of the form 
\[
G_{\Gamma, x_0}(t) = \frac{Q(t)}{(1-t^4)(1-t^6)(1-t^{12})}, 
\]
where $Q(t)$ is a polynomial of degree $\deg Q \le \beta + 22$. 

With the help of a computer program (breadth-first search algorithm), 
the first $\beta + 22 + 1$ ($= 419$) terms of the growth sequence $(s_{\Gamma, x_0, i})_{i \ge 0}$ can be computed. 
Using them, we can calculate $G_{\Gamma, x_0}(t)$ as follows:  
\small
\begin{align*}
G_{\Gamma, x_0}(t) 
&=  \frac{\bigl( \text{The terms of $(1-t^4)(1-t^6)(1-t^{12}) \sum _{i = 0}^{418}  s_{\Gamma, x_0, i} t^i $ of degree $418$ or less.} \bigr)}{(1-t^4)(1-t^6)(1-t^{12})} \\
&= \frac{1 + 4t + 8t^2 + 14t^3 + 23t^4 + 34t^5 + 31t^6 + 28 t^7 + 4t^8 - 4t^9 +t^{10} -8t^{11} + 8t^{12}}
{(1-t^3)^2(1-t^4)}.
\end{align*}
\normalsize
By the form of $G_{\Gamma, x_0}(t)$, 
we can conclude that the growth sequence $(s_{\Gamma, x_0, i})_{i \ge 0}$ is a quasi-polynomial on $i \ge 3$, 
and $12$ is its period. 

\begin{rmk}
By Theorem \ref{thm:C2}, we can obtain $C_2 = 1.25$. 
\end{rmk}

\subsection{The carbon allotrope {\sf CFS}} \label{subsection:ex3}
In this subsection, we consider the $3$-dimensional periodic graph $\Gamma$ shown in Figure \ref{fig:sacada29_1}. 
This graph corresponds to a carbon allotrope called {\sf CFS} (\#29 in {\sf SACADA} database). 
In {\sf SACADA} database, the fundamental region of {\sf CFS} is taken as shown in Figure \ref{fig:sacada29_1}. 
In what follows, we will proceed with this fundamental region. 
Then, we have $\#(V_{\Gamma}/L) = 6$. 
Note that all vertices of $\Gamma$ are symmetric, and hence, 
the growth sequence does not depend on the choice of its start point $x_0$. 

\begin{figure}[htbp]
\centering
\includegraphics[width=9cm]{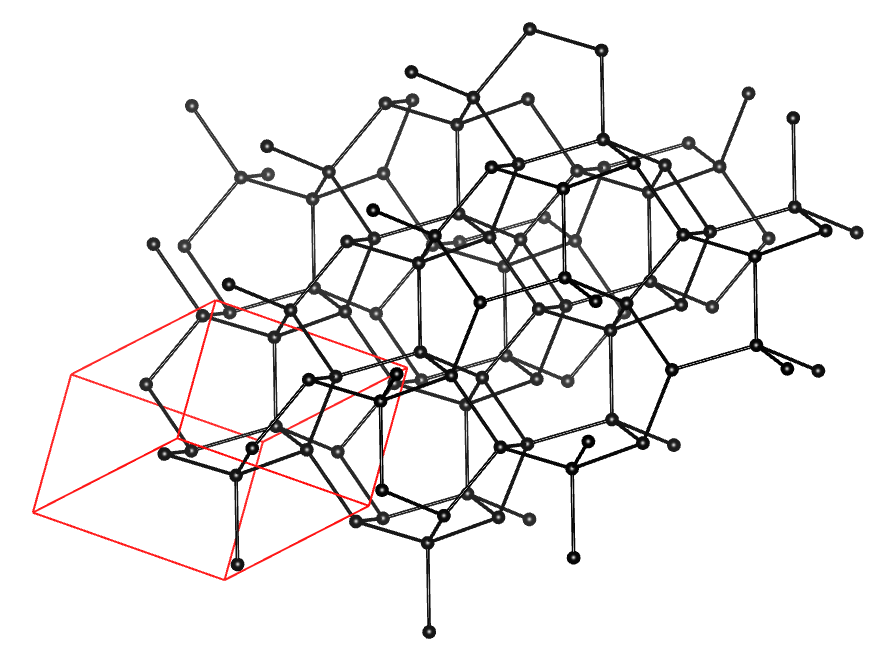}
\caption{The carbon allotrope {\sf CFS}.}
\label{fig:sacada29_1}
\end{figure}

With the help of a computer program, 
$\operatorname{Im}(\nu)$ can be computed as in Figure \ref{fig:sacada29uvec}. 
The growth polytope $P := P_{\Gamma}$ has $32$ vertices (marked in red) and $54$ facets ($48$ triangles and $6$ quadrilaterals) as in Figure \ref{fig:sacada29p}. 

\begin{figure}[htbp]
\centering
\includegraphics[width=9cm]{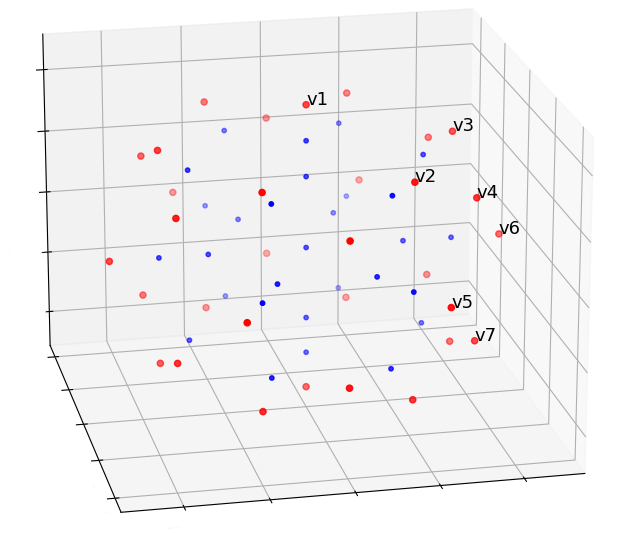}
\caption{$\operatorname{Im}(\nu)$.}
\label{fig:sacada29uvec}
\end{figure}

\begin{figure}[htbp]
\centering
\includegraphics[width=7cm]{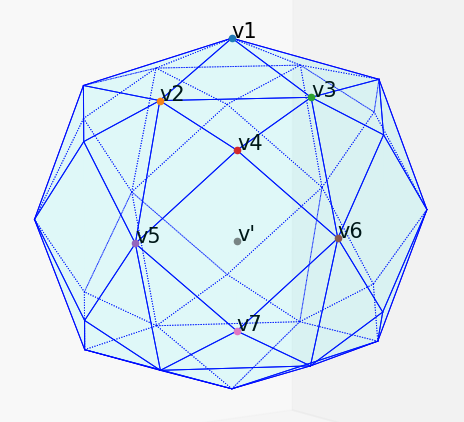}
\caption{$P$.}
\label{fig:sacada29p}
\end{figure}

With the help of a computer program, 
$C_1$ and $C'_2$ can be computed according to Proposition \ref{prop:C2} as follows:
\[
C_1 = 0.6 , \qquad C'_2 = 7. 
\]

We define $v_1, v_2, \ldots, v_7 \in \operatorname{Im}(\nu)$ as in Figure \ref{fig:sacada29uvec}. 
Let $\sigma_1, \ldots , \sigma_4 \in \operatorname{Facet}(P)$ denote the face of $P$ satisfying 
\begin{align*}
V(\sigma _1) &= \{ v_1, v_2, v_3 \}, \quad
V(\sigma _2) = \{ v_2, v_3, v_4 \}, \\
V(\sigma _3) &= \{ v_2, v_4, v_5 \}, \quad
V(\sigma _4) = \{ v_4, v_5, v_6, v_7 \}. 
\end{align*}
Then, for each $i$, we have $\operatorname{Im}(\nu) \cap \sigma _i = V(\sigma _i)$. 
For each $i = 1,2,3$, $\sigma _i$ is a simplex, and hence 
$\sigma _i$ has the trivial triangulation $T_{\sigma _i} = \{ \sigma _i \}$. 
The facet $\sigma_4$ is a rhombus. 
Let $v'$ be the intersection point of the diagonals $v_4 v_7$ and $v_5 v_6$. 
We take a triangulation $T_{\sigma _4} = \{ \Delta_1, \Delta_2, \Delta_3, \Delta _4 \}$ 
such that 
\begin{align*}
V(\Delta _1) &= \{ v', v_4, v_5 \}, \quad 
V(\Delta _2) = \{ v', v_7, v_5 \}, \\ 
V(\Delta _3) &= \{ v', v_4, v_6 \}, \quad 
V(\Delta _4) = \{ v', v_7, v_6 \}. 
\end{align*}
Note that $\operatorname{Im}(\nu)$ has a symmetry such that 
any facet of $P$ can be translated to either $\sigma _1$, $\sigma _2$, $\sigma _3$ or $\sigma _4$. 
Furthermore, $\operatorname{Im}(\nu)$ also has a symmetry such that 
$\Delta _1$ can be translated to $\Delta _2$, $\Delta _3$ or $\Delta _4$. 
In what follows, we shall calculate the invariants for the simplices $\sigma_1$, $\sigma_2$, $\sigma_3$ and $\Delta _1$. 

First, we shall compute the invariants $s^F _{\sigma, v}$'s according to the construction in the proof of Lemma \ref{lem:r}. 
For each $i = 1, \ldots , 4$, if $v \in V(\sigma _i)$, we have $\mathcal{H}''_{\sigma _i, v} = \{\{ v \}\}$. 
We also have 
\[
\mathcal{H}''_{\sigma _4, v'} = \{ \{ v_4, v_5 \}, \{ v_5, v_7 \}, \{ v_7, v_6 \}, \{ v_6, v_4 \} \}. 
\]
With the help of a computer program, we have 
\begin{align*}
\operatorname{Len}(\{ v_1 \}) &= 
\operatorname{Len}(\{ v_5 \}) = 
\operatorname{Len}(\{ v_6 \}) =  \{ 3 \}, \\
\operatorname{Len}(\{ v_2 \}) &= 
\operatorname{Len}(\{ v_3 \}) = \{ 4 \}, \quad
\operatorname{Len}(\{ v_4 \}) = 
\operatorname{Len}(\{ v_7 \}) = \{ 6 \}, \\ 
\operatorname{num}(v_1, 3) &= 
\operatorname{num}(v_5, 3) = 
\operatorname{num}(v_6, 3) = 2 \\
\operatorname{num}(v_2, 4) &= 
\operatorname{num}(v_3, 4) = 
\operatorname{num}(v_4, 6) = 
\operatorname{num}(v_7, 6) = 1. 
\end{align*}
For each $i = 1, \ldots , 4$, if $v \in V(\sigma _i)$, 
we have $m_{\sigma _i , v } (v) = \operatorname{cpx}_{\sigma _i}(v)$. 
Here, $\operatorname{cpx}_{\sigma _i}(v)$ is determined as follows: 
\begin{align*}
\operatorname{cpx}_{\sigma _1}(v_1) &= 
\operatorname{cpx}_{\sigma _3}(v_5) = 
\operatorname{cpx}_{\sigma _4}(v_5) = 
\operatorname{cpx}_{\sigma _4}(v_6) = 3, \\
\operatorname{cpx}_{\sigma _1}(v_2) &= 
\operatorname{cpx}_{\sigma _2}(v_2) = 
\operatorname{cpx}_{\sigma _3}(v_2) = 
\operatorname{cpx}_{\sigma _1}(v_3) = 
\operatorname{cpx}_{\sigma _2}(v_3) = 4, \\
\operatorname{cpx}_{\sigma _2}(v_4) &= 
\operatorname{cpx}_{\sigma _3}(v_4) = 
\operatorname{cpx}_{\sigma _4}(v_4) = 
\operatorname{cpx}_{\sigma _4}(v_7) = 6. 
\end{align*}
Since $v'$ is the center of the rhombus $\sigma _4$, we have 
\[
\operatorname{cpx}_{\sigma _4}(v')=12, \quad 
m_{\sigma _4 , v'} (v_4) = m_{\sigma _4 , v'} (v_5) = m_{\sigma _4 , v'} (v_6) = m_{\sigma _4 , v'} (v_7) = 6. 
\]
Therefore, the invariants $s_{\sigma, v} ^F$'s are computed as in Table \ref{table:ah3}. 

\begin{table}[h]
  \centering

{\renewcommand\arraystretch{1.3}
\begin{tabular}{c||ccc}
    \hline
    $v$ & $v_1$ & $v_2$ & $v_3$ \\ 
    \hline \hline
    $\operatorname{cpx}_{\sigma _1} (v) = m_{\sigma _1, v} (v)$ & $3$ & $4$ & $4$ \\
   	$s_{\sigma _1} ^{\{ v \}}(v)$ & $6$ & $4$ & $4$ \\
	$a_{\sigma _1 , v}^{\{ v \}} (v) = h_{\sigma _1, v}^{\{ v \}}$ & $\frac{3}{4}$ & $\frac{4}{5}$ & $\frac{4}{5}$ \\
	$\alpha^{\prime \{ v \}} _{\sigma _1, v} = \alpha' _{\sigma _1, v}$ & $4C_1 + 3C'_2 + 11$ & $5C_1 + 4C'_2 + 9$ & $5C_1 + 4C'_2 + 9$ \\
	$\beta _{\sigma _1, \sigma _1}$ & \multicolumn{3}{c}{$14C_1 + 11C'_2 + 18$} \\
    \hline
\end{tabular}
}

\vspace{3mm}

{\renewcommand\arraystretch{1.3}
\begin{tabular}{c||ccc}
    \hline
    $v$ & $v_2$ & $v_3$ & $v_4$ \\ 
    \hline \hline
    $\operatorname{cpx}_{\sigma _2} (v) = m_{\sigma _2, v} (v)$ & $4$ & $4$ & $6$ \\
   	$s_{\sigma _2} ^{\{ v \}}(v)$ & $4$ & $4$ & $6$ \\
	$a_{\sigma _2, v}^{\{ v \}} (v) = h_{\sigma _2, v}^{\{ v \}}$ & $\frac{4}{5}$ & $\frac{4}{5}$ & $\frac{6}{7}$ \\
	$\alpha^{\prime \{ v \}} _{\sigma _2, v} = \alpha' _{\sigma _2, v}$ & $5C_1 + 4C'_2 + 9$ & $5C_1 + 4C'_2 + 9$ & $7C_1 + 6C'_2 + 11$ \\
	$\beta _{\sigma _2, \sigma _2}$ & \multicolumn{3}{c}{$17C_1 + 14C'_2 + 15$} \\
    \hline
\end{tabular}
}

\vspace{3mm}

{\renewcommand\arraystretch{1.3}
\begin{tabular}{c||ccc}
    \hline
    $v$ & $v_2$ & $v_4$ & $v_5$ \\ 
    \hline \hline
    $\operatorname{cpx}_{\sigma _3} (v) = m_{\sigma _3, v} (v)$ & $4$ & $6$ & $3$ \\
   	$s_{\sigma _3} ^{\{ v \}}(v)$ & $4$ & $6$ & $6$ \\
	$a_{\sigma _3, v}^{\{ v \}} (v) = h_{\sigma _3, v}^{\{ v \}}$ & $\frac{8}{9}$ & $\frac{6}{7}$ & $\frac{3}{4}$ \\
	$\alpha^{\prime \{ v \}} _{\sigma _3, v} = \alpha' _{\sigma _3, v}$ & $9C_1 + 8C'_2 + 9$ & $7C_1 + 6C'_2 + 11$ & $4C_1 + 3C'_2 + 11$ \\
	$\beta _{\sigma _3, \sigma _3}$ & \multicolumn{3}{c}{$20C_1 + 17C'_2 + 18$} \\
    \hline
\end{tabular}
}

\vspace{3mm}

{\renewcommand\arraystretch{1.3}
\begin{tabular}{c||cccccccc}
    \hline
	$v$ & $v_4$ & $v_5$ & $v_6$ & $v_7$ & \multicolumn{4}{c}{$v'$} \\
	$u$ & $v_4$ & $v_5$ & $v_6$ & $v_7$ & $v_4$ & $v_5$ & $v_6$ & $v_7$ \\
    \hline \hline
    $\operatorname{cpx}_{\sigma _4} (v)$ & $6$ & $3$ & $3$ & $6$ & \multicolumn{4}{c}{$12$} \\
	$m_{\sigma _4, v} (u)$ & $6$ & $3$ & $3$ & $6$ & $6$ & $6$ & $6$ & $6$ \\
    \hline
\end{tabular}
}

\vspace{3mm}

{\renewcommand\arraystretch{1.3}
\begin{tabular}{c||cccccc}
    \hline
    $v$ & $v_4$ & $v_5$ & \multicolumn{4}{c}{$v'$} \\ 
    $F$ & $\{ v_4 \}$ & $\{ v_5 \}$ & $\{ v_4, v_5 \}$ & $\{ v_5, v_7 \}$ & $\{ v_7, v_6 \}$ & $\{ v_6, v_4 \}$ \\
    \hline \hline
    $s_{\sigma _4, v}^F (v)$ & $6$ & $6$ & $15$ & $15$ & $15$ & $15$ \\
    $a_{\Delta _1, v}^F (v)$ & $\frac{6}{7}$ & $\frac{3}{4}$ & $\frac{12}{13}$ & $\frac{12}{13}$ & $\frac{12}{13}$ & $\frac{12}{13}$ \\ 
    $h_{\Delta _1, v}^F$ & $\frac{6}{7}$ & $\frac{3}{4}$ & $\frac{6}{7}$ & $\frac{6}{7}$ & $\frac{6}{7}$ & $\frac{6}{7}$ \\ 
    $\alpha' _{\Delta _1, v}$ & $7 C_1 + 6 C'_2 + 11$ & $4 C_1 + 3 C'_2 + 11$ & \multicolumn{4}{c}{$14C_1 + 12 C'_2 + 40$} \\
    $\beta _{\Delta_1, \Delta_1}$ & \multicolumn{6}{c}{$25 C_1 + 21 C'_2 + 41$} \\
    \hline
\end{tabular}
}
\vspace{4mm}
\caption{Invariants for $\Gamma$.}
\label{table:ah3}
\end{table}

Next, we have 
\begin{align*}
\operatorname{Im}(\nu) \setminus \{ v_1 \} &\subset H \left( \frac{3}{4} v_1, v_2, v_3 \right), \\
\operatorname{Im}(\nu) \setminus \{ v_2 \} &\subset H \left( v_1, \frac{4}{5} v_2, v_3 \right) = H \left( \frac{4}{5} v_2, v_3, v_4 \right) \subset H \left( \frac{8}{9} v_2, v_4, v_5 \right), \\
\operatorname{Im}(\nu) \setminus \{ v_3 \} &\subset H \left( v_1, v_2, \frac{4}{5} v_3 \right) = H \left( v_2, \frac{4}{5} v_3, v_4 \right), \\
\operatorname{Im}(\nu) \setminus \{ v_4 \} &\subset H \left( v_2, v_3, \frac{6}{7}v_4 \right) = H \left( v_2, \frac{6}{7}v_4, v_5 \right) = H \left(\frac{6}{7}v_4, v_5 ,v' \right), \\
\operatorname{Im}(\nu) \setminus \{ v_5 \} &\subset H \left( v_2, v_4, \frac{3}{4} v_5 \right) = H \left(v_4, \frac{3}{4} v_5, v' \right).
\end{align*}
Therefore, 
the invariants $a^{\{ v \}}_{\sigma _i, v}(v)$'s and $h^{\{ v \}}_{\sigma _i, v}$'s for $i=1,2,3$ 
and $a^{\{ v \}}_{\Delta _1, v}(v)$'s and $h^{\{ v \}}_{\Delta _1, v}$'s for $v = v_4, v_5$ can be given as in Table \ref{table:ah3}. 
Furthermore, we have
\[
\operatorname{Im}(\nu) \setminus \{ v_4, v_5 \} \subset H(v_2, v_6, v_7). 
\]
Note that the six points
\[
v_2, \quad v_6, \quad v_7, \quad \frac{6}{7}v_4, \quad \frac{6}{7}v_5, \quad \frac{12}{13}v'
\]
are on a same hyperplane. 
Therefore, we have 
\begin{align*}
\operatorname{Im}(\nu) \setminus \{ v_4, v_5 \} 
\subset & H \left( \frac{6}{7}v_4, \frac{6}{7}v_5, \frac{12}{13}v' \right)
= H \left( v_7, \frac{6}{7}v_5, \frac{12}{13}v' \right) \\
&= H \left( \frac{6}{7}v_4, v_6, \frac{12}{13}v' \right)
= H \left( v_6, v_7, \frac{12}{13}v' \right). 
\end{align*}
Therefore, we can take 
\begin{align*}
&a^{\{ v_4, v_5 \}}_{\Delta _1, v'}(v') = a^{\{ v_4, v_5 \}}_{\Delta _2, v'}(v') = a^{\{ v_4, v_5 \}}_{\Delta _3, v'}(v') = a^{\{ v_4, v_5 \}}_{\Delta _4, v'}(v') = \frac{12}{13}, \\
&h^{\{ v_4, v_5 \}}_{\Delta _1, v'} = h^{\{ v_4, v_5 \}}_{\Delta _2, v'} = h^{\{ v_4, v_5 \}}_{\Delta _3, v'} = h^{\{ v_4, v_5 \}}_{\Delta _4, v'} = \frac{6}{7}. 
\end{align*}
By symmetry, we can also take
\[
a^{\{ v_5, v_7 \}}_{\Delta _1, v'}(v')
=a^{\{ v_7, v_6 \}}_{\Delta _1, v'}(v')
=a^{\{ v_6, v_4 \}}_{\Delta _1, v'}(v') = \frac{12}{13}, \quad
h^{\{ v_5, v_7 \}}_{\Delta _1, v'}
=h^{\{ v_7, v_6 \}}_{\Delta _1, v'}
=h^{\{ v_6, v_4 \}}_{\Delta _1, v'} = \frac{6}{7}.
\]

Hence, the invariants $\alpha ' _{\Delta, v}$'s and $\beta_{\Delta, \Delta}$'s are computed as in Table \ref{table:ah3}. 
By symmetry, we have 
\begin{align*}
\beta 
= C'_2 + \beta _{\Delta_1, \Delta_1}
= 25 C_1 + 22 C'_2 + 41
= 210. 
\end{align*}

Therefore, by Corollary \ref{cor:main} (and Theorem \ref{thm:main}(3)), 
it follows that the growth series $G_{\Gamma, x_0}$ is of the form 
\[
G_{\Gamma, x_0}(t) = \frac{Q(t)}{(1-t^6)(1-t^{12})^2}, 
\]
where $Q(t)$ is a polynomial of degree $\deg Q \le \beta + 30$. 

With the help of a computer program (breadth-first search algorithm), 
the first $\beta + 30 + 1$ ($= 241$) terms of the growth sequence $(s_{\Gamma, x_0, i})_{i \ge 0}$ can be computed. 
Using them, we can calculate $G_{\Gamma, x_0}(t)$ as follows:  
\small
\begin{align*}
G_{\Gamma, x_0}(t) 
&=  \frac{\bigl( \text{The terms of $(1-t^6)(1-t^{12})^2 \sum _{i = 0}^{240}  s_{\Gamma, x_0, i} t^i $ of degree $240$ or less.} \bigr)}{(1-t^6)(1-t^{12})^2} \\
&= \frac{1 + 4t + 12t^2 + 25t^3 +38t^4 + 52t^5 + 54t^6 + 44t^7 + 27t^8 + 8t^9 - t^{11} + 2 t^{12} + 2 t^{13}}{(1- t^3)(1-t^4)^2}. 
\end{align*}
\normalsize
By the form of $G_{\Gamma, x_0}(t)$, 
we can conclude that the growth sequence $(s_{\Gamma, x_0, i})_{i \ge 0}$ is a quasi-polynomial on $i \ge 3$, 
and $12$ is its period. 

\begin{rmk}
By Theorem \ref{thm:C2}, we can obtain $C_2 = 1.2414...\ $. 
\end{rmk}

\appendix
\section{Implementation of the algorithm}\label{section:a}
We prepare an implementation of the algorithm in {\sf Python} to compute the growth series of unweighted $2$-dimensional periodic graphs according to Remark \ref{rmk:implement}. 
For details, see the source code:
\begin{center}
\url{https://github.com/yokozuna57/Ehrhart_on_PG}
\end{center}
Here, we will only explain the input format. 
Let $(\Gamma, L)$ be an $n$-dimensional periodic graph. 
Let $c :=\# V_\Gamma / L$. 
When we choose representatives of $V_\Gamma / L$ and label them as $0$, $1$, \dots, $c-1$, 
we can identify $V_{\Gamma} = \{0,1,\ldots, c-1\} \times \mathbb{Z}^{\tt dim}$, where we set ${\tt dim} := n$. 

Thanks to the definition of the periodic graph, we can recover all the combinatorial data $E_{\Gamma}$, $s_{\Gamma}$ and $t_{\Gamma}$ from only the neighborhoods of vertices $(0,\mathbf{0})$, $(1,\mathbf{0})$, \dots, $(c-1,\mathbf{0})$.
In the case of the Wakatsuki graph (introduced in Example \ref{ex:inv} and \cite{IN}*{Example 2.6}), the input is as follows:
\begin{lstlisting}
dim=2
c=3
edges=[
    [(1,(0,0)),(1,(-1,0)),(1,(-1,-1)),(2,(0,0))],
    [(0,(0,0)),(0,(1,0)),(0,(1,1)),(2,(0,0))],
    [(0,(0,0)),(1,(0,0))]
]
pos=[(0,0),(0.5,0.5),(0.5,0)]
\end{lstlisting}
Here, we choose $v_0'$, $v_1'$ and $v_2'$ in Figure \ref{fig:WG1} as representatives.
The variable \verb~edge~ represents the edges of the quotient graph $\Gamma / L$.
More precisely, \verb~edges[i]~ is the list of vertices that have an edge from $v'_i = (i,(0,0))$ to them. 
For example, 
\[
\text{\verb~edges[0] = [(1,(0,0)),(1,(-1,0)),(1,(-1,-1)),(2,(0,0))]~}
\]
represents the list of edges $e$ with $s(e) = v'_0$, namely, $e'_0$, $e'_1$, $e'_2$ and $e'_3$ in Figure \ref{fig:WG2}. 
The right hand side is the list of the target vertices of $e'_0$, $e'_1$, $e'_2$ and $e'_3$.

In addition, the variable \verb~pos~ is used to give information on the periodic realization $\Phi: V_{\Gamma} \to L_{\mathbb{R}} := L \otimes _{\mathbb{Z}} \mathbb{R}$. 
More precisely, the variable \verb~pos~ presents the coordinations of $\Phi\left((0,\mathbf{0})\right)$, $\Phi\left((1,\mathbf{0})\right)$, \dots, $\Phi\left((c-1,\mathbf{0})\right) \in L_{\mathbb{R}}$. 
In this example, 
\[
\text{\verb~edges[0] = [(0,0),(0.5,0.5),(0.5,0)]~}
\]
represents
\[
\Phi(v'_0) = (0,0), \quad 
\Phi(v'_1) = (1/2,1/2), \quad
\Phi(v'_2) = (1/2,0)
\]
as in Figure \ref{fig:WG1}. 
Note that the choice of \verb~pos~ does not affect the graph structure of $\Gamma$, and hence, it does not affect the growth series of the periodic graph, but it is used in the program when calculating the values of $C_1$ and $C'_2$.
For the program to work correctly, each element of \verb~pos~ should be chosen in $[0,1)$.
If no specific realization is in mind, you can always choose \verb~pos=[(0,0,...,0)]*dim~.

\end{document}